\documentclass[preprint,11pt]{article}
\usepackage{fullpage}
\usepackage{amssymb,amsfonts,amsmath,amsthm,amscd,dsfont,mathrsfs}
\usepackage{graphicx,float,psfrag,epsfig,amssymb}
\usepackage[usenames,dvipsnames,svgnames,table]{xcolor}
\definecolor{darkgreen}{rgb}{0.0,0,0.9}
\usepackage{hyperref}
\usepackage{wrapfig}
\usepackage{relsize}
\usepackage{color}
\usepackage{pict2e}
\usepackage[tight]{subfigure}
\usepackage{algorithm}
\usepackage[noend]{algorithmic}
\usepackage{caption}
\usepackage{nameref}
\usepackage{makecell}
\usepackage[font={small}]{caption} 
\usepackage{float}
\usepackage{multirow}

\DeclareMathAlphabet{\mathpzc}{OT1}{pzc}{m}{it}

\newtheorem{propo}{Proposition}[section]
\newtheorem{lemma}[propo]{Lemma}
\newtheorem{definition}[propo]{Definition}

\newtheorem{thm}[propo]{Theorem}

\theoremstyle{definition}

\newtheorem{remark}[propo]{Remark}
\newtheorem{example}[propo]{Example}

\def\btR{{\boldsymbol {\tilde{R}}}}
\def\tV{\widetilde{V}}
\def\tR{\widetilde{R}}
\def\bone{{\boldsymbol 1}}
\def\Unif{{\sf Unif}}

\def\cF{{\cal F}}
\def\cA{{\cal A}}
\def\cR{{\cal R}}
\def\cH{{\cal H}}

\def\cE{{\cal E}}

\def\tM{\widetilde{M}}

\def\event{\mathcal{E}}

\def\bopt{\beta^{\rm opt}}

\def\AssI{{\sf G1}\,}
\def\AssII{{\sf G2}\,}
\def\AssIII{{\sf G3}\,}
\def\AssIV{{\sf G4}\,}

\def\naturals{{\mathbb N}}
\def\integers{{\mathbb Z}}
\def\reals{{\mathbb R}}

\def\eps{{\varepsilon}}
\def\prob{{\mathbb P}}
\def\E{{\mathbb E}}
\def\Var{{\rm Var}}

\def\bgo{\bgamma^{\rm opt}}

\def\L0{{L_0}}

\def\de{{\rm d}}
\def\<{\langle}
\def\>{\rangle}

\def\bZ{{\boldsymbol Z}}
\def\bX{{\boldsymbol X}}
\def\bY{{\boldsymbol Y}}
\def\bR{{\boldsymbol R}}

\def\bx{{\boldsymbol x}}
\def\by{{\boldsymbol y}}
\def\bz{{\boldsymbol z}}
\def\bp{{\boldsymbol p}}
\def\br{{\boldsymbol r}}

\def\tbp{\widetilde{\boldsymbol p}}
\def\tbp{\widetilde{\boldsymbol p}}
\def\btheta{{\boldsymbol \theta}}
\def\bTheta{{\boldsymbol \Theta}}
\def\balpha{{\boldsymbol \alpha}}

\def\bgamma{{\boldsymbol \gamma}}
\def\bxi{{\boldsymbol \xi}}

\def\tbR{\widetilde{\boldsymbol R}}
\def\tR{\widetilde{R}}

\def\talpha{\widetilde{\alpha}}

\def\F{{\sf F}}
\def\ind{{\mathbb I}}

\def\F{{\sf F}}
\def\normal{{\sf N}}
\def\Exp{{\sf Exp}}

\def\sT{{\sf T}}

\def\id{{\rm I}}

\def\LORD{\textsc{Lord\,}}

\def\bopt{\beta^{\rm opt}}

\def\event{\mathcal{E}}

\def\v*{v_0}
\def\T*{T_0}

\def\u*{u_0}
\def\F*{F_0}

\definecolor{olivegreen}{rgb}{0,0.6,0.4}

\def\tW{\widetilde{W}}
\def\FWER{{\rm FWER}}
\def\mFDR{{\rm mFDR}}
\def\sFDR{{\rm sFDR}}
\def\FDR{{\rm FDR}}
\def\FDX{{\rm FDX}}
\def\FDP{{\rm FDP}}
\def\cH{{\mathcal{H}}}
\def\cI{{\mathcal{I}}}

\def\BH{{\rm BH}}

\def\dg{\dot{g}}
\def\cR{{\mathcal R}}
\def\cB{{\mathcal B}}
\def\hnu{\widehat{\nu}}

\def\RL{R^{\rm L}}

\def\cI{\mathcal{I}}
\def\cImin{\cI_{\min}}
\def\cImax{\cI_{\max}}

\newcommand{\ajcomment}[1]{}

\makeatletter
\newcommand{\labitem}[2]{%
\def\@itemlabel{\text{#1}}
\item
\def\@currentlabel{#1}\label{#2}}
\makeatother

\title{Online Rules for Control of False Discovery Rate and False Discovery Exceedance\footnote{The authors would like to thank the Co-Editor, Associate Editor, and referees
for their valuable comments that helped us improve the paper significantly. A. Javanmard was partially supported by a CSoI fellowship during the course of this work
(NSF Grant CCF-0939370). A. Montanari was
supported in part by NSF grants CCF-1319979 and DMS-1106627, and the AFOSR grant FA9550-13-1-0036}}

\author{Adel Javanmard\footnote{Department of Data Sciences and Operations, University of Southern California. Email: \url{ajavanma@usc.edu} }
             \;\; and \;\; Andrea~Montanari\footnote{Department of Electrical Engineering and Department of Statistics, Stanford University. Email: \url{montanar@stanford.edu}}
            }

\begin{document}

\maketitle

\begin{abstract}
Multiple hypothesis testing is a core problem in statistical inference and arises in almost every scientific field.  
Given a set of null hypotheses $\cH(n) = (H_1,\dotsc, H_n)$, 
Benjamini and Hochberg~\cite{benjamini1995controlling} introduced the false discovery rate (\FDR), which is the expected proportion
of false positives among rejected null hypotheses, and
proposed a testing procedure that controls $\FDR$
below a pre-assigned significance level. Nowadays $\FDR$ is the criterion of choice for large-scale multiple
hypothesis testing.

In this paper we consider the problem of controlling $\FDR$ in an \emph{online manner}.
 Concretely, we consider an ordered --possibly infinite-- sequence of null hypotheses $\cH =
 (H_1,H_2,H_3,\dots )$ where, at  each step $i$, the
 statistician must  decide whether to reject  hypothesis 
$H_i$ having access only to the previous decisions. 
This model was introduced by Foster and Stine~\cite{alpha-investing}.

We study a class of \emph{generalized alpha investing} procedures, first introduced by Aharoni and Rosset \cite{generalized-alpha}. We prove
that any rule in this class controls online $\FDR$, provided $p$-values corresponding to true nulls are independent from the other $p$-values. 
Earlier work only established $\mFDR$ control.
Next, we obtain conditions under which generalized alpha investing controls $\FDR$ in
the presence of general $p$-values dependencies. We also develop a modified set of procedures
that allow to control the false discovery exceedance (the tail of the proportion of false discoveries).
Finally, we evaluate the performance of online procedures on both synthetic and real data, comparing
them with offline approaches, such as adaptive Benjamini-Hochberg.
\end{abstract}

%
%

%
%
\section{Introduction}

The common practice in claiming a scientific discovery is to support such claim with a $p$-value
as a measure of statistical significance. Hypotheses with $p$-values below a significance level $\alpha$, typically 
$0.05$, are considered to be \emph{statistically significant}. While this ritual controls type I errors for single testing problems, in case of testing multiple
hypotheses it  leads to a large number of false positives (false discoveries).
Consider, for instance, a setting in which $N$ hypotheses are to be tested, but only a few of them, say $s$, are non-null.
If we test all of the hypotheses at a fixed significance level $\alpha$, each of $N-s$ truly null hypotheses can be falsely rejected with probability $\alpha$.  
Therefore, the number of false discoveries  --equal to $\alpha(N-s)$ in expectation-- can substantially exceed the number $s$ of true non-nulls.

The false discovery rate (FDR) --namely, the expected fraction of discoveries that are false positives-- is the 
criterion of choice for statistical inference in large scale hypothesis testing problem. In their groundbreaking work \cite{benjamini1995controlling}, 
Benjamini and Hochberg (BH) developed a procedure to control FDR below a pre-assigned level, while allowing for a large number of true discoveries when 
many non-nulls are present.
The  BH procedure  remains --with some improvements-- the state-of-the-art in the context of multiple hypothesis testing,
and has been implemented across genomics \cite{reiner2003identifying}, brain imaging \cite{genovese2002thresholding}, marketing \cite{optimizely}, and many other applied domains.

Standard FDR control techniques, such as the BH procedure~\cite{benjamini1995controlling},
require aggregating $p$-values for all the tests and processing them jointly. This is impossible in a number of applications which are
best modeled as an online hypothesis testing problem \cite{alpha-investing} (a more formal definition
will be provided below):

\smallskip

\emph{Hypotheses arrive sequentially in a stream. At each step, the analyst must decide whether to reject the current
null hypothesis without having access to the number of hypotheses (potentially infinite) or the future $p$-values, but solely based on the previous decisions.}
\smallskip

This is the case --for instance-- with publicly available datasets, where new hypotheses are tested in an on-going fashion
by different researchers \cite{generalized-alpha}. Similar constraints arise in marketing research, where multiple A-B tests are carried
out on an ongoing fashion \cite{optimizely}. Finally, scientific research as a whole suffers from the same problem: a stream of hypotheses are tested on an ongoing
basis using a fixed significance level, thus leading to large numbers of false positives \cite{ioannidis2005most}.
We refer to Section \ref{sec:related work} for further discussion.

In order to illustrate the online scenario, consider an approach that would  control the family-wise error rate (FWER), i.e. the probability of rejecting at least one true
null hypothesis. Formally
\begin{equation}
\FWER(n) \equiv \sup_{\btheta \in \bTheta}\,
\prob_\theta\Big(V^\theta(n) \ge 1 \Big)\,.
\end{equation}
where $\btheta$ denotes the model parameters (including the set of non-null hypotheses)
and $V^{\theta}(n)$ the number of false positives among the first $n$ hypotheses.
This metric can be controlled by choosing different significance levels $\alpha_i$ for tests 
$H_i$, with $\balpha=(\alpha_i)_{i\ge 1}$ summable, e.g., $\alpha_i = \alpha 2^{-i}$. 
Notice that the analyst only needs to know the number of tests performed before the current one, in order to implement this scheme.
However, this method leads to small statistical power.  In particular, making a discovery at later steps becomes very unlikely.

In contrast, the BH procedure assumes that all the $p$-values are given a priori. 
Given $p$-values $p_1, p_2, \dotsc, p_N$ and a significance level $\alpha$, BH follows 
the steps below:
\begin{enumerate}
\item Let $p_{(i)}$ be the $i$-th $p$-value in the (increasing) sorted order, and define $p_{(0)} = 0$.  Further. let
\begin{align}\label{eq:iBH}
i_{\BH} \equiv \max\Big\{0\le i\le N:\, p_{(i)}\le \alpha i/N \Big\}\,.
\end{align}
\item Reject $H_j$ for every test with $p_j \le p_{(i_{\BH})}$.
\end{enumerate}
As mentioned above, BH controls the false discovery rate defined as
\begin{eqnarray}
\FDR(N) \equiv \sup_{\btheta\in \bTheta} \E_\theta\left(\frac{V^{\theta}(N)}{R(N)\vee 1}\right)\,,
\end{eqnarray}
where $R(N)$ is the total the number of rejected hypotheses.
Note that BH requires the knowledge of \emph{all} $p$-values to
determine the significance level for testing the hypotheses. Hence, it
does not address the online scenario.

In this paper, we study  methods for \emph{online} control of false discovery rate. 
Namely, we consider a sequence of hypotheses $H_1, H_2, H_3, \dotsc$
that arrive sequentially 
in a stream, with corresponding $p$-values $p_1$, $p_2$, $\dots$. 
We aim at developing a testing mechanism that ensures false discovery
rate remains below 
a pre-assigned level $\alpha$. 
A testing procedure provides a sequence of significance levels $\alpha_i$, with decision rule:
\begin{eqnarray}\label{eq:Ti}
R_i = \begin{cases}
1& \text{if }p_i\le\alpha_i\quad\quad\;\text{(reject $H_i$),}\\
0& \text{otherwise}\quad \quad \text{(accept $H_i$).}
\end{cases}
\end{eqnarray}
In \emph{online} testing, we require significance levels to be functions of prior outcomes:
\begin{eqnarray}
\alpha_i = \alpha_i(R_1, R_2, \dotsc, R_{i-1})\,.
\end{eqnarray} 

Foster and Stine~\cite{alpha-investing} introduced the above setting and proposed a class of 
procedures named \emph{alpha investing rules.} Alpha investing starts with an initial wealth,
at most $\alpha$, of allowable false discovery  rate.  The wealth is spent for
testing different hypotheses.  Each time a discovery occurs, the
alpha investing procedure earns  a contribution toward its wealth to
use for further tests. Foster and Stine~\cite{alpha-investing}  proved that alpha investing rules control a 
modified metric known as mFDR, defined as below:
\begin{align}
\mFDR_{\eta}(n)\equiv \sup_{\btheta\in \bTheta}  \frac{\E(V^{\theta}(n))}{\E(R(n))+\eta}\,.
\end{align}
In words, $\mFDR$ is the ratio of the expected number of false discoveries 
to the expected number of discoveries. As illustrated in the Appendix~\ref{sec:mfdr_fdr}, mFDR and FDR can be very different in
situations with high variability. While FDR is the expected proportion of false discoveries, mFDR is the 
ratio of  two expectations and hence is not directly related to any single sequence quantity.

Several recent papers \cite{lockhart2014significance,g2015sequential,li2016accumulation} consider a `sequential hypothesis testing' problem
that arises in connection with sparse linear regression. Let us emphasize that the problem treated in  \cite{lockhart2014significance,g2015sequential} 
is substantially different from the one analyzed here. For  instance, as discussed in Section \ref{sec:related work},
the methods of \cite{g2015sequential} achieve vanishingly small statistical power for the present problem.

\subsection{Contributions}

In this paper, we study a  class of procedures that are known as
\emph{generalized alpha investing}, and were first introduced by
Aharoni and Rosset in \cite{generalized-alpha}. As in alpha investing~\cite{alpha-investing}, generalized alpha investing
makes use of a potential sequence (wealth) that increases every time a null hypothesis is rejected, and decreases otherwise. However: 
$(i)$~The pay-off and pay-out functions are general functions of
past history; $(ii)$ The pay-out is not tightly determined by the
testing level $\alpha_i$. 
This additional freedom allows to construct interesting new rules.

The contributions of this paper are summarized as follows.
%

{\bf Online control of FDR.} We prove that generalized
  alpha investing rules control  FDR,
under the assumption of independent $p$-values, and provided they are \emph{monotone} (a technical condition defined in the sequel). To the best of our
knowledge, this is the first work\footnote{Special cases were
  presented in our earlier technical report \cite{javanmard2015online}.} that guarantees online control of
FDR. 

{\bf Online control of FDR for dependent $p$-values.}
Dependencies among  $p$-values can arise for multiple reasons. For
instance the same data can be re-used to test a new hypothesis, or the
choice of a new hypothesis can depend on the past outcomes. We present
a general upper bound on the FDR for dependent $p$-values under
generalized alpha investing.

{\bf False discovery exceedance.} FDR can be viewed as the
  expectation of false discovery
  proportion (FDP). In some cases, FDP may not be well represented
  by its expectation, e.g., when the number of discoveries is small. In
  these cases, FDP might be sizably larger than its expectation
  with significant probability. In
  order to provide tighter control, we develop bounds on the false discovery
  exceedance (FDX), i.e. on the tail probability of FDP.

{\bf Statistical power.} In order to compare different  procedures, we develop lower bounds on fraction of non-null
  hypotheses that are discovered (statistical power),  under a mixture model where each null hypothesis is  false with
 probability $\pi_1$, for a fixed arbitrary $\pi_1$.

We focus in particular on a concrete example of generalized alpha investing
rule  (called \LORD below) that   we consider particularly compelling. We use our lower bound
to guide the choice of parameters for this rule.

 {\bf Numerical Validation.} We validate our procedures on synthetic
   and real data in Section~\ref{sec:simulation-syn} and Appendix~\ref{sec:AdClick}, showing that they control FDR
   and mFDR in an online setting. We further compare them with BH and
   Bonferroni procedures. 
   We observe that generalized alpha investing procedures can benefit from ordering of hypotheses. Specifically, they can achieve higher statistical power compared to offline benchmarks such as adaptive BH, when fraction of non-nulls is small and hypotheses can be a priori ordered in such a way that those most \emph{likely} to be rejected appear first in the sequence.

%

\subsection{Further related work}
\label{sec:related work}

\noindent\emph{General context.} An increasing effort was devoted to reducing the risk of fallacious research findings.
Some of the prevalent issues such as publication bias, lack of
replicability and multiple comparisons on a dataset were
discussed in Ioannidis's 2005
papers~\cite{ioannidis2005most,ioannidis2005contradicted} 
and in~\cite{prinz2011believe}. 

\smallskip

\noindent\emph{Statistical databases.} Concerned with the above issues and the
importance of data sharing in the 
genetics community, ~\cite{rosset2014novel} proposed an approach to
public database 
management, called Quality Preserving Database (QPD). 
A QPD  makes a shared data resource amenable to perpetual use for hypothesis 
testing while controlling FWER and maintaining statistical power of
the tests.  In this scheme, for testing 
a new hypothesis, the investigator should pay a price in form of
additional samples that  should be added to the database. The number
of required samples for each test 
depends on the required effect size and the power for the corresponding test. A key feature of QPD is that 
type I errors are controlled at the management layer and the investigator is not concerned with $p$-values for the tests. Instead, investigators provide effect size, assumptions on the distribution of the data, and the desired statistical power. A critical limitation of QPD is that all samples, including those currently in the database and those that will be added, are assumed to have the same quality and are coming from a common underlying distribution.  Motivated by similar concerns in practical data analysis, ~\cite{dwork2014preserving}  applies insights from differential privacy to efficiently use samples to answer adaptively chosen estimation queries.   
These papers however do not address the problem of controlling FDR in online multiple testing.

\smallskip

\noindent\emph{Online feature selection.} Building upon
alpha investing procedures,~\cite{lin2011vif} develops 
VIF, a method for feature selection in large regression problems. VIF is accurate and computationally very efficient; it uses a one-pass search over the pool of features and applies alpha investing to test each feature for adding to the model. VIF regression avoids overfitting due to the property that alpha investing controls $\mFDR$. Similarly, one can incorporate $\LORD$ in VIF regression to perform fast online feature selection and provably avoid overfitting.

\smallskip

\noindent\emph{High-dimensional and sparse regression.} There has been 
significant interest over the last two years in developing hypothesis
testing procedures for high-dimensional regression, especially in
conjunction with sparsity-seeking methods. Procedures for computing
$p$-values of low-dimensional coordinates were developed in 
\cite{zhang2014confidence,van2014asymptotically,javanmard2014confidence,javanmard2013hypothesis,javanmard2013nearly}.
Sequential and selective inference  methods were proposed in 
\cite{lockhart2014significance,fithian2014optimal,taylor2014exact}.
Methods to control FDR were put forward in \cite{barber2015controlling,bogdan2015slope}.

As exemplified by VIF regression, online hypothesis testing methods can be useful
in this context as they allow to select a subset of regressors through
a one-pass procedure. Also they can be used in conjunction with the
methods of \cite{lockhart2014significance}, where a sequence of
hypothesis is generated by including an increasing number of
regressors (e.g. sweeping values of the regularization parameter).

In particular, \cite{g2015sequential,li2016accumulation} develop multiple hypothesis testing procedures for ordered tests.
Note, however, that these approaches fall short of addressing the issues 
we consider, for several reasons:
$(i)$  They are not online, since they reject the first $\hat{k}$ null
hypotheses, where $\hat{k}$ depends on all the $p$-values.
 $(ii)$ They require knowledge of all
past $p$-values (not only discovery events) to compute the current
score. $(iii)$ Since they are constrained to reject all hypotheses before
$\hat{k}$, and accept
them after, they cannot achieve any discovery rate
increasing with $n$, let alone nearly linear in $n$. 
For instance in the mixture model of Section \ref{sec:power},
if the fraction of true non-null is $\pi_1<\alpha$, then the methods of \cite{g2015sequential,li2016accumulation} 
achieves $O(1)$ discoveries out of $\Theta(n)$ true non-null. In other
words their power is of order $1/n$ in this simple case. 

\subsection{Notations}

Throughout the paper, we typically use upper case symbols (e.g. $X,Y,Z,\dots$) to denote random variables, and lower case 
symbols for deterministic values (e.g. $x,y,z,\dots$). Vectors are denoted by boldface, e.g. $\bX,\bY, \bZ,\dots$ for
random vectors, and $\bx,\by, \bz,\dots$ for deterministic vectors. Given a vector $\bX=(X_1,X_2,\dots,X_n)$, 
we use $\bX_{i}^j = (X_i,X_{i+1},\dots,X_j)$ to denote the sub-vector with indices between $i$ and $j$.
We will often consider sequences indexed by the same `time index' as for the hypotheses $\{H_1,H_2,H_3,\dots\}$.
Given such a sequence $(X_i)_{i\in\naturals}$, we denote by $X(n) \equiv\sum_{i=1}^nX_i$ its partial sums.

We denote the  standard Gaussian density by $\phi(x) = e^{-x^2/2}/\sqrt{2\pi}$, and the Gaussian distribution function
by $\Phi(x) = \int_{-\infty}^x \phi(t)\, \de t$.
We use the standard big-O notation. In particular $f(n) = O(g(n))$ as $n\to\infty$ if there exists a constant $C>0$ 
such that $|f(n)|\le C\, g(n)$ for all $n$ large enough.
We also use $\sim$ to denote asymptotic equality, i.e. $f(n)\sim g(n)$ as $n\to \infty$, means $\lim_{n\to\infty}f(n)/g(n)=1$.
We further use $\asymp$ for equality up to constants, i.e. if $f(n) = \Theta(g(n))$, then there exist constants $C_1,C_2>0$ such that $C_1|g(n)|\le f(n) \le 
C_2 |g(n)|$ for all $n$ large enough.

%
%
 \section{Generalized alpha investing}

In this section we define generalized alpha investing rules, and provide some concrete examples. 
Our definitions and notations follow the paper of Aharoni and Rosset that first introduced generalized alpha investing \cite{generalized-alpha}.

\subsection{Definitions}
\label{sec:DefGen}

Given a sequence of input $p$-values
$(p_1,p_2,\dots)$, a \emph{generalized alpha investing} rule generates a sequence of decisions $(R_1,R_2,\dots)$ (here $R_{j}\in\{0,1\}$
and $R_j=1$ is to be interpreted as rejection of null hypothesis $H_{j}$)
by using test levels $(\alpha_1,\alpha_2,\alpha_3,\dots)$. After each decision $j$, the rule updates a potential function $W(j)$ as follows:
\begin{itemize}
\item If hypothesis $j$ is accepted, then the potential function is decreased by a pay-out $\varphi_j$.
\item If hypothesis $j$ is rejected, then the potential is increased by an amount $\psi_j-\varphi_j$. 
\end{itemize}
In other words, the pay-out $\varphi_j$ is the amount paid for testing a new hypothesis, and the pay-off $\psi_j$ 
is the amount earned if a discovery is made at that step.

Formally, a generalized alpha investing rule is specified by three (sequences of) functions $\alpha_j,\varphi_j,\psi_j:\{0,1\}^{j-1}\to\reals_{\ge 0}$,
determining test levels, pay-out and pay-off. Decisions are taken by testing at level $\alpha_j$ 
\begin{eqnarray}\label{eq:TiRep}
R_j = \begin{cases}
1,& \text{if }p_j\le\alpha_j=\alpha_j(R_1,\dots,R_{j-1}),\\
0,& \text{otherwise.}
\end{cases}
\end{eqnarray}
The potential function is updated via:
\begin{align}
W(0) &= w_0 \, ,\\
W(j) & = 
W(j-1)-\varphi_j(\bR_1^{j-1})+R_j\, \psi_j(\bR_1^{j-1})\,,
\label{eq:Wupdate}
\end{align}
with $w_0\ge 0$ an initial condition.
Notice in particular that $W(j)$ is a function of $(R_1,\dots,R_j)$.

A valid generalized alpha investing rule is required to satisfy the following conditions, for a constant $b_0>0$:
\begin{enumerate}
\item[{\sf G1.}] For all $j\in\naturals$ and all $\bR_{1}^{j-1}\in\{0,1\}^{j-1}$, 
letting $\psi_j = \psi_j(\bR_1^{j-1})$, $\varphi_j = \varphi_j(\bR_1^{j-1})$, $\alpha_j = \alpha_j(\bR_1^{j-1})$,
we have
\begin{align}
\psi_j&\le \varphi_j+b_0\, ,\label{eq:A1a}\\
\psi_j&\le \frac{\varphi_j}{\alpha_j}+b_0-1\,, \label{eq:A1b}  \\
\varphi_j &\le W(j-1)\,. \label{eq:Nonneg}
\end{align}
\item[{\sf G2.}]  For all $j\in\naturals$, and all $\bR_1^{j-1}\in\{0,1\}^{j-1}$, if  $W(j-1)=0$ 
then $\alpha_j=0$. 
\end{enumerate}
Notice that Condition~\eqref{eq:Nonneg} and ${\sf G2}$ are well posed since $W(j-1)$, $\varphi_j$ and $\alpha_j$
are functions of $\bR_1^{j-1}$. Further, because of~\eqref{eq:Nonneg},
the  function $W(j)$ remains non-negative for all $j\in \naturals$.  

We later show that generalized alpha investing guarantees $\FDR$ control as a function of $b_0$ and $w_0$.

Throughout, we  shall denote by $\cF_j$ the $\sigma$-algebra generated by 
the random variables $\{R_1,\dots,R_{j}\}$. 

\begin{definition}\label{def:mon}
For $\bx$, $\by\in\{0,1\}^n$, we write $\bx\preceq\by$ if $x_j\le y_j$ for all $j\in\{1,\dots,n\}$.
We say that an online rule is \emph{monotone} if  the functions $\alpha_j$ are monotone non-decreasing
with respect to this partial ordering (i.e. if $\bx\preceq \by$ implies $\alpha_j(\bx)\le \alpha_j(\by)$).
\end{definition}

\begin{remark}
Our notation differs from \cite{generalized-alpha} in one point, namely we use  $w_0$ for the initial potential (which is denoted by $\alpha\eta$ 
in  \cite{generalized-alpha}) and $b_0$ for the constant appearing in Equations.~(\ref{eq:A1a}), (\ref{eq:A1b})  (which is denoted by $\alpha$ 
in  \cite{generalized-alpha}). We prefer to reserve $\alpha$ for the FDR level\footnote{The use of $\eta$ in \cite{generalized-alpha} was related
to control of $\mFDR_{\eta}$ in that paper.}.
\end{remark}

\begin{remark}
In a generalized alpha investing rule, as we reject more hypotheses the potential $W(j)$ increases and hence we can use large
test levels $\alpha_j$. In other words, the burden of proof decreases as we reject more hypotheses. This is similar to the BH rule, where the most significant $p$-values is compared to a Bonferroni cutoff, the second most significant to twice this cutoff and so on. 
\end{remark}

\subsection{Examples}

Generalized $\alpha$-investing rules comprise a large variety of online hypothesis testing methods. We next describe some specific subclasses that 
are useful for designing specific procedures.

\subsubsection{Alpha Investing}\label{subsec:AI}
Alpha investing, introduced by Foster and Stine \cite{alpha-investing}, is a special case
of generalized alpha investing rule. In this case the potential  is decreased by $\alpha_j/(1-\alpha_j)$ if hypothesis $H_j$ is
not rejected, and increased by a fixed amount $b_0$ if it is rejected. In formula, the potential evolves according to
\begin{align}
W(j) = W(j-1)-(1-R_j)\frac{\alpha_j}{1-\alpha_j} +R_jb_0\, .
\end{align}
This fits the above framework by defining $\varphi_j = \alpha_j/(1-\alpha_j)$ and $\psi_j = b_0+\alpha_j/(1-\alpha_j)$.
Note that this rule depends on the choice of the test levels $\alpha_j$, and of the parameter $b_0$. 
The test levels $\alpha_j$ can be chosen arbitrarily, provided that they satisfy condition (\ref{eq:Nonneg}), which is equivalent to $\alpha_j/(1-\alpha_j)\le W(j-1)$.

\subsubsection{Alpha Spending with Rewards}\label{sec:reward}
Alpha spending with rewards was introduced in \cite{generalized-alpha}, as a special sub-class of
generalized alpha investing rules, which are convenient for some specific applications.

In this case, test levels are chosen to be proportional to the
pay-out function, $\alpha_j = \varphi_j/\kappa$,  with a proportionality coefficient $\kappa$. Conditions 
(\ref{eq:A1a}) and (\ref{eq:A1b}) coincide with\footnote{Note that \cite{generalized-alpha} rescales the potential function by $\kappa$, and hence the condition on $\psi_j$
is also rescaled.}
\begin{align}
0\le \psi_j\le \min\Big(\kappa\alpha_j+b_0,\kappa-1+b_0\Big)\,. 
\end{align}
The choice of penalties $\varphi_j$ is arbitrary as long as constraint~\eqref{eq:Nonneg} is satisfied. For instance, \cite{generalized-alpha}
uses $\varphi_j = c_1W(j-1)$ with $c_1\in (0,1)$. 

\subsubsection{{LORD}}

As a running example, we shall use a simple procedure that we term \LORD, for  Levels based On Recent Discovery.
$\LORD$ is easily seen to be a special case of alpha spending with rewards, for $\kappa=1$.

Below, we present three different versions of $\LORD$. For a concrete exposition, choose any sequence of non-negative numbers 
$\bgamma = (\gamma_i)_{i\in\naturals}$, which is monotone non-increasing
(i.e. for $i\le j$ we have $\gamma_i\ge \gamma_j$) and such that
$\sum_{i=1}^\infty \gamma_i = 1$. We refer to Section \ref{sec:power} for concrete choices of this sequence.

At each time $i$, we let $T_{i}$ be the set of discovery times up to time $i$. We further define
$\tau_i$ as the last time a discovery was made
before $i$:
\[
T(i)= \Big\{\ell \in\{1,\dots,i-1\}\;  :\; R_{\ell}=1\Big\}\,,\quad
\quad \tau_i \equiv \max \{\ell: \ell \in T(i)\}\,.\]
At each step, if a discovery is made, we add an amount
$b_0$ to the current wealth. Otherwise, we remove an amount of the current test level from the wealth. 
Formally, we set
\begin{align}\label{eq:LORD-general}
W(0)  = w_0\,,\quad  \psi_i = b_0\, ,\quad \varphi_i = \alpha_i\,, 
%
\end{align}
where  $\{W(j)\}_{j\ge 0}$ is defined
recursively via Equation~(\ref{eq:Wupdate}).
Note that $\tau_i$ and $T(i)$ are measurable on $\cF_{i-1}$, and hence
$\varphi_i,\psi_i$ are functions of $\bR_1^{i-1}$ as claimed, while
$W(i)$ is a function of $\bR_1^i$.
 Therefore, the above rule defines an online multiple hypothesis testing procedure.

We present three versions of $\LORD$ which differ in the way that the test levels $\alpha_i$ are set.

\begin{itemize}
\item $\LORD 1$: We set the test levels solely based on the time of the last discovery. Specifically, 
\begin{align}\label{alpha:LORD1}
\alpha_ i = \begin{cases}
\gamma_i w_0 & \text{ if } i \le t_1\,,\\
\gamma_{i-\tau_i} b_0 & \text{ if } i > t_1\,,
\end{cases}
\end{align}
where $t_1$ denotes the time of first discovery.
In words, up until the first discovery is made, we set levels by discounting the initial wealth, i.e.,  $\gamma_i w_0$. After the first discovery is made, we use a fraction $\gamma_{i-\tau_i}$ of $b_0$ to spend in testing null hypothesis $H_i$.

\item $\LORD 2$: We set the test levels based on the previous discovery times. Specifically, 
\begin{align}\label{alpha:LORD2}
\alpha_ i = \gamma_i w_0 + \Big(\sum_{\ell \in T(i)} \gamma_{i-\ell}\Big) b_0\,.
\end{align}

\item $\LORD 3$: In this alternative, the significance levels $\alpha_i$ depend on
the past only through the time of the last discovery, and the wealth accumulated at that
time. Specifically, 
\begin{align}
\alpha_i = \alpha_i = \gamma_{i-\tau_i}\, W(\tau_i)\, ,\label{eq:PhiLORD}
\end{align}
\end{itemize} 

In the next lemma, we show that all the three versions of $\LORD$ are generalized alpha investing rules. Further, $\LORD 1$ and $\LORD 2$ are monotone rules (see Definition~\ref{def:mon}), while $\LORD 3$ is not necessarily a monotone rule without making further assumptions on sequence $\bgamma$. 
\begin{lemma}\label{LORD-mon}
The rules $\LORD 1$, $\LORD 2$ and $\LORD 3$ are instances of generalized alpha investing rules. Further, the rules $\LORD1$ and $\LORD 2$ are monotone.
  \end{lemma}
Lemma~\ref{LORD-mon} is proved in Appendix~\ref{proof:LORD-mon}.

\section{Control of false discovery rate}

\subsection{$\FDR$ control for independent test statistics}

As already mentioned, we are interested in testing a --possibly
infinite-- sequence of null hypotheses $\cH = (H_i)_{i\in \naturals}$.
The set of first $n$ hypotheses will be denoted by $\cH(n) = (H_i)_{1\le i\le n}$.
Without loss of generality, we assume $H_i$ concerns the value of a
parameter $\theta_i$, with  $H_i= \{\theta_i = 0\}$. Rejecting the null
hypothesis $H_i$ can be interpreted as  $\theta_i$ being significantly
non-zero. 
We will denote by $\Theta$ the set of possible values for the parameters
$\theta_i$, and by $\bTheta = \Theta^{\naturals}$ the space of
possible values of the sequence $\btheta = (\theta_i)_{i\in\naturals}$

Under the null hypothesis $H_i : \,\theta_i =0$, the corresponding $p$-value is uniformly random in $[0,1]$:
\begin{align}
p_i\sim \Unif([0,1])\,.
\end{align}
Recall that $R_i$ is the indicator that a discovery is made at time
$i$, and $R(n) =\sum_{i=1}^nR_i$ the total number of discoveries up to
time $n$.  
Analogously,  let $V^\btheta_i$ be the  indicator that a false
discovery occurs at time  $i$ and $V^\theta(n)
=\sum_{i=1}^nV^{\theta}_i$ the total number of false discovery up to
time $n$. 
Throughout the paper, superscript $\theta$ is used to distinguish unobservable variables such as $V^\theta(n)$, from statistics such as  
$R(n)$. However, we drop the superscript when it is clear from the context.

There are various criteria of interest for multiple testing methods. We will mostly focus on the \emph{false discovery rate (FDR)}~\cite{benjamini1995controlling},
and we repeat its definition here for the reader's convenience. 
We first define the \emph{false discovery proportion} (\FDP) as follows. For $n\ge 1$,
\begin{eqnarray}
\FDP^\theta(n) \equiv \frac{V^{\theta}(n)}{R(n)\vee 1}\,. \label{eq:FDPdef}
\end{eqnarray}
The false discovery rate is defined as
\begin{eqnarray}
\FDR(n) \equiv \sup_{\btheta\in \bTheta} \E_\theta\Big(\FDP^{\theta}(n)\Big)\,.
\end{eqnarray}

Our first result establishes FDR control for all monotone generalized
alpha investing procedures, where the monotonicity of a testing rule is given by Definition~\ref{def:mon}. Its proof is presented in Appendix~\ref{sec:ProofFDRControl}.
\begin{thm}\label{thm:FDR-control}
Assume the $p$-values $(p_i)_{i\in \naturals}$ to be
independent. Then, for any monotone generalized alpha investing rule with
$w_0+b_0\le \alpha$, we have
\begin{align}
\sup_n\,\FDR(n) \le \alpha\, . \label{eq:Control-1}
\end{align}
The same holds if only the $p$-values corresponding to true nulls are mutually independent, and independent 
from the non-null $p$-values.
\end{thm}

\begin{remark}
By applying Theorem~\ref{thm:FDR-control} and Lemma~\ref{LORD-mon}, we obtain that $\LORD 1$ and $\LORD$ controls $\FDR$ at level $\alpha$, as long as
$w_0 + b_0 \le \alpha$. For $\LORD 3$, such result cannot be obtained directly from Theorem~\ref{thm:FDR-control} because this rule is not necessarily a monotone rule
without making further assumptions on the sequence $\bgamma$. Nevertheless, in our numerical experiments, we focus on $\LORD 3$ and as we show empirically that
it also control $\FDR$. \footnote{Henceforth, whenever we mention $\LORD$ rule, we are referring to $\LORD 3$.}
\end{remark}

\begin{remark}\label{remark:BetterFDR}
In Appendix~\ref{sec:ProofFDRControl}, we prove a somewhat stronger version of Theorem \ref{thm:FDR-control},
namely $\FDR(n) \le b_0 \E\{R(n)/(R(n)\vee 1)\} +w_0\E\{1/(R(n)\vee 1)\}$. In particular, $\FDR(n) \lesssim b_0$ when 
the total number of discoveries $R(n)$ is large, with high probability. This is the case --for instance-- when the hypotheses 
to be tested comprise
a large number of `strong signals' (even if these form a small proportion of the total number of hypotheses).
\end{remark}
Another possible strengthening of Theorem \ref{thm:FDR-control} is
obtained by considering a new metric, that we call  $\sFDR_{\eta}(n)$ (for smoothed FDR):\footnote{Some authors \cite{barber2016knockoff} refer to this quantity 
as ``modified FDR". We will not follow this terminology since its acronym (mFDR) gets confused with ``marginal FDR"
\cite{alpha-investing,generalized-alpha}.} 
\begin{align}\label{eq:sFDR-aj}
\sFDR_{\eta}(n) \equiv \sup_{\btheta\in\Theta}\E\Big\{\frac{ V^{\theta}(n)}{R(n)+\eta}\Big\}\, .
\end{align}
The following theorem  bounds $\sFDR_{w_0/b_0}(n)$ for monotone generalized alpha investing rules (cf. Definition~\ref{def:mon}).
\begin{thm}\label{thm:FDR-control-2}
Under the assumptions of Theorem \ref{thm:FDR-control}, for any $w_0,b_0>0$, we have
\begin{align}
\sup_{n}\, \sFDR_{w_0/b_0}(n)\le b_0\, .\label{eq:Control-2}
\end{align}
\end{thm}
Note that Equation~(\ref{eq:Control-2}) implies (\ref{eq:Control-1}) by using $R(n)+(w_0/b_0)\le (b_0+w_0)R(n)/b_0$ for $R(n)\ge 1$. 
Also, $\E\big\{V^{\theta}(n)/(R(n)+(w_0/b_0))\big\}\approx \FDR(n)$ if $R(n)$ is large with high probability. 

Let us  emphasize that the guarantee in Theorem \ref{thm:FDR-control-2} is  different from the one in 
\cite{alpha-investing,generalized-alpha}, which instead use $\mFDR_{\eta}(n)\equiv\E\{V^{\theta}(n)\}/(\E\{R(n)\}+\eta)$.
 As mentioned earlier, $\mFDR$ does not correspond to a single-sequence property.

\begin{remark}\label{remark:LowerBound}
In Appendix~\ref{sec:FDR-lowerbound} we show that Theorems \ref{thm:FDR-control} and \ref{thm:FDR-control-2} 
cannot be substantially improved, unless specific restrictions are imposed on the generalized alpha investing rule. 
In particular,
we prove that there exist generalized alpha investing rules for which $\lim\inf_{n\to\infty}\FDR(n)\ge b_0$, 
and $\lim_{n\to\infty}\sFDR_{w_0/b_0} =  b_0$.
\end{remark}

\subsection{$\FDR$ control for dependent test statistics}\label{sec:depFDP}

In some applications, the assumption of independent $p$-values is not
warranted. This is the case --for instance-- of multiple related  hypotheses being
tested on the same experimental data.
Benjamini and Yekutieli \cite{benjamini2001control} introduced a
property called \emph{positive regression dependency from a subset}
$I_0$ (PRDS on $I_0$) to capture a positive dependency structure
among the test statistics. They showed that if the joint distribution
of the test statistics is PRDS on the subset of test statistics corresponding to 
true null hypotheses, then BH controls $\FDR$. (See Theorem 1.3
in~\cite{benjamini2001control}.) 
Further, they proved that BH controls $\FDR$ under general dependency
if its threshold is adjusted by replacing 
$\alpha$ with $\alpha/(\sum_{i=1}^N \frac{1}{i})$ in equation~\eqref{eq:iBH}.

Our next result establishes an upper bound on the $\FDR$ of generalized alpha investing
rules, under general $p$-values dependencies.
For a given generalized alpha investing rule, 
let $\cR_i\equiv\{\br_{1}^i\in\{0,1\}^i:\, \prob(\bR_1^i=\br_1^i)>0\}$, the set of decision sequences that have non-zero probability.
\begin{definition}
An \emph{index sequence} is a sequence of deterministic functions $\cI = (\cI_i)_{i\in \naturals}$ with 
$\cI_i: \{0,1\}^{i} \to \reals_{\ge 0}$. For an index sequence $\cI$, let
\begin{align}
\RL_i(s)&\equiv \min_{\br_1^{i-1}\in\cR_{i-1}}\Big\{\sum_{j=1}^{i-1} r_j\, :\, \cI_{i-1}(\br_1^{i-1}) \ge s\Big\}\,, \label{eq:RL}\\
\cImin(i) &\equiv \min_{\br_1^{i}\in\cR_{i}}\, \cI_i(\br_1^{i}) \, ,
\;\;\cImax(i) \equiv \max_{\br_1^{i}\in\cR_{i}}\, \cI_i(\br_1^{i}) \, .
\end{align}
\end{definition}
As concrete examples of the last definition, for a generalized alpha investing rule,
the  current potentials $\{W(i)\}_{i\in \naturals}$, potentials
at the last rejection $\{W(\tau_i)\}_{i\in \naturals}$ and total number of rejections $\{R(i)\}_{i\in\naturals}$ are index sequences.

\begin{thm}\label{thm:gen_dep}
Consider a generalized alpha investing rule and assume that the test level $\alpha_j$ is determined based on index function
$\cI_{j-1}$. Namely, for each $j\in\naturals$ there exists a function $g_j:\reals_{\ge 0} \to [0,1]$ such that
$\alpha_j= g_j(\cI_{j-1}(\bR_1^{j-1}))$. Further, assume  $g_j(\,\cdot\,)$  to be nondecreasing
and weakly differentiable with weak derivative $\dg_j(s)$. 

Then, the following upper bound holds for general dependencies among $p$-values:
\begin{eqnarray}\label{eq:BoundDependent}
\FDR(n) \le  \sum_{i=1}^n \bigg\{ g_i(\cImin(i-1))+
\int_{\cImin(i-1)}^{\cImax(i-1)}\, \frac{\dg_i(s)}{\RL_i(s)+1} \de s\bigg\}\,.
\end{eqnarray}
\end{thm}
The proof of this theorem is presented in Appendix~\ref{sec:ProofDependent}.

\begin{example}[$\FDR$ control for dependent test statistics via modified $\LORD$]
We can modify  \LORD as to achieve FDR control even under dependent test statistics. As before, 
we let $\psi_i=b_0$. However,
we fix a sequence $\bxi= (\xi_i)_{i\in\naturals}$, $\xi_i\ge 0$, and set test levels according to rule 
$\alpha_i = \varphi_i=\xi_i\, W(\tau_i)$. 
In other words, compared with the original $\LORD$ procedure, we discount the capital accumulated at the last discovery
as a function of the number of hypotheses tested so far, rather than the number of hypotheses tested since the last discovery.

This rule satisfies the assumptions
of Theorem~\ref{thm:gen_dep}, with index sequence $\cI_{i-1} = W(\tau_i)$ and  $g_i(s) = \xi_i s$. Further, $\cImin(0) =w_0$, $\cImin(i-1) =b_0$ for $i\ge 2$,
and $\cImax(i-1)\le w_0+b_0(i-1)$, and $\RL_i(s) \ge (\frac{s-w_0}{b_0})_+$. Substituting in Equation~(\ref{eq:BoundDependent}), we obtain,
assuming $w_0\le b_0$
 \begin{eqnarray*}
\FDR(n) &\le& w_0 \xi_1 + 
\sum_{i=2}^n \left(b_0\xi_i + \int_{b_0}^{w_0+b_0(i-1)} \frac{b_0 \xi_i}{s-w_0+b_0}\,\de s\right)\\
&\le& w_0 \xi_1 +  \sum_{i=2}^n b_0 \xi_i (1+\log(i))\\
& \le& \sum_{i=1}^n b_0 \xi_i (1+\log(i)) \,.
\end{eqnarray*}
Hence, this rule controls \FDR\, below level $\alpha$ under general dependency structure, if coefficients $(\xi_i)_{i\in \naturals}$ are set such that $\sum_{i=1}^\infty \xi_i(1+\log(i))\le \alpha/b_0$.
\end{example}
%
%
\section{Statistical power}\label{sec:power}

The class of generalized alpha investing rules is quite broad. In order to compare different approaches, 
it is important to estimate their statistical power.

Here, we consider a mixture model wherein each null hypothesis is false with probability $\pi_1$ independently
of other hypotheses, and the $p$-values corresponding to different hypotheses are mutually independent.
Under the null hypothesis $H_i$, we have $p_i$ uniformly distributed in $[0,1]$ and under its alternative,
$p_i$ is generated according to a distribution whose c.d.f is denoted by $F$. We let $G(x) = \pi_0\,x+ \pi_1\, F(x)$,
with $\pi_0+\pi_1=1$, be
the marginal distribution of the $p$-values.  For presentation clarity, we assume that $F(x)$ is continuous. 
 
While the mixture model is admittedly idealized, it offers a natural ground to compare online procedures to offline 
procedures. Indeed, online approaches are naturally favored if the true non-nulls arise at the beginning of the sequence of 
hypotheses, and naturally unfavored if they only appear later. 
On the other hand, if the $p$-values can be processed offline, we can always apply an online rule after a random re-ordering 
of the hypotheses. By exchangeability, we expect the performance to be similar to the ones in the mixture model. 

The next theorem lower bounds the statistical power of \LORD under the mixture model. This lower bound applies to any of the three versions of $\LORD$.
\begin{thm}\label{thm:dependence}
Consider the mixture model with $G(x)$ denoting the marginal distribution of $p$-values. Further, let $\Omega_0(n)$ 
(and its complement $\Omega_0^c(n)$) be the subset of true nulls (non-nulls), among the first $n$ hypotheses. 
Then, the average power of \LORD rule is  almost surely bounded as follows:
\begin{align}\label{eq:powerLORD}
\lim\inf_{n\to \infty} \frac{1}{|\Omega_0^c(n)|} \sum_{i\in \Omega_0^c(n)} R_i  \ge 
\Big(\sum_{m=1}^\infty \prod_{\ell=1}^m \big(1-G(b_0 \gamma_\ell)\big) \Big)^{-1}\,.
\end{align}
\end{thm}
Proof of Theorem~\ref{thm:dependence} is deferred to Appendix~\ref{app:dependence}. The lower bound is in fact the exact power for a 
slightly weaker rule that resets the potential at level $b_0$ after each discovery (in other words, Equation~(\ref{eq:PhiLORD}) is replaced by $\varphi_i =\gamma_{i-\tau_i}b_0$). This procedure is weaker only when multiple discoveries are made in a short interval of time. Hence, the above bound is expected to be accurate when $\pi_1$ is small, and discoveries are rare.

Recall that in \LORD, parameters $\bgamma= (\gamma_\ell)_{\ell=1}^\infty$ can be any sequence of non-negative, monotone non-increasing numbers that sums up to one. This leaves a great extent of flexibility in choosing $\gamma$. The above 
lower bound on statistical power under the mixture model provides useful insight on what are good choices of $\bgamma$.

We first simplify the lower bound further. We notice that 
$\prod_{\ell=1}^m \big(1-G(b_0 \gamma_\ell)\big) \le \exp\big(-\sum_{\ell=1}^m G(b_0 \gamma_\ell)\big)$.
Further, by the monotonicity property of $\bgamma$, we have $G(b_0\gamma_\ell) \ge G(b_0\gamma_m)$ for $\ell\le m$. Thus, 
\begin{align*}
\lim_{n\to \infty} \frac{1}{|\Omega_0^c(n)|} \sum_{i\in \Omega_0^c(n)} R_i \ge \cA(G,\bgamma)\,, \quad 
\cA(G,\bgamma) = \Big(\sum_{m=1}^\infty e^{-m G(b_0 \gamma_m)} \Big)^{-1}\,.
\end{align*}
In order to choose $\bgamma$, we use the lower bound $\cA(G,\bgamma)$ as a surrogate objective function.
We let $\bgo$ be the sequence that maximizes $\cA(G,\bgamma)$. The following proposition characterizes the 
asymptotic behavior of $\bgo$.
\begin{propo}\label{pro:optpower}
Let $\bgo$ be the sequence that maximizes $\cA(G,\bgamma)$ under the constraint $\sum_{\ell=1}^\infty \gamma_m = 1$. 
Further suppose that $F(x)$ is concave and differentiable on an interval $[0,x_0)$ for some $x_0 \in (0,1)$.   
Then there is a constant $\eta = \eta(G,\pi_1)$ independent of $m$ such that, for all $m$ large enough, the following holds true:
\begin{eqnarray*}
\quad\;  \frac{1}{b_0} G^{-1}\Big(\frac{1}{m} \log\Big(\frac{m(1-\pi_1)}{\eta}\Big) \Big)\le \gamma^{\rm opt}_m
\le \frac{1}{b_0} G^{-1}\Big(\frac{2}{m} \log \Big(\frac{1}{\eta G^{-1}(1/m)}\Big) \Big) \, .
\end{eqnarray*}
\end{propo}
The proof of Proposition~\ref{pro:optpower} is given in Appendix~\ref{app:optpower}.

The concavity assumption of $F(x)$ requires the density of non-null $p$-values  (i.e., $F'(x)$) to be non-increasing in a neighborhood $[0,x_0)$. This is a reasonable assumption because significant $p$-values are generically small and the assumption states that, in a neighborhood of zero, smaller values have higher density than larger values. 
In Appendix~\ref{gammaEx}, we compute the optimal sequence $\bgo$ for two case examples.

%
%
\section{Numerical simulations}\label{sec:simulation-syn}

In this section we carry out some numerical experiments with synthetic data. For an application with real data,
we refer to Appendix~\ref{sec:AdClick}.

\subsection{Comparison with off-line rules}
\label{sec:ComparisonOffline}

In our first experiment, we consider hypotheses $\cH(n) = (H_1, H_2, \dotsc, H_n)$ concerning the means of  normal 
distributions. The null hypothesis is  $H_j:\; \theta_j = 0$.  We observe test statistics 
$Z_j = \theta_j +\eps_j$, where $\eps_j$ are independent  standard normal random variables.  
Therefore, one-sided $p$-values are given by  $p_j = \Phi(-Z_j)$, and two sided $p$-values by $p_j = 2\Phi(-|Z_j|)$. 
Parameters $\theta_j$ are set according to a  mixture model:
\begin{align}
\theta_j \sim \begin{cases}
0 & \text{w.p.} \quad 1-\pi_1\,,\\
F_1 & \text{w.p.}\quad  \pi_1\,.
\end{cases}
\end{align}
In our experiment, we set $n = 3000$ and  and use the following three choices of the non-null distribution:
%
%

{\bf Gaussian.} In this case the alternative $F_1$ is $\normal(0,\sigma^2)$ with $\sigma^2 = 2\log n$. 
This choice of $\sigma$ produces parameters $\theta_j$
in the interesting regime in which they are detectable, but not easily so.
In order to see this recall that, under the global null hypothesis, $Z_i\sim \normal(0,1)$ and  $\max_{i\in [n]}Z_i\sim \sqrt{2\log n}$ 
with high probability. Indeed $\sqrt{2\log n}$ is the minimax amplitude for estimation in the sparse Gaussian 
sequence model \cite{DJ94a,Jo94a}.

In this case we carry out two-sided hypothesis testing.

{\bf Exponential.} In this case the alternative  $F_1$ is exponential $\Exp(\lambda)$ with mean $\lambda^{-1} = \sqrt{2\log n}$.
The rationale for this choice is the same given above. 
The alternative is known to be non-negative, and hence we carry out one-sided hypothesis testing.

{\bf Simple.} In this example, the non-nulls are constant and equal to $A = \sqrt{\log n}$. 
Again, we carry out one-sided tests in this case. 
%

We consider three online testing rules, namely alpha investing (AI),  $\LORD$ (a special case of alpha spending with 
rewards) and  Bonferroni.  We also simulate the expected reward optimal (ERO) 
alpha investing rule introduced in \cite{generalized-alpha}. For a brief overview of the ERO notion, recall that in a generalized alpha investing rule, pay-out $\varphi_j$,
test level $\alpha_j$ and the reward $\psi_j$ should satisfy inequalities~\eqref{eq:A1a} and~\eqref{eq:A1b}. An ERO procedure finds the optimal point of trade-off
between $\alpha_j$ and $\psi_j$, for a given value of $\varphi_j$, where optimality criterion is the expected reward of the current test, i.e., $\E(R_j) \psi_j$.  
We compare performance of these online methods with the (adaptive) BH procedure, which as emphasized already, is an offline testing rule:
it has access to the number of hypotheses and $p$-values in advance, while the former algorithms receive $p$-values in an online manner, without knowing the total number of hypotheses. We use Storey's variant of BH rule,
 that is better suited to cases in which the fraction of non-nulls $\pi_1$ is not necessarily small
\cite{storey2002direct}. 
In all cases, we set as our objective to control $\FDR$  below $\alpha=0.05$.

The different procedures are specified as follows: 
%
%

{\bf Alpha Investing.} We set  test levels according to 
\begin{align}\label{eq:alphaj-clust}
\alpha_j = \frac{W(j)}{1+j-\tau_j}\,,
\end{align}
where $\tau_j$ denotes the time of the most recent discovery before time $j$. This proposal was 
introduced by~\cite{alpha-investing} and boosts statistical power in cases in which the non-null hypotheses
appear in batches. We use parameters $w_0=0.005$ (for the initial potential), and $b_0=\alpha-w_0 = 0.045$ (for the rewards). 
The rationale for this choice is that $b_0$ controls the 
evolution of the potential $W(n)$ for large $n$, while $w_0$ controls its initial value. Hence, the behavior of the 
resting rule for large $n$ is mainly driven by $b_0$.

Note that, by \cite[Corollary 2]{generalized-alpha}, this is an ERO alpha investing rule 
\footnote{Note that, since $\theta_j$ is unbounded under the alternative the maximal power is equal to one.},
under the Gaussian and exponential alternatives. It is worth noting that for the case of exponential alternatives,
alpha investing is indeed an ERO procedure, cf \cite[Theorem 2]{generalized-alpha}.

{\bf ERO alpha investing.} For the case of simple alternative, the maximum power achievable at test $i$
is $\rho_i= \Phi(A+\Phi^{-1}(\alpha_i))$. In this case, we
consider  ERO alpha investing \cite{generalized-alpha} defined by $\varphi_i = (1/10)\cdot
W(i-1)$, and with $\alpha_i$, $\psi_i$ given implicitly by the solution of $\varphi_i/\rho_i=\varphi_i/\alpha_i-1$
and $\psi_i = \varphi_i/\alpha_i+b_0-1$.  We use parameters $b_0=0.045$ and $w_0=0.005$.

{\bf LORD.}  We use $\LORD 3$, and choose the sequence $\bgamma = (\gamma_m)_{m\in \naturals}$ as follows:
\begin{align}\label{eq:gamma-sim}
\gamma_m = C \, \frac{\log (m\vee 2)}{me^{\sqrt{\log m}}} \,,
\end{align}
with $C$ determined by the condition 
$\sum_{m=1}^{\infty} \gamma_m = 1$, which yields  $C\approx 0.07720838$. This choice of $\bgamma$ is 
loosely motivated by Example E.2, given in Appendix~\ref{gammaEx}. Notice, however, that we do not assume the data to be generated with the model
treated in that example. Further, for this case we set parameters $w_0=0.005$ (for the initial potential), and $b_0= 0.045$ (for the rewards). 

{\bf Bonferroni.} We set the test levels as $\alpha_m = \gamma_m \alpha$, where the values of $\gamma_m$ are set
as per Equation~\eqref{eq:gamma-sim}, and therefore $\sum_{m=1}^\infty \alpha_m = \alpha$. 

{\bf Storey.} It is well known that the classical BH procedure satisfies  $\FDR\le \pi_0 \alpha$ where $\pi_0$ is the proportion of  true nulls. 
A number of adaptive rules have been proposed that use a plug-in estimate of $\pi_0$ as a multiplicative correction
 in the BH procedure~\cite{storey2002direct,meinshausen2006estimating,jin2007estimating,jin2008proportion}. Following~\cite{Blanchard:2009}, the adaptive test thresholds are given by $\alpha H(\bp)i/n$ (instead of $\alpha i/n$), where $H(\bp)$ is an estimate of $\pi_0^{-1}$, determined as a function of $p$-values, $\bp = (p_1, \dotsc, p_n)$. 

Here, we focus on Storey-$\lambda$ estimator given by~\cite{storey2002direct}:
\begin{align}
H(\bp) = \frac{(1-\lambda)n}{\sum_{i=1}^n \ind(p_i >\lambda)+1}\, .
\end{align}
Storey's estimator is in general an underestimate of $\pi_0^{-1}$.  A standard choice of $\lambda = 1/2$ is used in the SAM software~\cite{storey2003sam}. 
In~\cite{Blanchard:2009}, it is shown that the choice $\lambda = \alpha$ can have better properties under dependent $p$-values. 
In our simulations we tried both choices of $\lambda$.

Our empirical results are presented in Figure~\ref{fig:FDRPower}.
As we see all the rules control $\FDR$ below the nominal level $\alpha = 0.05$,  as guaranteed by 
Theorem \ref{thm:FDR-control}.
While BH and the generalized alpha investing schemes 
($\LORD$, alpha investing, ERO alpha investing) exploit most of  the allowed amount of false discoveries, 
Bonferroni is clearly too conservative.
A closer look reveals that the generalized alpha investing schemes are somewhat more
conservative than BH. Note however that the present simulations assume the non-nulls to arrive at random times, which is
a more benign scenario than the one considered in Theorem \ref{thm:FDR-control}, where arrival times of non-nulls are adversarial.

In terms of power, 
 $\LORD$ appears particularly effective for small $\pi_1$, while standard alpha investing suffers a loss of power for large $\pi_1$. This is related to the fact that $\varphi_j=\alpha_j/(1-\alpha_j)$ in this case. As a consequence the 
rule can effectively stop after a large number of discoveries, because $\alpha_j$ gets close to one.

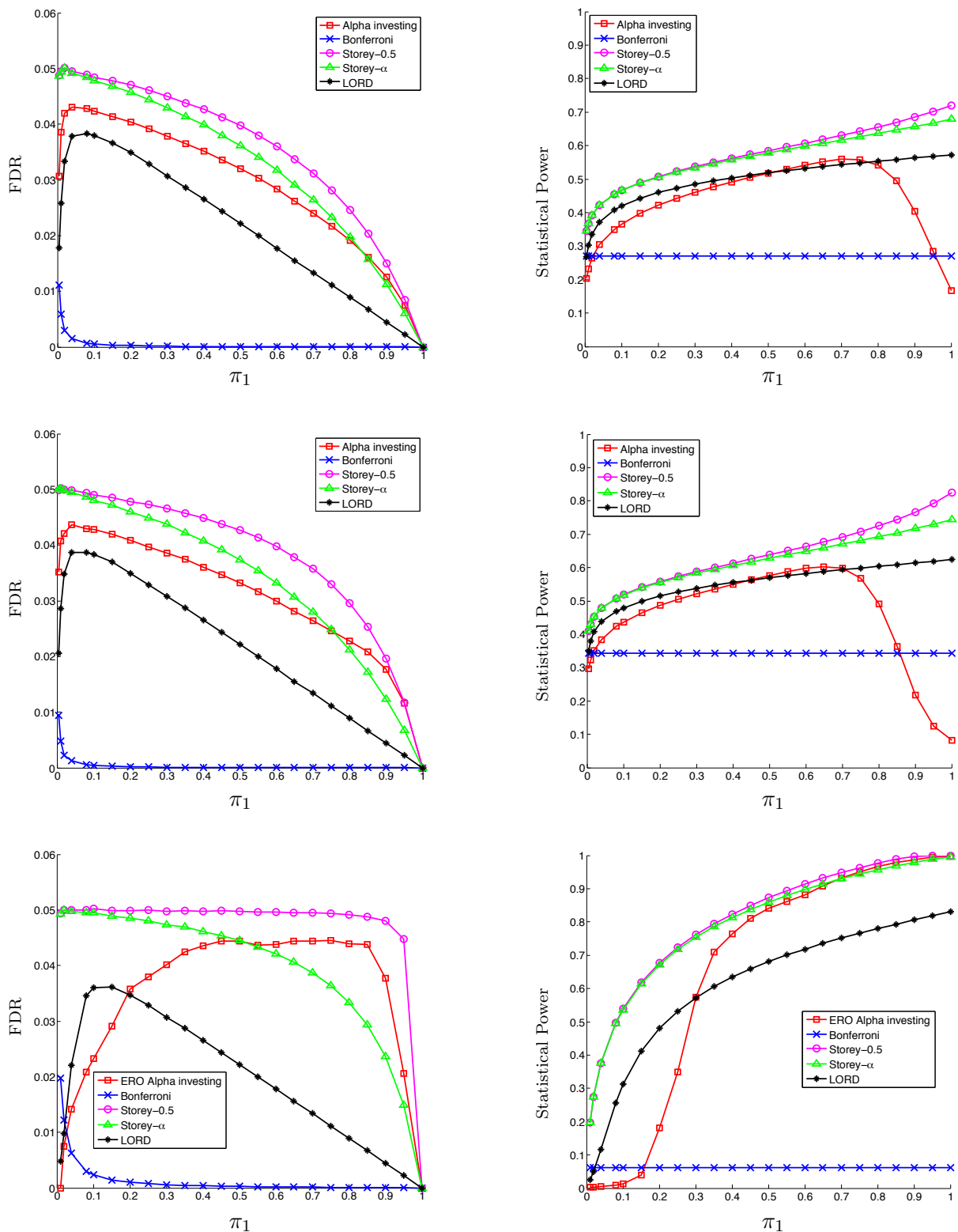
\begin{figure}
\vspace{-0.5cm}

\begin{tabular}{lll}
\includegraphics[width = 6.7cm]{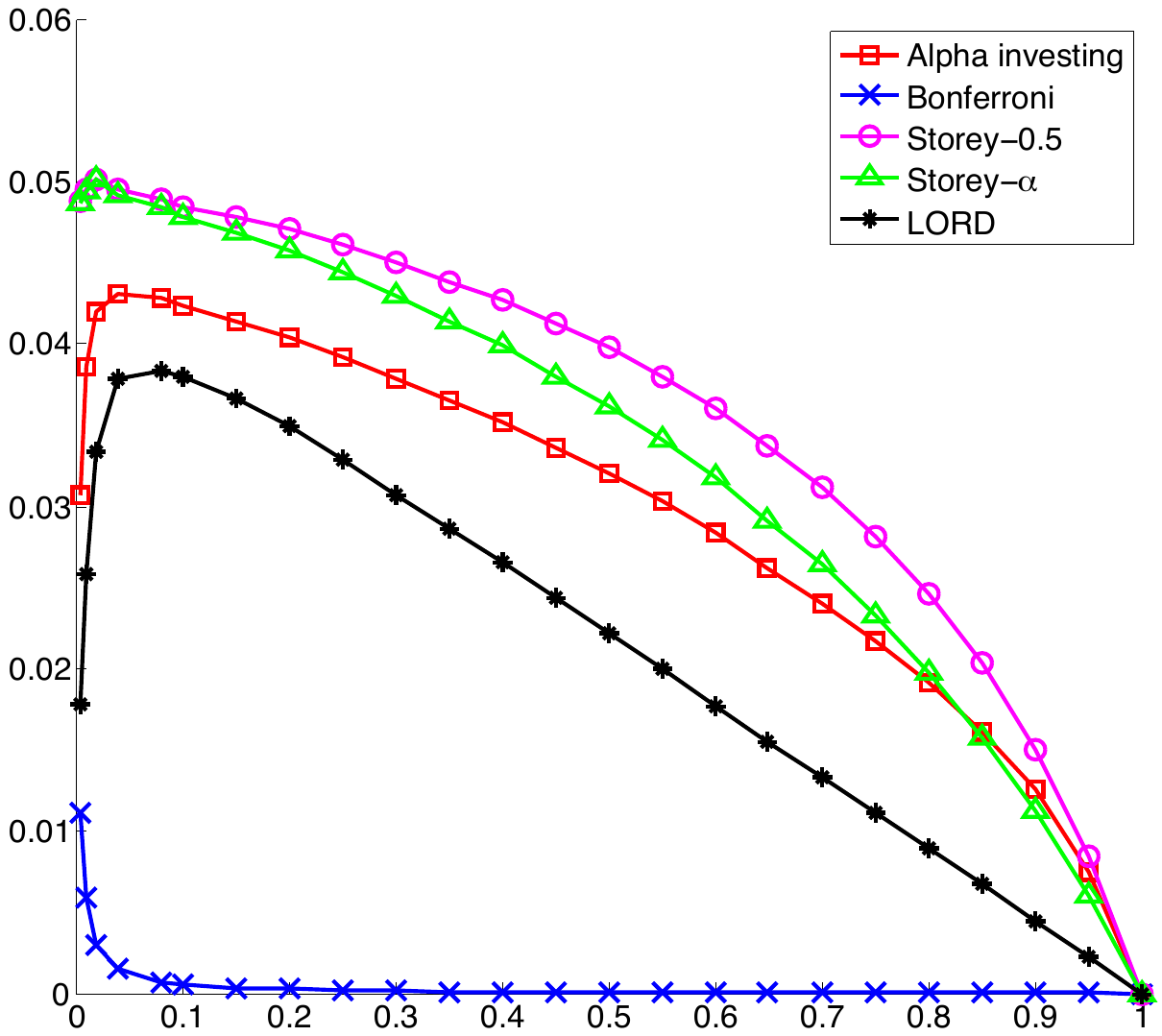}&\phantom{AAAAA}&
\includegraphics[width=6.6cm]{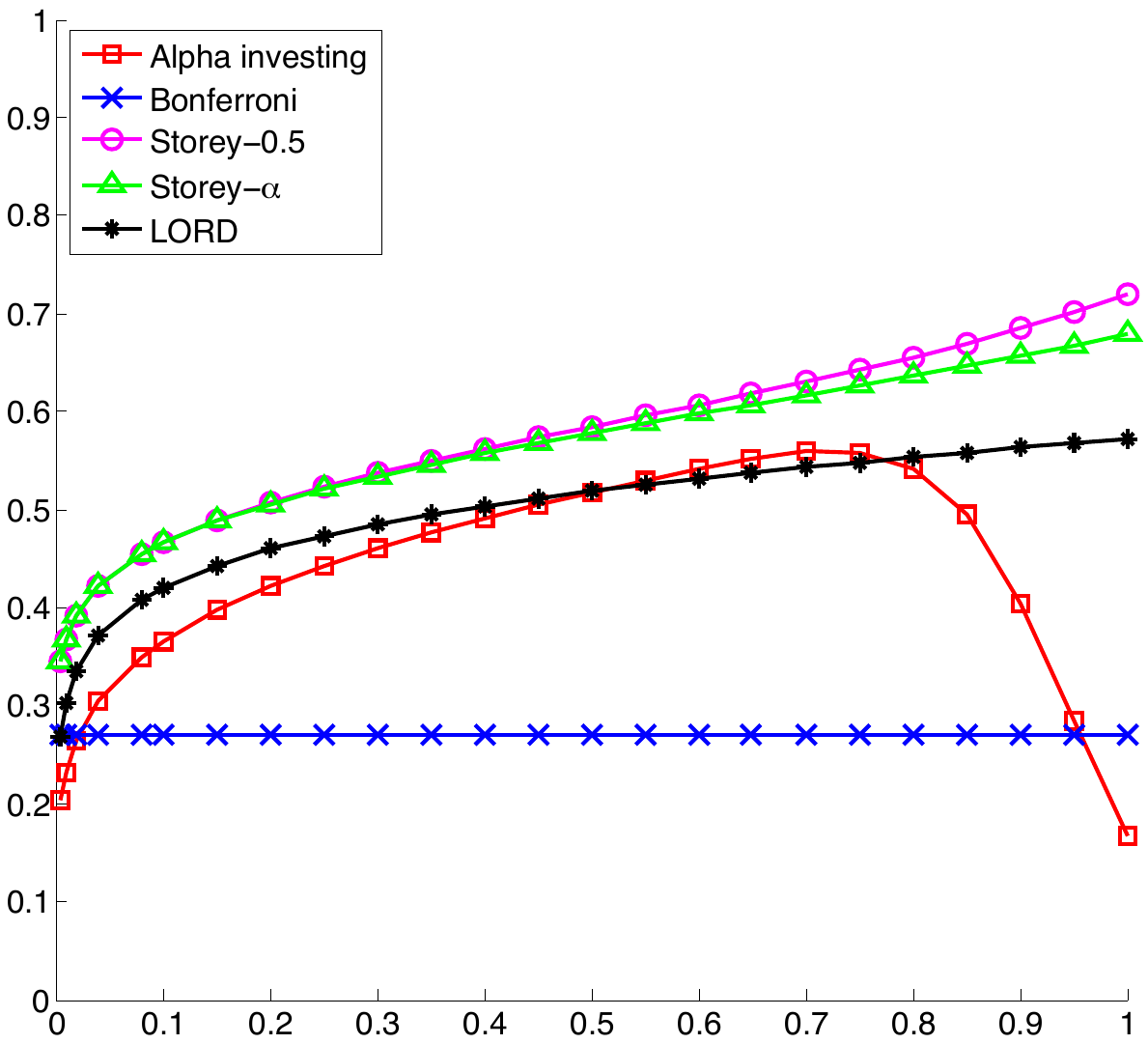}\\
&&\\
&&\\
\includegraphics[width = 6.7cm]{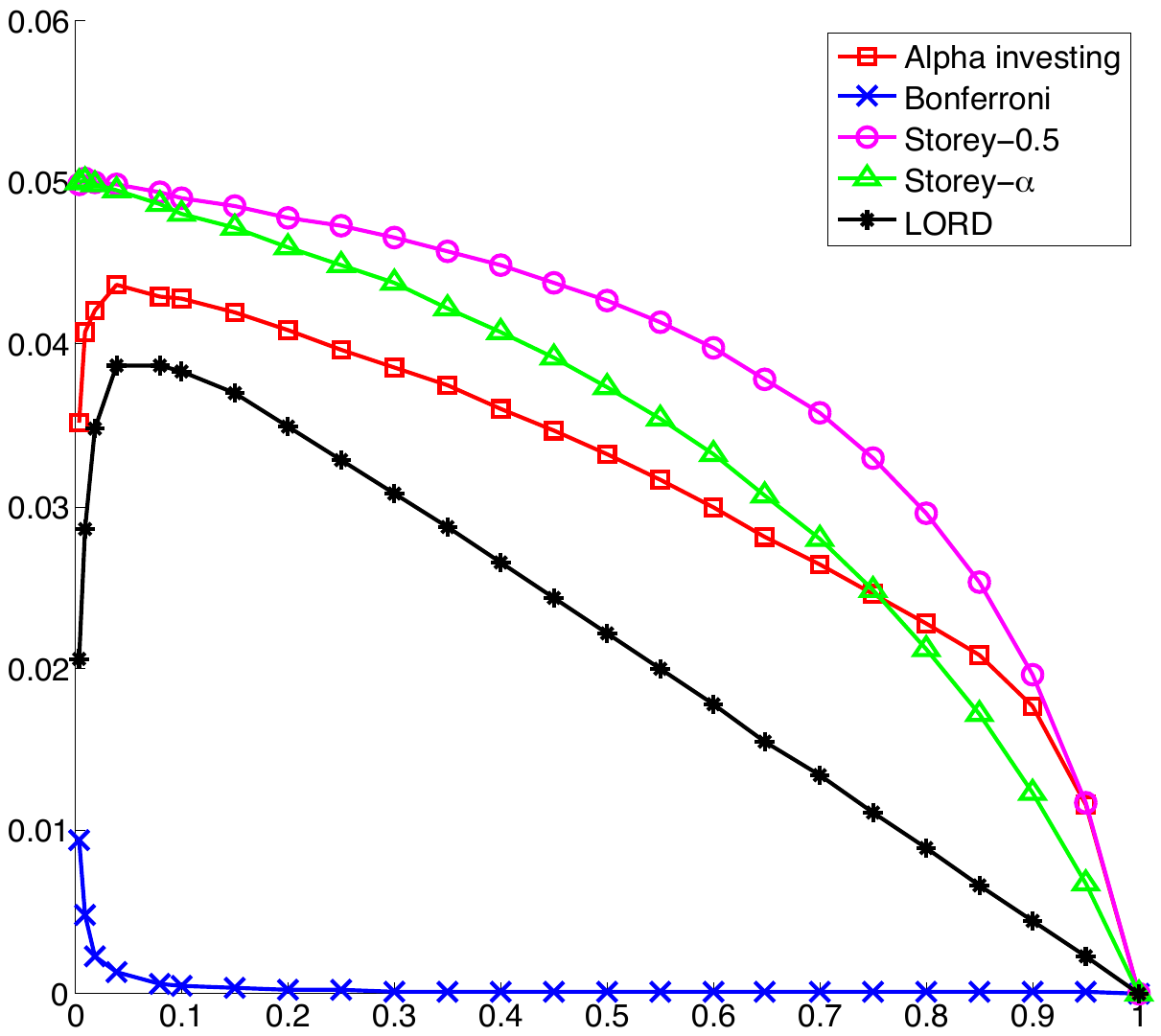}&\phantom{AAAAA}&
\includegraphics[width=6.6cm]{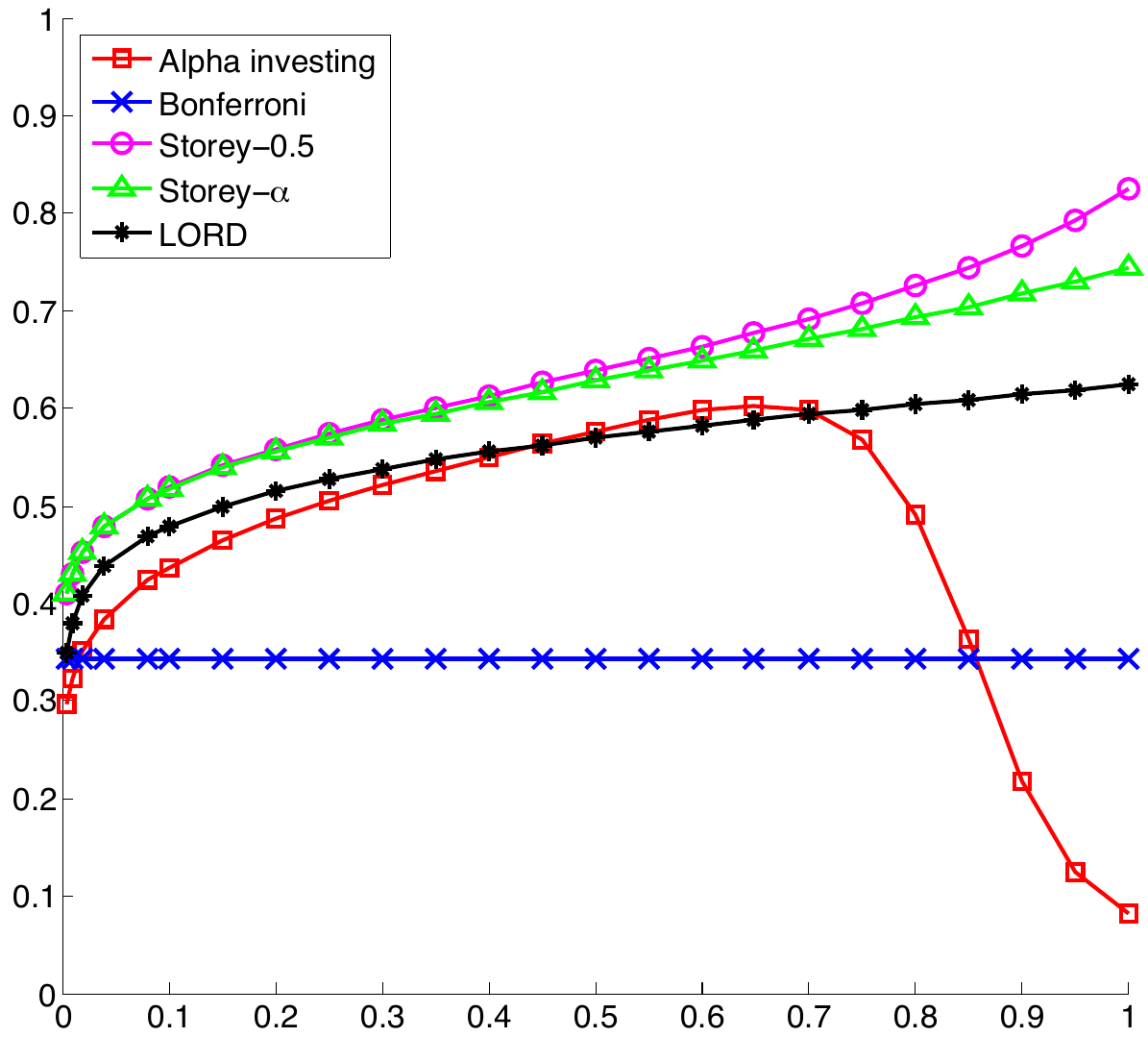}\\
&&\\
&&\\
\includegraphics[width = 6.7cm]{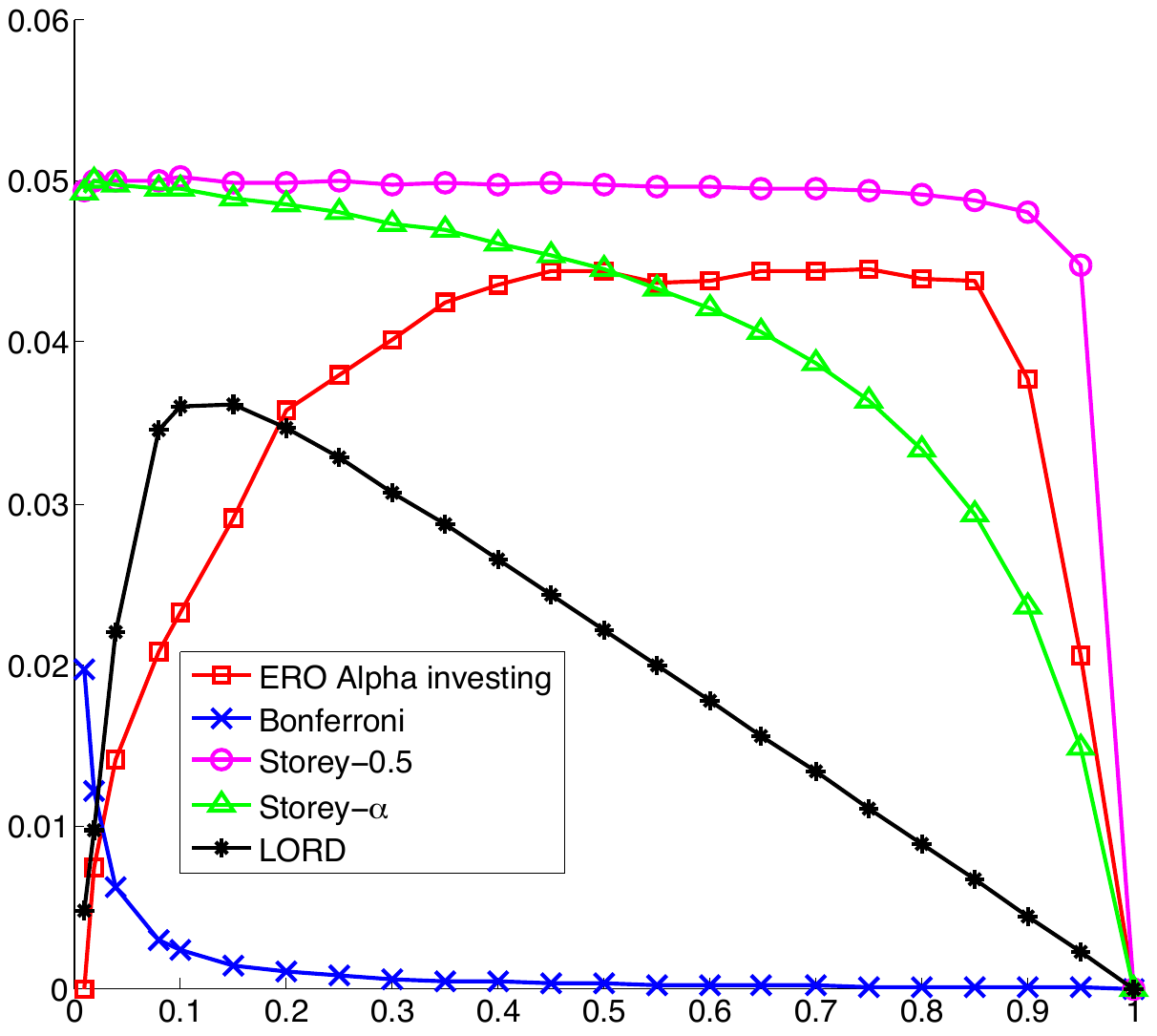}&\phantom{AAAAA}&
\includegraphics[width=6.6cm]{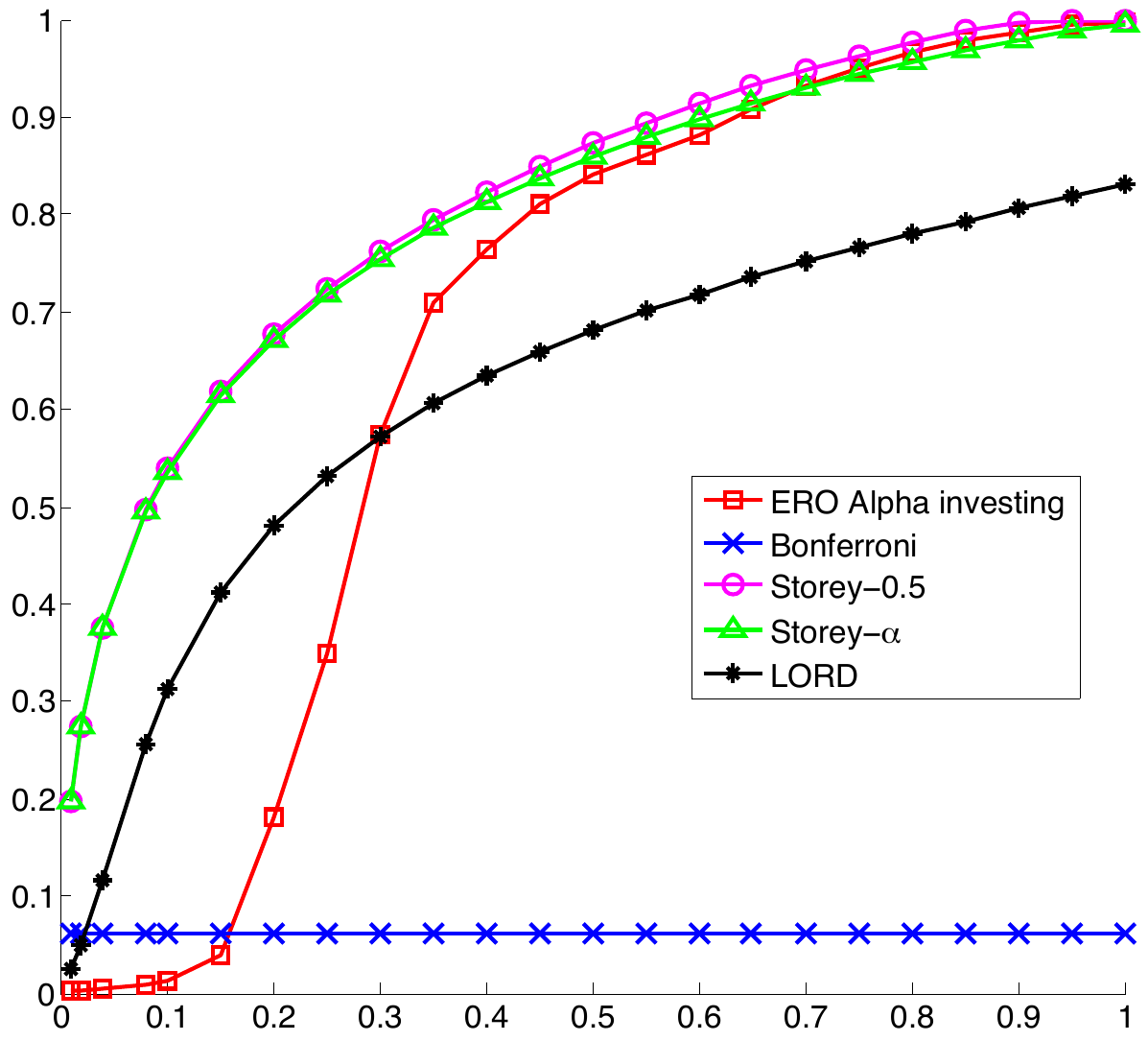}\\
&&\\
&&
\end{tabular}
\put(-355,125){$\pi_1$}
\put(-100,125){$\pi_1$}
\put(-460,215){\rotatebox{90}{{\scriptsize $\FDR$}}}
\put(-208,175){\rotatebox{90}{{\scriptsize Statistical Power}}}
\put(-355,-78){$\pi_1$}
\put(-100,-78){$\pi_1$}
\put(-460,12){\rotatebox{90}{{\scriptsize $\FDR$}}}
\put(-208,-28){\rotatebox{90}{{\scriptsize Statistical Power}}}
\put(-355,-280){$\pi_1$}
\put(-100,-280){$\pi_1$}
\put(-460,-188){\rotatebox{90}{{\scriptsize $\FDR$}}}
\put(-208,-228){\rotatebox{90}{{\scriptsize Statistical Power}}}
\vspace{-0.75cm}

    \caption{{ FDR and statistical power versus fraction of non-null  hypotheses $\pi_1$ for setup described in
      Section~\ref{sec:simulation-syn} with $\alpha = 0.05$. 
The three  rows correspond to Gaussian, exponential, and simple  alternatives (from top to bottom).
FDR and power are computed by averaging over $20,000$ independent trials (for Gaussian and exponential alternatives)
or $500$ trials (for simple alternatives). Here
    hypotheses are considered in random order of arrival. }}\label{fig:FDRPower}
\end{figure}

Figure \ref{fig:FDR-achieved} showcases the FDR achieved by various rules as a function of $\alpha$, for $\pi_1=0.2$
and exponential alternatives. For alpha investing and \LORD we use parameters $b_0=0.9\alpha$ and
$w_0=0.1\alpha$.  The generalized alpha investing rules under consideration have $\FDR$ below the nominal  
$\alpha$, and track it fairly closely. The gap is partly due to the fact that, for large number of discoveries, 
the $\FDR$ of generalized alpha investing rules is closer to $b_0$ than to $\alpha=b_0+w_0$, cf. Remark \ref{remark:BetterFDR}.

\begin{figure}[]
    \centering
        \includegraphics[width = 8cm]{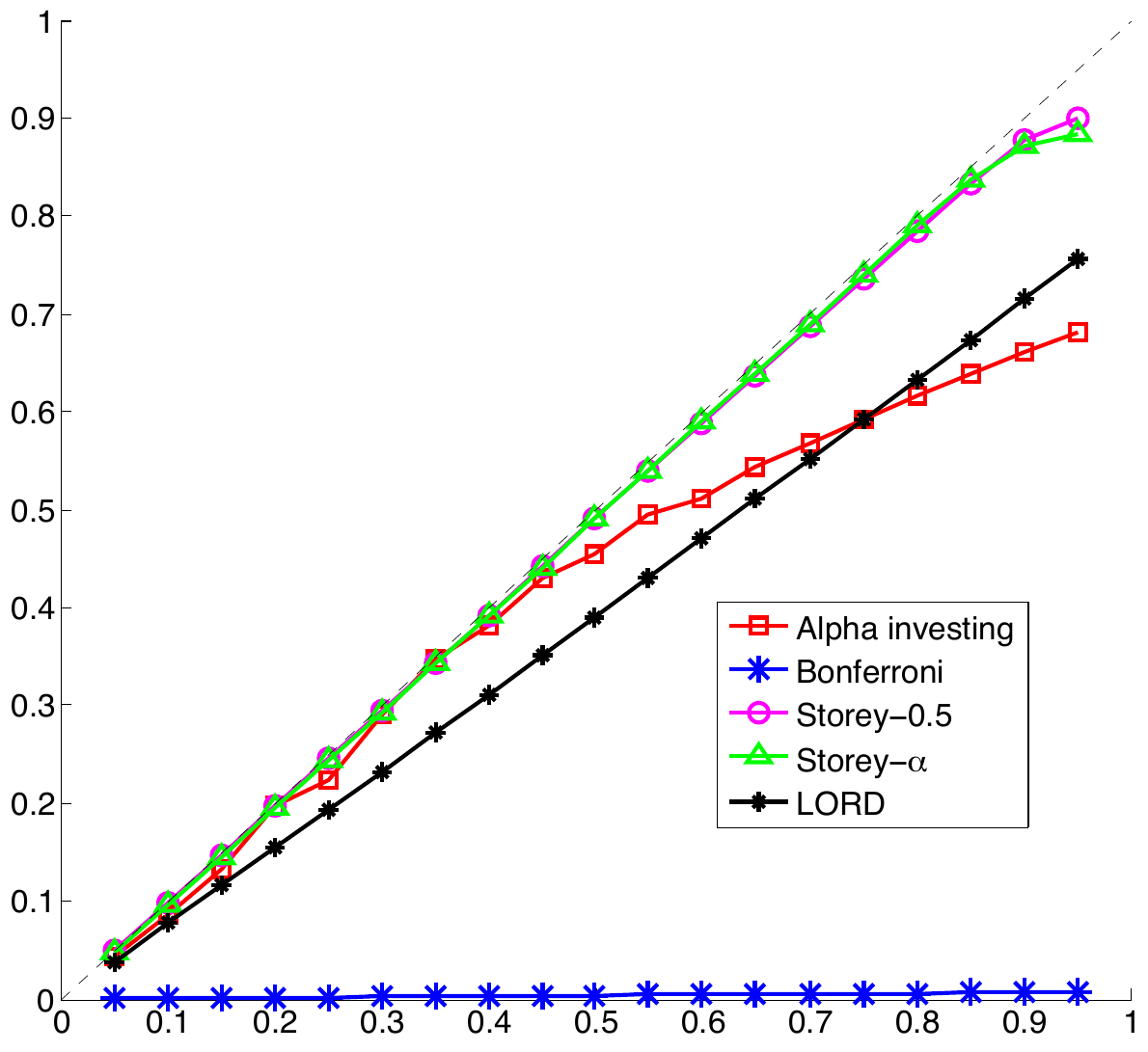}
        \put(-115,-10){$\alpha$}
        \put(-105,-20){\phantom{a}}
         \put(-250,105){{\scriptsize $\FDR$}}
        \label{fig:FDRALL}
    \caption{{ FDR achieved by various methods compared to the target FDR $\alpha$ as $\alpha$ varies. 
Here we use $n=3000$ hypotheses with a proportion $\pi_1 = 0.2$ of non-nulls and exponential alternatives.
The FDR is estimated by averaging over $20,000$ independent trials.}}\label{fig:FDR-achieved}
\end{figure}

\subsection{The effect of ordering}
\label{sec:Ordered}
By definition, the BH rule is insensitive to the order in which the hypotheses are presented. On the contrary, 
the outcome of online testing rules depends on this ordering. This is a weakness, because the ordering of hypotheses
can be adversarial, leading to a loss of  power, but also a strength. Indeed, in some applications, hypotheses can be ordered, using side information,
such that those most \emph{likely} to be rejected come first. 
In these cases, we expect generalized alpha investing
procedures to be potentially \emph{more powerful than benchmark offline rules} as BH.

For instance, Li and Barber \cite{li2016accumulation} analyze a drug-response dataset proceeding in two steps.
First, a family of hypotheses (gene expression levels) are ordered using side information, and then a multiple hypothesis 
testing procedure is applied to the ordered data\footnote{The procedure of \cite{li2016accumulation} is designed
as to reject the first $\hat{k}$ null hypotheses, and accept the remaining $n-\hat{k}$. However, this specific structure  is a design choice,
and is not a constraint arising from the application.}. Other approaches, such as distributing the weights unevenly among the hypotheses~\cite{Genovese06}
are also potentially useful in settings where there is side information about the hypotheses that are more likely to be non-null.

In order to explore the effect of a favorable ordering of the hypotheses, we reconsider the exponential 
model in the previous section, and simulate a case in which side information is available.  
For each trial, we generate the mean $(\theta_j)_{1\le j\le n}$, and two independent sets of
observations $Z_j = \theta_j+\eps_j$, $Z_j' = \theta_j+\eps'_j$, with $\eps_j\sim\normal(0,1)$, $\eps'_j\sim\normal(0,\sigma^2)$
independent.
We then compute the corresponding (one-sided) $p$-values $(p_j)_{1\le j\le n}$, $(p'_j)_{1\le j\le n}$. We use the $p$-values
$(p'_j)_{1\le j\le n}$ to order the hypotheses\footnote{Note that ordering by increasing $p'_j$ is equivalent to ordering by decreasing $|Z'_j|$ and the latter
can be done without knowledge of the noise variance $\sigma^2$. }
 (in such a way that these $p$-values are increasing along the ordering). We then use the
other set of $p$-values $(p_j)_{1\le j\le n}$ to test the null hypotheses $H_{j,0}:\, \theta_j=0$ along this ordering.

Let us emphasize that, for this simulation, better statistical power would be achieved if we computed a single $p$-value $p_j$ by processing jointly $Z_j$ and $Z'_j$. 
However, in real applications, the two sources of  information are heterogenous and this joint processing is not warranted, see  
\cite{li2016accumulation} for a discussion of this point. 

Figure \ref{fig:Ordered} reports the $\FDR$ and statistical power in this setting. We used \LORD with 
parameters $(\gamma_m)_{m\ge 1}$ given by Equation~(\ref{eq:gamma-sim}), and simulated two noise levels for the side information:
$\sigma^2=1$ (noisy ordering information) and $\sigma^2=1/2$ (less noisy ordering).
As expected, with a favorable ordering the $\FDR$ decreases significantly. The statistical power increases as long as the  fraction of 
non-nulls $\pi_1$ is not too large. This is expected: when the fraction of non-nulls is large, ordering is less relevant. 

In particular, for small $\pi_1$, the gain in power can be as large as $20\%$ (for $\sigma^2=1$) and as 
$30\%$ (for $\sigma^2=1/2$). The resulting power is superior to adaptive BH \cite{storey2002direct} for $\pi_1\lesssim 0.15$  
 (for $\sigma^2=1$), or $\pi_1\lesssim 0.25$ (for $\sigma^2=1/2$). 
 
  \begin{figure}
\vspace{-0.5cm}

\begin{tabular}{lll}
\includegraphics[width = 6.7cm]{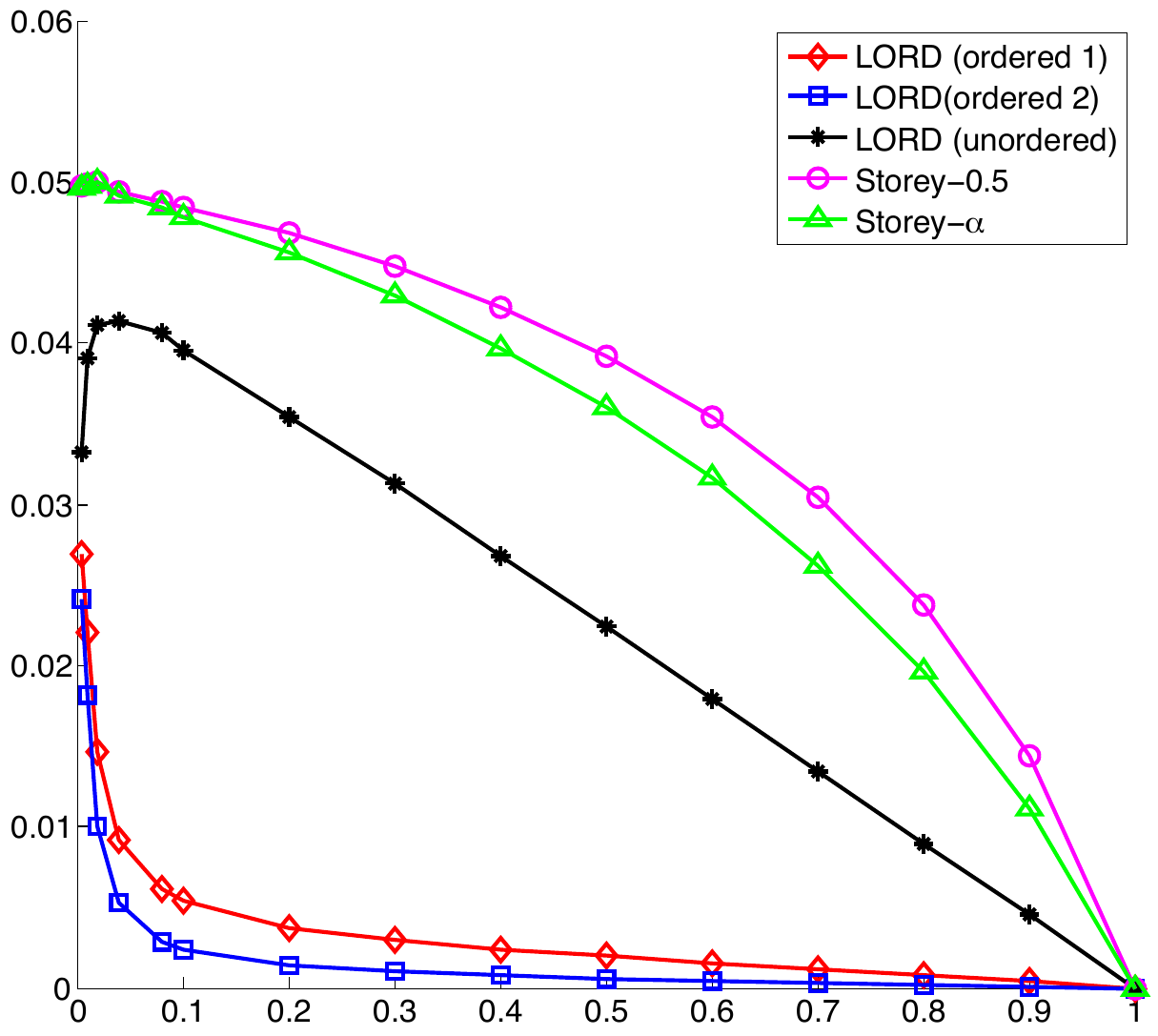}&\phantom{AAAAA}&
\includegraphics[width=6.6cm]{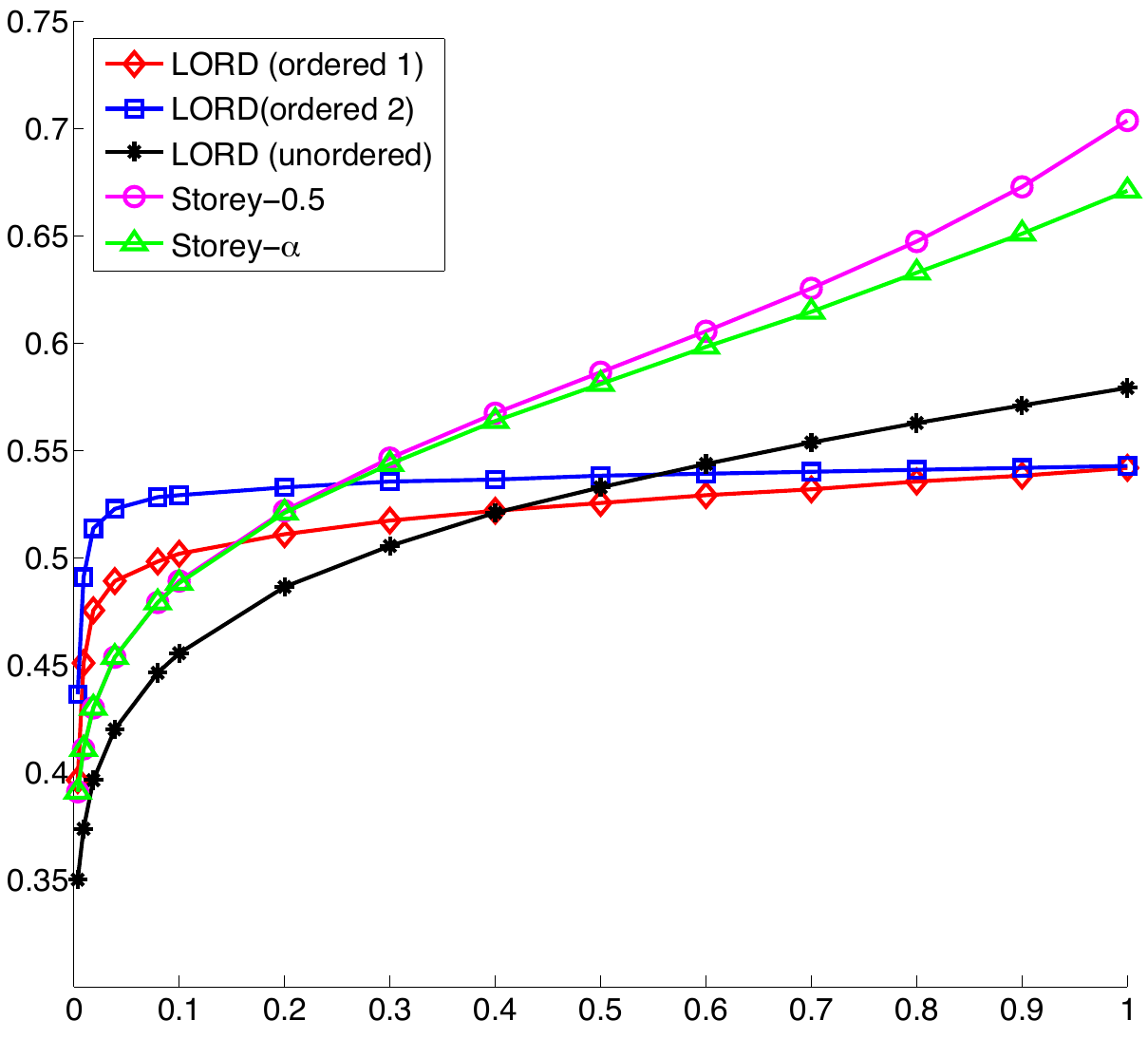}
\end{tabular}
\put(-355,-87){$\pi_1$}
\put(-105,-87){$\pi_1$}
\put(-460,0){\rotatebox{90}{FDR}}
\put(-208,-55){\rotatebox{90}{Statistical Power}}

\caption{{FDR and statistical power for \LORD with favorably ordered hypotheses (setup of Section \ref{sec:Ordered}). Here $n=3000$, $\pi_1=0.2$  and
data are obtained by averaging over $20,000$ trials.
Unordered: Null and non-null hypotheses are ordered at random. Ordered 1: hypotheses are ordered using very noisy side information ($\sigma^2=1$). 
Ordered 2: hypotheses are ordered using less noisy side information ($\sigma^2=1/2$).} }\label{fig:Ordered}
\end{figure}

\subsection{$\FDR$ control versus $\mFDR$ control}

Aharoni and Rosset \cite{generalized-alpha} proved that generalized alpha investing rules control $\mFDR_{w_0/b_0}$. 
Formally,
\begin{align}
\mFDR_{w_0/b_0}(n) \equiv \sup_{\btheta\in\Theta}\frac{\E\, V^{\theta}(n)}{\E\,R(n)+(w_0/b_0)}\le b_0\, .
\end{align}
As mentioned before (see also Appendix~\ref{sec:mfdr_fdr}), this metric has been criticized because 
it does not control a property of the realized sequence
of tests; instead it controls a ratio of expectations. 

Our Theorem~\ref{thm:FDR-control-2} controls a different metric that we called $\sFDR_{\eta}(n)$: 
\begin{align}
\sFDR_{w_0/b_0}(n) \equiv \sup_{\btheta\in\Theta}\E\Big\{\frac{ V^{\theta}(n)}{R(n)+(w_0/b_0)}\Big\}\le b_0\, .
\end{align}
This quantity is the expected ratio, and hence passes the above criticism. Note that both theorems yield control at level $\alpha=b_0$,
for the same class of rules.

Finally, Theorem \ref{thm:FDR-control} controls a more universally accepted metric, namely $\FDR$, at level
$\alpha=w_0+b_0$. A natural question is whether, in practice, we should choose $w_0$, $b_0$ as to guarantee $\FDR$ control
(and hence set $w_0+b_0\le \alpha$) or instead be satisfied with $\mFDR$ and $\sFDR$ control, which allow for $b_0=\alpha$ and hence
potentially larger statistical power. 

While an exhaustive answer to this question is beyond the scope of this paper, we repeated the simulations in
Figure \ref{fig:FDRPower}, using the two different criteria. The results, provided in Appendix~\ref{sec:mfdr_fdr}, suggest that this question 
might not have a simple answer. On one hand, under the setting of Figure \ref{fig:FDRPower} (independent $p$-values, large number of discovery)
$\mFDR$ and $\sFDR$ seem stringent enough criteria. On the other, the gain in statistical power that is obtained from these criteria, rather than $\FDR$, is somewhat marginal. 

%
%
\section{Control of False Discovery Exceedance}

Ideally, we would like to control the proportion of false discoveries in any given
realization of our testing procedures. We recall that this is given by (cf. Equation~(\ref{eq:FDPdef}))
\begin{align}
\FDP^{\theta}(n)\equiv \frac{V^{\theta}(n)}{R(n)\vee 1}\,. 
\end{align}
False discovery rate is  the \emph{expected} proportion of false discoveries.
However --in general-- control of $\FDR$ does not prevent $\FDP$ from varying , even when its average is bounded. 
 In real applications, the actual $\FDP$ might be far from its  expectation.
For instance, as pointed out by \cite{owen2005variance}, the variance of \FDP\, can be
large if the test statistics are correlated.

Motivated by this concern,  the \emph{false discovery exceedance} is defined as
\begin{align}
\FDX_\gamma(n)\equiv \sup_{\btheta\in\Theta}\prob\big(\FDP^\theta(n)\ge \gamma\big)\,.
\end{align}
for a given tolerance parameter $\gamma\ge 0$. Controlling $\FDX$ instead of $\FDR$ gives a stronger preclusion 
from large fractions of false discoveries.

Several methods have been proposed to  control $\FDX$ in an offline setting. Van der Laan, Dudoit and Pollard~\cite{van2004augmentation}
observed that any procedure that controls $\FWER$, if augmented by a
sufficiently small number of rejections, also controls 
$\FDX$.
Genovese and Wasserman~\cite{genovese2006exceedance} suggest
controlling $\FDX$ by inverting a set of uniformity tests on the vector of $p$-values.
Lehmann and Romano~\cite{lehmann2012generalizations} proposed a step-down method to control $\FDX$. 

A natural criterion to impose in the online setting would be the control of $\sup_{n\ge 1}\FDX_{\gamma}(n)$. However, this 
does not preclude the possibility of large proportions of false
discoveries at some (rare) random times $n$.
It could be --as a cartoon example-- that $\FDP^{\theta}(n) = 1/2$ independently
with probability $\alpha$ at each $n$, and $\FDP^{\theta}(n)=
\gamma/2$ with probability $1-\alpha$. In this case $\sup_{n\ge
  1}\FDX_{\gamma}(n)\le \alpha$ but
$\FDP^{\theta}(n)=1/2$ almost surely for infinitely many times
$n$. This is an undesirable situation.  

A more faithful generalization of $\FDX$ to the online setting is therefore
\begin{align}
\FDX_{\gamma} \equiv  \sup_{\btheta\in\Theta}\prob\big(\sup_{n\ge 1}\, \FDP^\theta(n)\ge \gamma\big)\,.
\end{align}
We will next propose a class of generalized alpha investing rules for online control of $\FDX_{\gamma}$.

\subsection{The effect of reducing  test levels}
\label{sec:simulationT}

Before describing our approach, we demonstrate through an example that 
the $\FDP$ can differ substantially from its expectation. We also  want to illustrate how a naive modification of
the previous rules only achieves a better control of this variability at the price of a significant loss in power.

Note that the desired bound $\FDP^{\theta}(n) <\gamma$ follows if we can establish \mbox{$b_0R(n) -V(n) + (\gamma-b_0)>0$}
for some $\gamma\ge b_0 \ge 0$.
Recall that a generalized alpha investing procedure continues until the potential  $W(n)$ remains non-negative. Therefore, 
for such a procedure, it suffices to bound the probability that the stochastic process 
\mbox{$B(n)\equiv b_0 R(n) - W(n)-V(n)+ (\gamma -b_0)$} crosses zero. As we show in Lemma~\ref{subM2}, $B(n)$ 
is a submartingale, and thus in expectation it moves away from  
zero. In order to bound the deviations from the  expectation, consider the submartingale increments 
$B_j \equiv B(j)-B(j-1)$ given by 
\begin{align}
B_j = (b_0 -\psi_j)R_j +\varphi_j - V_j\,.
\end{align}
If the $j$-th null hypothesis is false, i.e.  $\theta_j \neq 0$, we have $V_j = 0$ and $B_j \ge 0$ by invoking assumption \AssI and noting that $R_j\in\{0,1\}$.
Under the null hypothesis, $V_j = R_j$, and 
\begin{align}\label{eq:subMvar}
{\Var}(B_j|\cF_{j-1}) = (b_0 -\psi_j-1)^2 \alpha_j (1-\alpha_j) \,.
\end{align}
Reducing ${\Var}(B_j|\cF_{j-1})$ lowers variations of the submartingale and hence the variation of the
false discovery proportions.  Note that for a generalized alpha investing rule, if we keep $b_0$, $\psi_j$  
unchanged and lower the test levels $\alpha_j$, the rule still satisfies conditions \AssI, \AssII and thus controls $\FDR$ at the desired level. On the other hand, this modification decreases ${\Var}(B_j|\cF_{j-1})$ as per Eq~\eqref{eq:subMvar}. 
In summary, reducing the test levels has the effect of reducing the 
variation of false discovery proportion at the expense of reducing statistical power.

We carry out a numerical experiment within a similar setup as the one discussed in Section~\ref{sec:simulation-syn}.
A   set of
$n$ hypotheses are tested, each specifying mean of a normal distribution, $H_j:\; \theta_j = 0$. 
The test statistics are independent, normally distributed random variables $Z_j \sim \normal(\theta_j,1)$.
For non-null hypotheses, we set $\theta_j = 3$. The total number of tests is $n = 1000$ of which the first 
$100$ are non-null.  

We consider three different testing rules, namely alpha investing, alpha spending with rewards and $\LORD$,
all ensuring $\FDR$ control at level  $\alpha = 0.05$. The details 
of these rules as well as the choice of parameters is the same as Section~\ref{sec:simulation-syn}. 

In order to study the effect of reducing test levels, for each of these rules we truncate them by a threshold value $T$, i.e.
we use $\alpha_j^T = \alpha_j \vee T$.
We plot the histogram of false discovery proportions using $30,000$ replications of the test statistics sequence. We further report standard deviation and $0.95$ quantile of FDPs. 
The results are shown in Figs.~\ref{fig:AIT}, \ref{fig:RewT}, \ref{fig:LORDT}. 

As a first remark, while all of the rules considered control 
FDR below $\alpha=0.05$, the actual false discovery proportion in Figs.~\ref{fig:AIT}, \ref{fig:RewT}, \ref{fig:LORDT}
has a very broad distribution. Consider for instance alpha investing, at threshold level $T=0.9$. Then
FDP exceeds $0.15$ (three times the nominal value) with probability 0.13.

Next we notice that reducing the test levels (by reducing $T$) has the desired effect of
reducing the variance of the FDP. This effect is more pronounced for alpha investing.  
Nevertheless quantifying this effect is challenging due to the complex dependence between $B_j$ and history 
$\cF_{j-1}$. This makes it highly nontrivial to adjust threshold $T$ to obtain $\FDX_\gamma\le \alpha$.
In the next section we achieve this through a different approach.
\begin{figure}[p]
    \centering
    \subfigure[$T = 0.9$]{
        \includegraphics[width = 4.9cm]{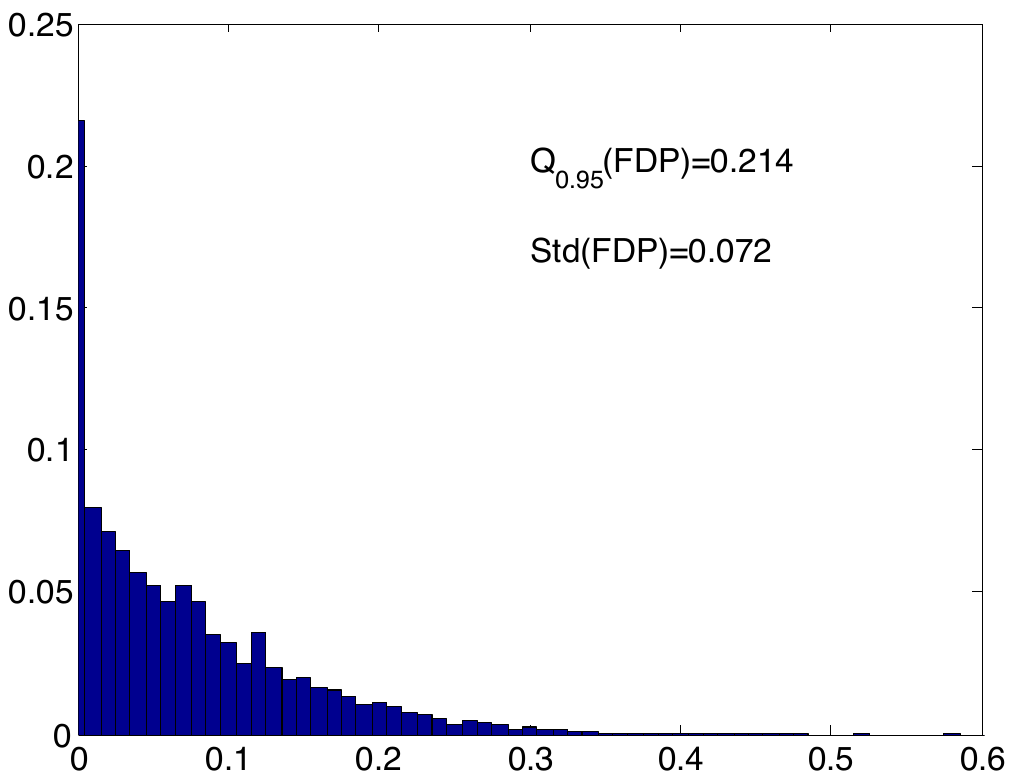}
        \put(-75,-10){{\scriptsize $\FDP$}}
        \put(-14,-15){\phantom{b}}
        \put(-150,40){\rotatebox{90}{{\scriptsize frequency}}}   
        \label{fig:AIT09}
        }
    \subfigure[$T = 0.6$]{
        \includegraphics[width=5cm]{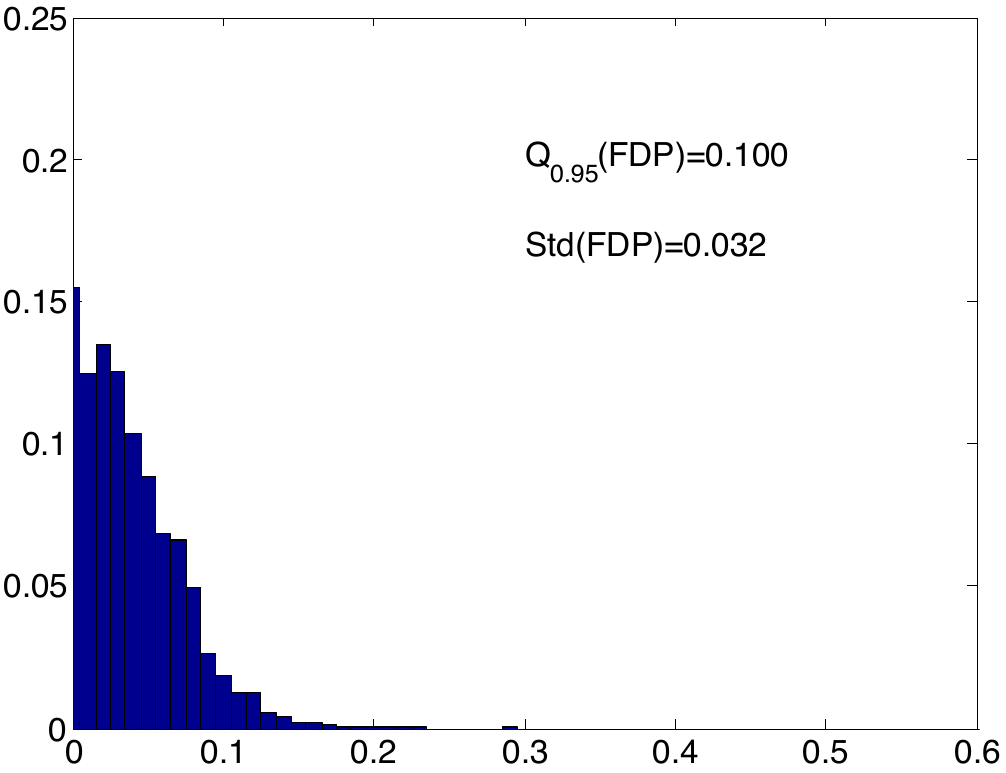}
        \put(-75,-10){{\scriptsize $\FDP$}}
        \put(-14,-15){\phantom{b}}
        \put(-150,40){\rotatebox{90}{{\scriptsize frequency}}}   
        \label{fig:AIT07}
        }
     \subfigure[$T = 0.3$]{
        \includegraphics[width=5cm]{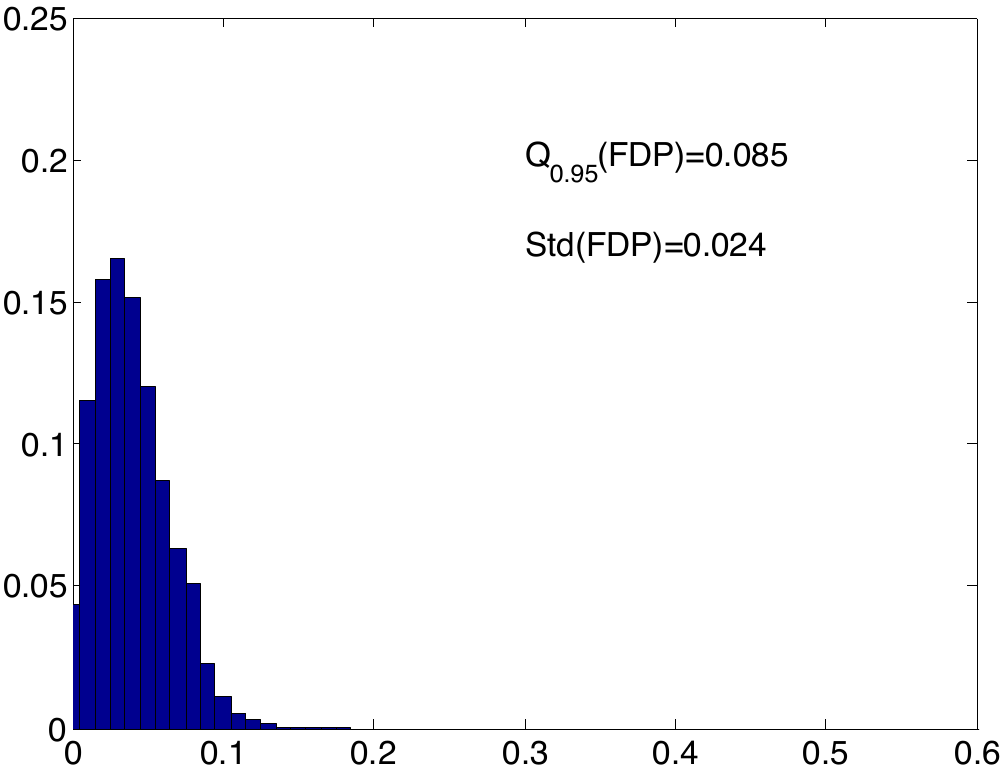}
        \put(-75,-10){{\scriptsize $\FDP$}}
        \put(-14,-15){\phantom{b}}
        \put(-150,40){\rotatebox{90}{{\scriptsize frequency}}}   
        \label{fig:AIT05}
        }
    \caption{{Histogram of FDP for alpha investing rule with different values of $T$}}\label{fig:AIT}
\end{figure} 
\begin{figure}[p]
    \centering
    \subfigure[$T = 0.9$]{
        \includegraphics[width = 5cm]{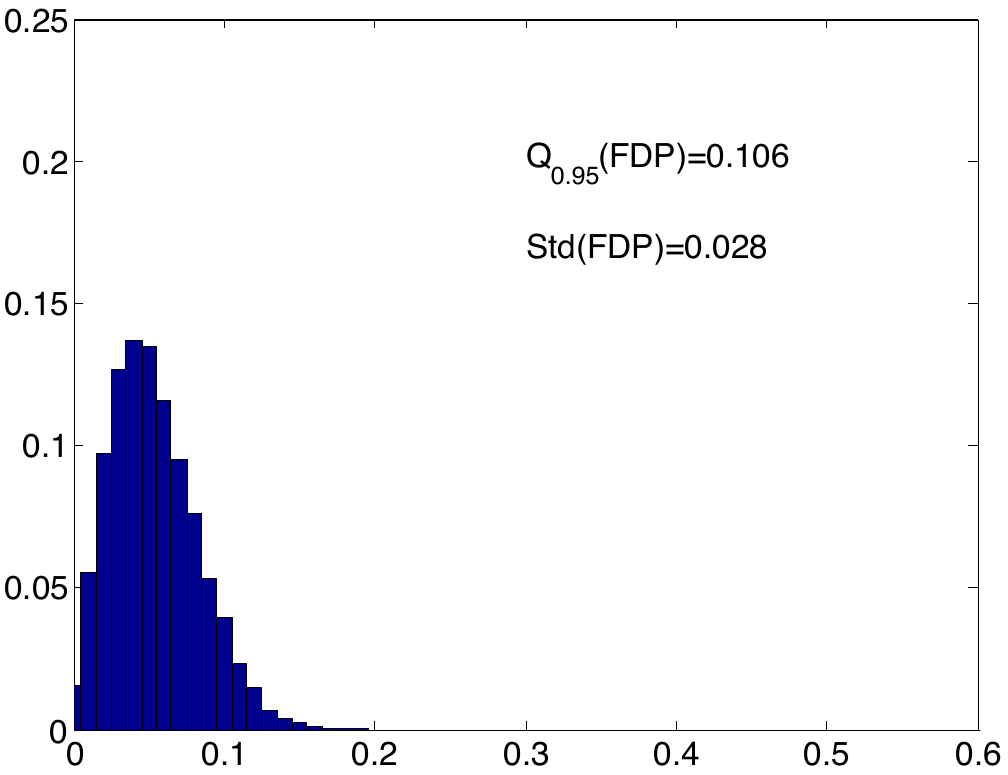}
        \put(-75,-10){{\scriptsize $\FDP$}}
        \put(-14,-15){\phantom{b}}
        \put(-150,40){\rotatebox{90}{{\scriptsize frequency}}}   
        \label{fig:RewT09}
        }
    \subfigure[$T = 0.6$]{
        \includegraphics[width=5cm]{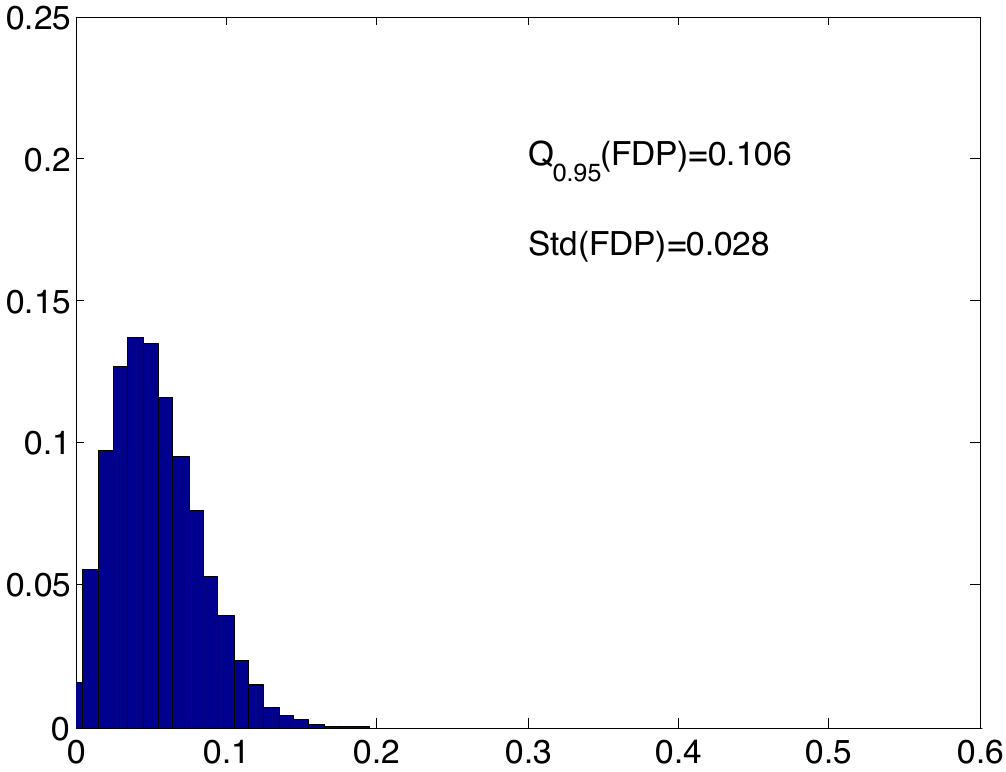}
        \put(-75,-10){{\scriptsize $\FDP$}}
        \put(-14,-15){\phantom{b}}
        \put(-150,40){\rotatebox{90}{{\scriptsize frequency}}}   
        \label{fig:RewT07}
        }
     \subfigure[$T = 0.3$]{
        \includegraphics[width=5cm]{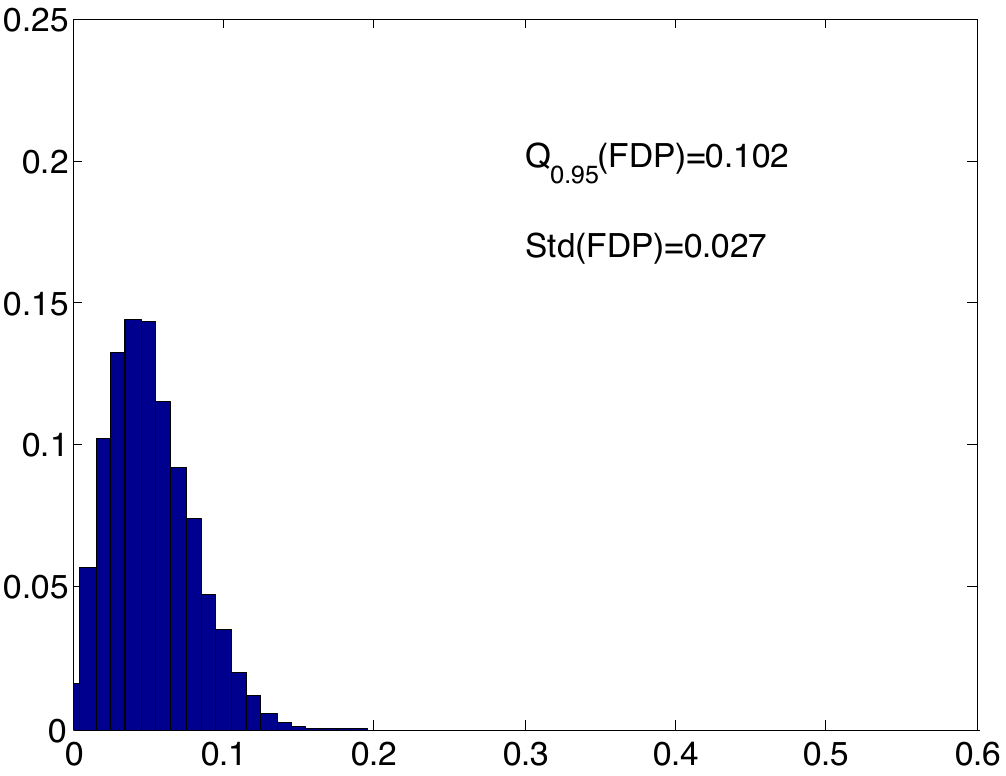}
        \put(-75,-10){{\scriptsize $\FDP$}}
        \put(-14,-15){\phantom{b}}
        \put(-150,40){\rotatebox{90}{{\scriptsize frequency}}}   
        \label{fig:RewT05}
        }
    \caption{{ Histogram of FDP for alpha spending with rewards for different values of $T$}}\label{fig:RewT}
\end{figure} 
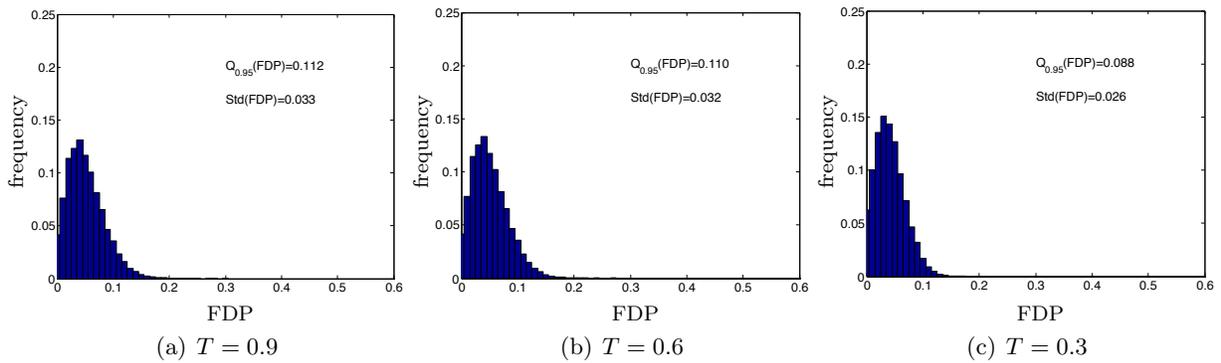
\begin{figure}[p]
    \centering
    \subfigure[$T = 0.9$]{
        \includegraphics[width = 5cm]{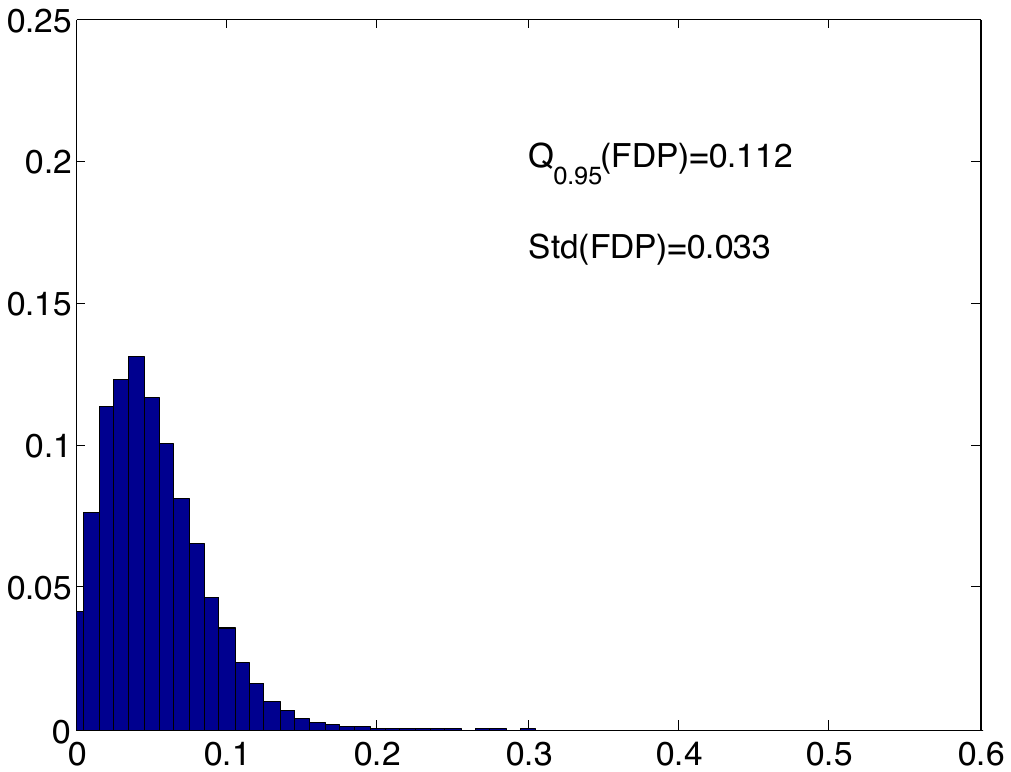}
        \put(-75,-10){{\scriptsize $\FDP$}}
        \put(-14,-15){\phantom{b}}
        \put(-150,40){\rotatebox{90}{{\scriptsize frequency}}}   
        \label{fig:LORD09}
        }
    \subfigure[$T = 0.6$]{
        \includegraphics[width=5cm]{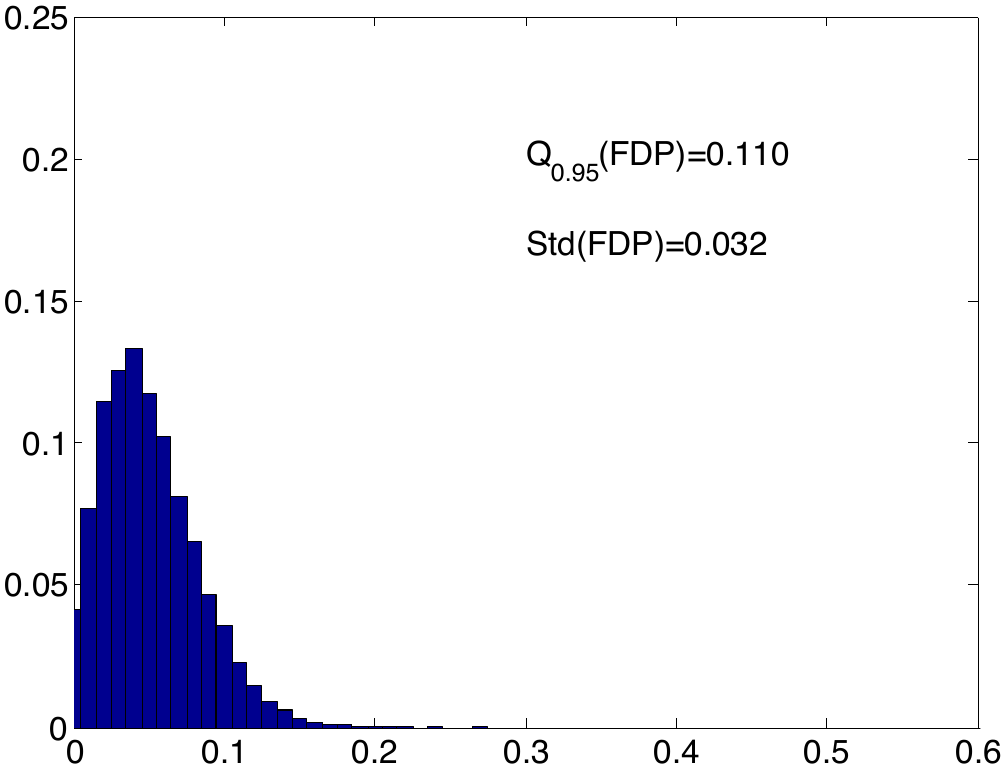}
        \put(-75,-10){{\scriptsize $\FDP$}}
        \put(-14,-15){\phantom{b}}
        \put(-150,40){\rotatebox{90}{{\scriptsize frequency}}}   
        \label{fig:LORD07}
        }
     \subfigure[$T = 0.3$]{
        \includegraphics[width=5cm]{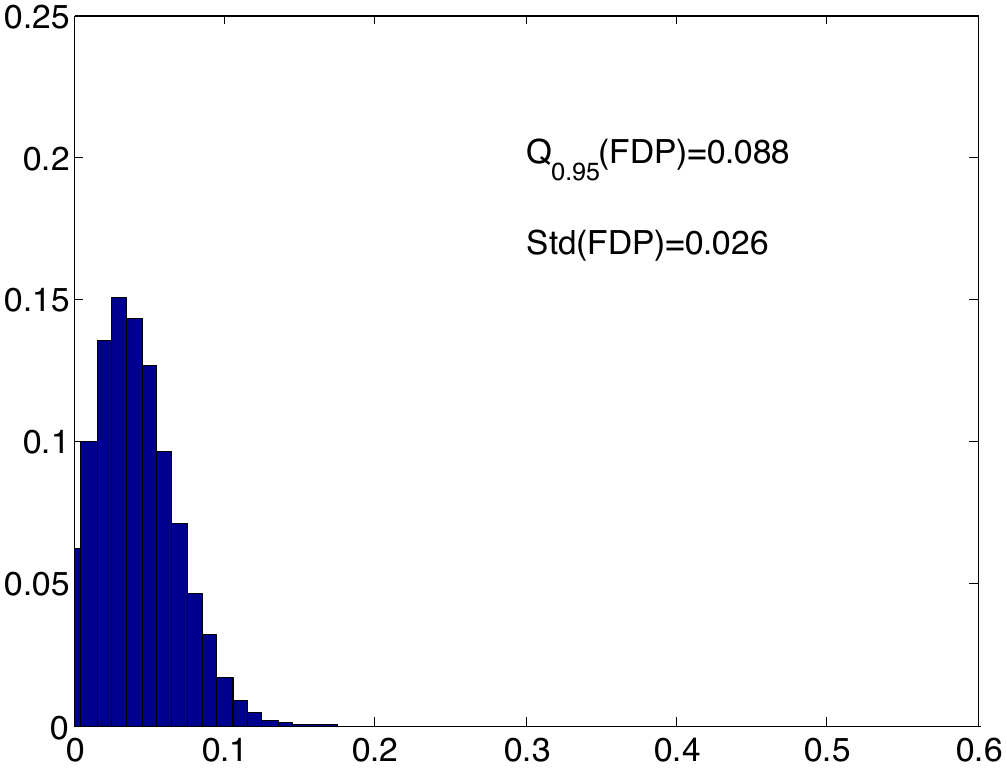}
        \put(-75,-10){{\scriptsize $\FDP$}}
        \put(-14,-15){\phantom{b}}
        \put(-150,40){\rotatebox{90}{{\scriptsize frequency}}}   
        \label{fig:LORD05}
        }
    \caption{{ Histogram of FDP for LORD rule with different values of $T$}}\label{fig:LORDT}
\end{figure} 

\subsection{Rules for controlling $\FDX_\gamma$}

Let $M(0) = \gamma - b_0-w_0 >0 $ and define, for $n\in\naturals$, $M(n) =  M(0) + \sum_{j=1}^n M_j$, where
\begin{align}\label{eq:Mj}
M_j \equiv \max\{(1+\psi_j-b_0) (\alpha_j - R_j), (b_0-\psi_j) R_j,\psi_j-b_0\}\,.
\end{align}
Note that $M(n)$ is a function of $(R_1,\dots,R_n)$,
i.e. it is measurable on $\cF_n$. We then require the following conditions in addition to 
\AssI and \AssII introduced in Section \ref{sec:DefGen}:

\noindent {\sf G3.} $w_0<\gamma - b_0$.

\noindent{\sf G4.} For $j\in\naturals$ and all $\bR_1^{j}\in\{0,1\}^{j}$, if 
\begin{align}\label{eq:FDXcondition}
M(j)+\xi_{j+1} > \frac{\gamma-b_0-w_0}{1-\alpha}\,,
\end{align}
then $\alpha_i = 0$ for all $i> j$, where we define $\xi_j \equiv \max\{(1+\psi_j-b_0)\alpha_j,|b_0-\psi_j|\}$.

Condition {\sf G4} is well posed since $M(j)$ and $\xi_{j+1}$
are functions of $\bR_1^{j}$.

Note that any generalized alpha investing rule can be modified as to
satisfy these conditions. Specifically, the rule keeps track of LHS of~\eqref{eq:FDXcondition} (it is an \emph{observable} quantity) and whenever inequality~\eqref{eq:FDXcondition} is violated, the test levels are set to zero onwards, i.e, $\alpha_i = 0$ for $i\ge j$. The sequence $(\xi_{j})_{j\in \naturals}$ is constructed in a way to be a predictable process that bounds $M_{j}$. Consequently,  $M(j)+\xi_{j+1}\in \cF_j$ bounds $M(j+1)$.

The decrement and increment values $\varphi_j$ and $\psi_j$ are determined in way to satisfy conditions ${\sf G2}$ and ${\sf G5}$.

We then establish $\FDX$ control under a certain negative dependency
condition on the test statistics. 
\begin{thm}\label{thm:FDX}
Assume that the $p$-values $(p_i)_{i\in\naturals}$ are such that, for each $j\in\naturals$,
and all $\btheta \in H_j$ (i.e. all $\btheta$ such that the null hypothesis $\theta_j=0$ holds), we have
\begin{align}
\prob_{\btheta}(p_j\le \alpha_j|\cF_{j-1})\le \alpha_j\, ,\label{eq:NegativeDependence}
\end{align}
almost surely.

Then, any generalized alpha investing rule that satisfies conditions
\AssIII, \AssIV above (together with \AssI and \AssII)
controls the false discovery exceedance:
\begin{align}
\FDX_{\gamma}\le \alpha\, .
\end{align}
\end{thm}
The proof of this theorem is presented in Appendix~\ref{sec:ProofFDX}.
Notice that the dependency condition (\ref{eq:NegativeDependence}) is satisfied, in particular, 
if the $p$-values are independent. 

\smallskip

\begin{example}
For given values of $\alpha\in (0,1)$ and $\gamma\in (\alpha,1)$, consider $\LORD$ algorithm with
$b_0 =\alpha$, $\psi_j = \alpha$ for $j\in \naturals$ and $w_0 = (\gamma-\alpha)/2$. By Equation~\eqref{eq:Mj}, we have $M_j = \alpha_j \ind(R_j = 0)$. In order to satisfy condition \AssIV, the rule keeps track of $M(n)$ and stops as soon as inequality~\eqref{eq:FDXcondition} is violated:
\begin{align}\label{eq:STLORD}
\alpha_{n+1} + \sum_{i=1}^n \alpha_i \ind(R_i = 0) >\frac{\gamma-\alpha}{2(1-\alpha)}\,.
\end{align}
Note that for $\LORD$, the potential sequence $W(n)$ always remain positive and thus the stopping criterion is defined solely based on the above inequality. Clearly, this rule satisfies assumptions \AssI, \AssII, \AssIII, \AssIV and by applying Theorem~\ref{thm:FDX} ensures $\FDX_\gamma \le \alpha$.

We use the above rule to control false discovery exceedance for the simulation setup described in Section~\ref{sec:simulationT} for values of $\alpha = 0.05$ and $\gamma = 0.15$. The results are summarized in Table~\ref{tbl:FDX}. The false discovery rates and proportions are estimated using $30,000$ realizations of test statistics. As we see the rule controls both $\FDR$ and $\FDX_\gamma$ below $\alpha$.
\begin{center}
\vspace{0.2cm}

{\small
\begin{tabular}{>{\centering\arraybackslash}p{1.3in}ccccc}
\hline
\multicolumn{6}{c}{Online control of $\FDX_\gamma$ using stopping criterion~\eqref{eq:STLORD}} \\
\hline
\multirow{2}{*}{$\pi_1$}  & \multirow{2}{*}{$0.005$}  & \multirow{2}{*}{$0.01$}  & \multirow{2}{*}{$0.02$} & \multirow{2}{*}{$0.03$} &\multirow{2}{*}{$0.04$} \\[0.5cm]
\hline
 $\FDX_\gamma$& 0.028 & 0.004 & 0.000  & 0.000 & 0.000  \\
$\FDR$ & 0.006 & 0.005 & 0.005  & 0.005 & 0.005  \\
Power & 0.666 & 0.699  & 0.679  & 0.658 & 0.639  \\
\hline
\end{tabular}
}
\captionof{table}{$\FDX_\gamma$ and $\FDR$ for $\LORD$ with stopping criterion~\eqref{eq:STLORD} using $30,000$ realizations of the test statistics. Here, $\alpha = 0.05$ and $\gamma = 0.15$, and $\pi_1$ represents the fraction of truly non-null hypotheses that appear at the beginning of the stream as described in Section~\ref{sec:simulationT}.}
\label{tbl:FDX}
\end{center}

\end{example}
%
%

\section{Discussion}
\label{sec:Discussion}

Our main result is that all generalized alpha investing rules control $\FDR$, 
provided they satisfy a natural monotonicity condition. This result can be regarded as reinforcing and complementing the conclusions of
\cite{generalized-alpha} which introduced generalized alpha investing, and proved $\mFDR$ control. Since the two metrics can be significantly 
different, with $\FDR$ somewhat more broadly accepted, this should develop more confidence towards the practical use of these methods.

Within this broad family, we believe that $\LORD$ is mainly appealing because of its simplicity: 
testing levels only depend on the the time of the most recent discovery,
and not on the whole past. This property also simplifies the analysis of $\LORD$. In particular, in
Section \ref{sec:power} we obtained bounds on the statistical power of the $\LORD$ under the mixture model, that could be used to set the parameters of the rule.
Further, a simple modification of $\LORD$ was suggested for the case of dependent $p$-values, cf. Section \ref{sec:depFDP}.

While our work broadly supports the use of generalized alpha investing rules (and, in particular, $\LORD$),
we believe that extra caution should be taken when the false discovery proportion can deviate significantly from its expectation
(which is the $\FDR$). This can be the case when the number of hypotheses is not very large, or there is significant correlation. 
In this case, the false discovery exceedance ($\FDX$) is a more meaningful metric, and additional 
constraints should be imposed on generalized alpha investing rules.

%
%

\newpage
\appendix
\section{FDR versus mFDR}\label{sec:mfdr_fdr}

The two main criteria discussed in the present paper are $\FDR(n)$ and $\mFDR_{\eta}(n)$ at level $\eta = w_0/b_0$. 
Recall that these are formally defined by
\begin{align}
\FDR(n) &\equiv \sup_{\btheta\in\Theta}\E\Big\{\frac{V^{\theta}(n)}{R(n)\vee 1}\Big\}\, ,\\
\mFDR_{\eta}(n) &\equiv \sup_{\btheta\in\Theta}\frac{\E\, V^{\theta}(n)}{\E\,R(n)+\eta}\, .
\end{align}
In addition, we introduced a new metric, that we called  $\sFDR_{\eta}(n)$ (for smoothed FDR): 
\begin{align}
\sFDR_{\eta}(n) \equiv \sup_{\btheta\in\Theta}\E\Big\{\frac{ V^{\theta}(n)}{R(n)+\eta}\Big\}\, .
\end{align}

Note that $\mFDR$ is different from the other criteria in that it does not control the probability of a property of the realized set of tests; rather it controls
the ratio of expected number of false discoveries to the expected number of discoveries.
In this appendix we want to document two points already mentioned in the main text:
\begin{enumerate}
\item $\FDR$ and $\mFDR$ can be --in general-- very different.  More precisely, we show through a numerical simulation that
controlling $\mFDR$ does not ensure controlling $\FDR$ at a similar level.
 This provides further motivation for  Theorem~\ref{thm:FDR-control}. 

We discuss this point in Section \ref{sec:Different}.
\item Theorem~\ref{thm:FDR-control} establishes  $\FDR(n)\le b_0+w_0$ and  Theorem~\ref{thm:FDR-control-2} ensures $\sFDR_{w_0/b_0}(n)\le b_0$.
Analogously, \cite{generalized-alpha} proved $\mFDR_{w_0/b_0}(n)\le b_0$. In other words, if we target $\mFDR$ or $\sFDR$ control, we 
can use larger values of $w_0$ and hence --potentially-- achieve larger power.

 We explore this point in Section \ref{sec:POWERmFDR}.
\end{enumerate}

\subsection{FDR and mFDR can be very different}
\label{sec:Different}

\begin{example}
\label{ex:SingleStep}
Since Theorem~\ref{thm:FDR-control} shows that generalized alpha investing procedures \emph{do control} FDR,
our first example will be of different type. Indeed, since we want to show that \emph{in general} FDR and mFDR are very
different, we will consider a very simple rule.

We observe $\bX= (X_1,X_2, \dots, X_n)$ where $X_j = \theta_j + \varepsilon_j$ and we want to test
null hypotheses $H_j:\theta_j = 0$. The total number of tests is $n = 3,000$ from which the first $n_0=2,700$ hypotheses are null and the remaining are non-null. For null cases,
$X_1, X_2,\dotsc, X_{n_0}$ are independent $\normal(0,1)$ observations. Under the alternative, we assume $\theta_j = 2$
and $(\varepsilon_{n_0+1}, \dotsc, \varepsilon_n)$ follows a multivariate normal distribution with covariance 
$\Sigma = \rho \bone\bone^{\sT}+(1-\rho)\id$, with $\bone$ the all-one vector. 
Here $\rho$ controls the dependency among the non-null test statistics.
In our simulation, we set $\rho = 0.9$.
It is worth noting that this setting is relevant to many applications
as it is commonly observed that the non-null cases are clustered.   

We consider a single step testing procedure, namely
\begin{eqnarray}
R_j = \begin{cases}
1 &\text { if } |X_i| \le t\,,\\
0 &\text{ if } |X_i|>t\,.
\end{cases}\label{eq:SingleStep}
\end{eqnarray}
The value of $t$ is varied from 2 to 4 and $\mFDR$ and $\FDR$ are computed by averaging over $10^4$ replications.
The result is shown in Figure~\ref{fig:mfdr_fdr}.  As we see the two
measures are very different. For instance, choosing $t=3$ controls $\mFDR$ below $\alpha=0.2$, 
but results in $\FDR\gtrsim 0.6$.
\end{example}

\begin{figure}[t]
\centering
 \includegraphics[width = 2.6in]{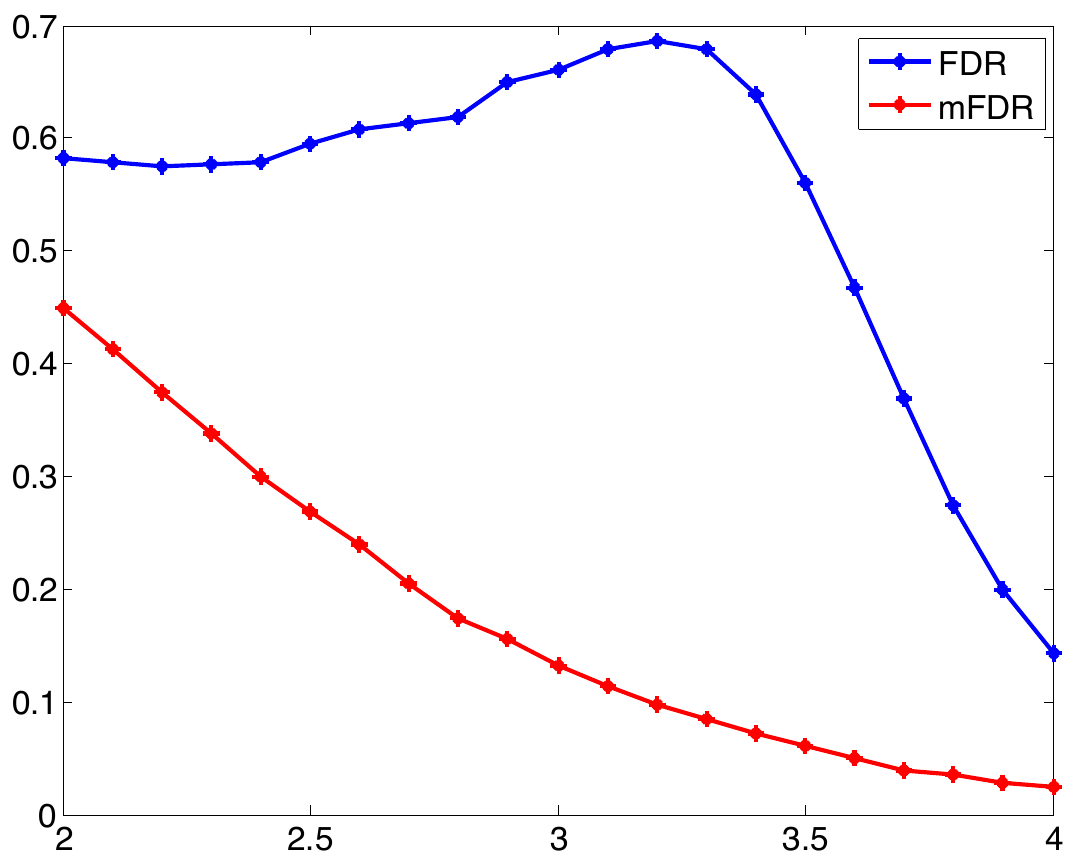}
 \put(-110,-10){threshold $t$}
  \caption{$\FDR$ and $\mFDR$ for the single step procedure of Equation~(\ref{eq:SingleStep})  under the 
setting of Example \ref{ex:SingleStep}.}
 \label{fig:mfdr_fdr}
 \end{figure}
\begin{figure}[h]
\centering
 \includegraphics[width = 2.6in]{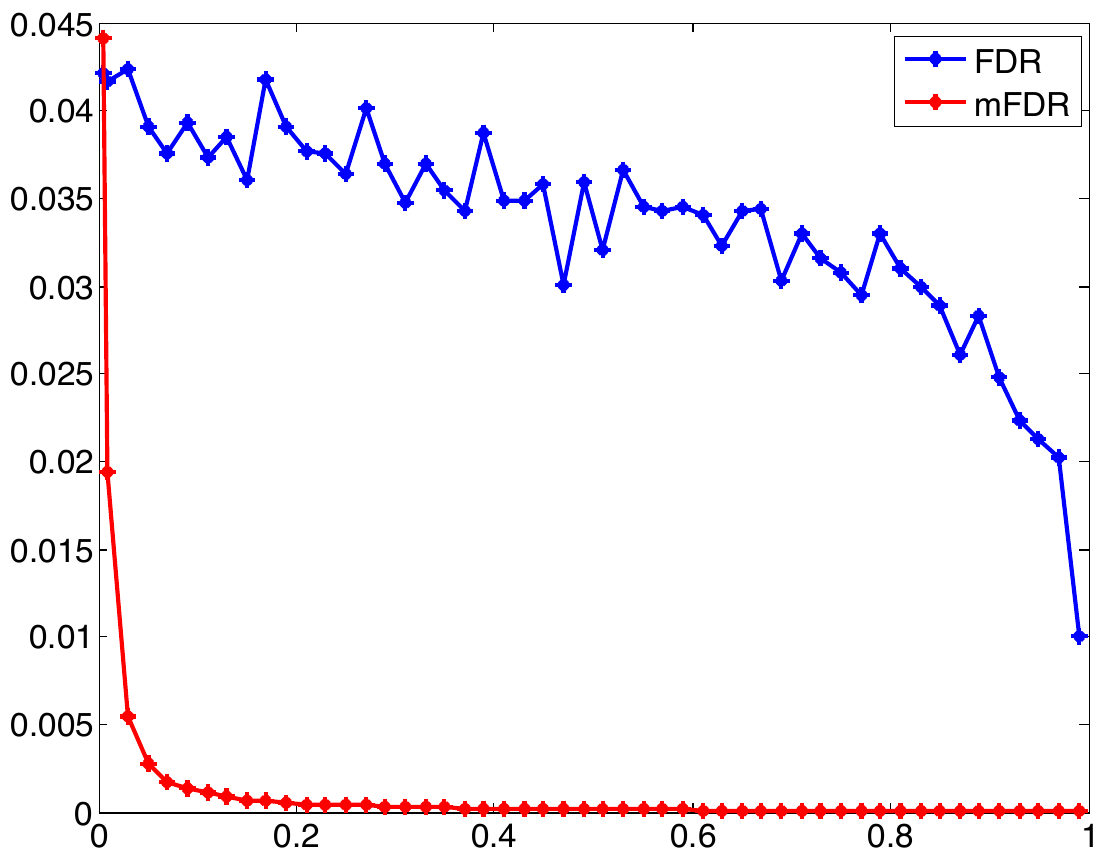}
 \put(-95,-10){$\pi_1$}
 \caption{$\FDR$ and $\mFDR$ for alpha investing rule under the setting of Example \ref{ex:SingleStep} for various fraction of non-null hypotheses $\pi_1$.}
 \label{fig:AlphaInvest_mfdr_fdr}
 \end{figure}

\begin{example}
We next consider the alpha investing rule, as described in Subsection~\ref{subsec:AI} with $\alpha_j$ set based on equation~\eqref{eq:alphaj-clust}, at nominal value $\alpha =0.05$.
 In this case Theorem~\ref{thm:FDR-control} guarantees $\FDR\le \alpha$. However $\FDR$ and $\mFDR$ can still be very 
different as demonstrated in Figure \ref{fig:AlphaInvest_mfdr_fdr}.

The hypothesis testing problem is similar to the one in the previous example. 
We consider a normal vector $\bX= (X_1,X_2,\dots,X_n)$, $X_i = \theta_i+\eps_i$, $n=3,000$, and want to test for the null 
hypotheses $\theta_i=0$. The noise covariance has the same structure as in the previous example, 
and the means are $\theta_j=4$ when the null is false.
Unlike in the previous  example, we consider a varying proportion $\pi_1$ of non-zero means. Namely,
the null is false for $i\in \{n_0+1,\dots, n\}$, with $(n-n_0) = \pi_1 n$. 

The results in Figure~\cite{alpha-investing} are obtained by averaging over $10^4$ replications. 
Alpha investing controls $\mFDR,\FDR\le 0.05$, as expected (Indeed conditions of Theorem~\ref{thm:FDR-control} hold
in this example since the $p$-values of true nulls are independent from other $p$-values).
However, the two metrics are drastically different and a bound on $\mFDR$ does not imply a bound on $\FDR$ at the same level. For instance, at $\pi_1=0.1$
we have $\mFDR\lesssim 0.001$ while $\FDR\approx 0.04$. 
\end{example}

\subsection{Comparing FDR and mFDR with respect to statistical power}
\label{sec:POWERmFDR}

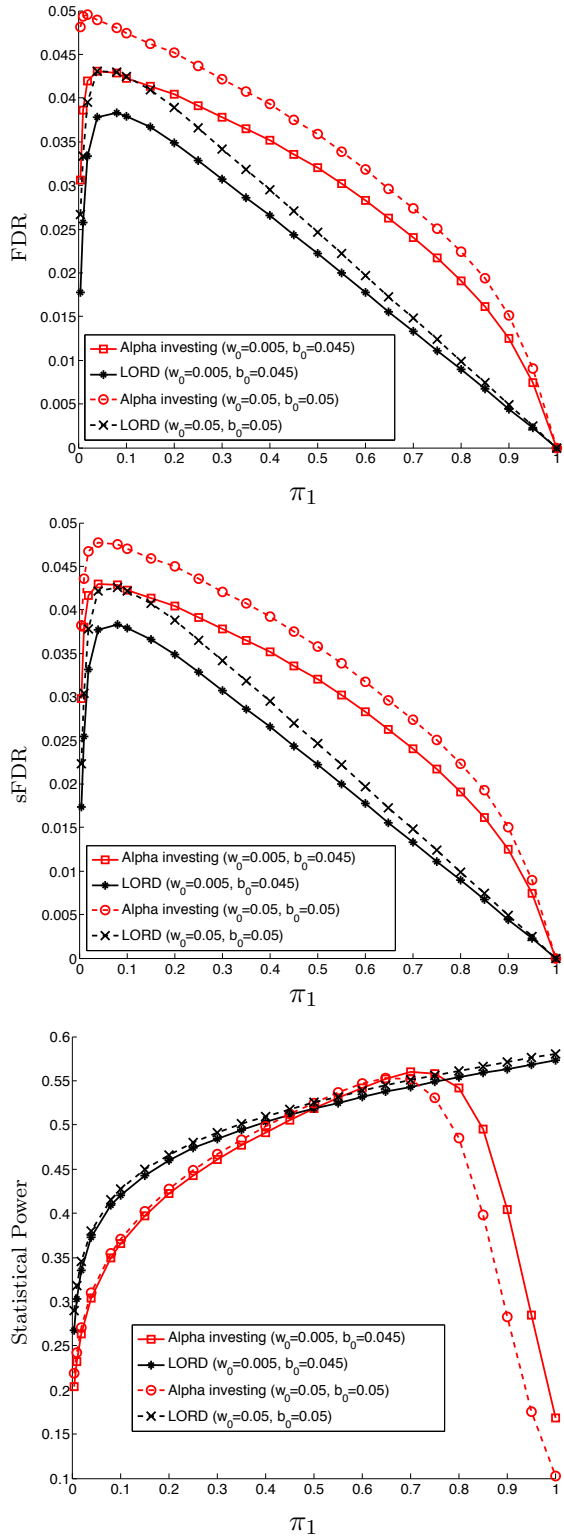
\begin{figure}
\vspace{-0.5cm}

\phantom{A}\hspace{5cm}
\begin{tabular}{c}
\includegraphics[width = 7.cm]{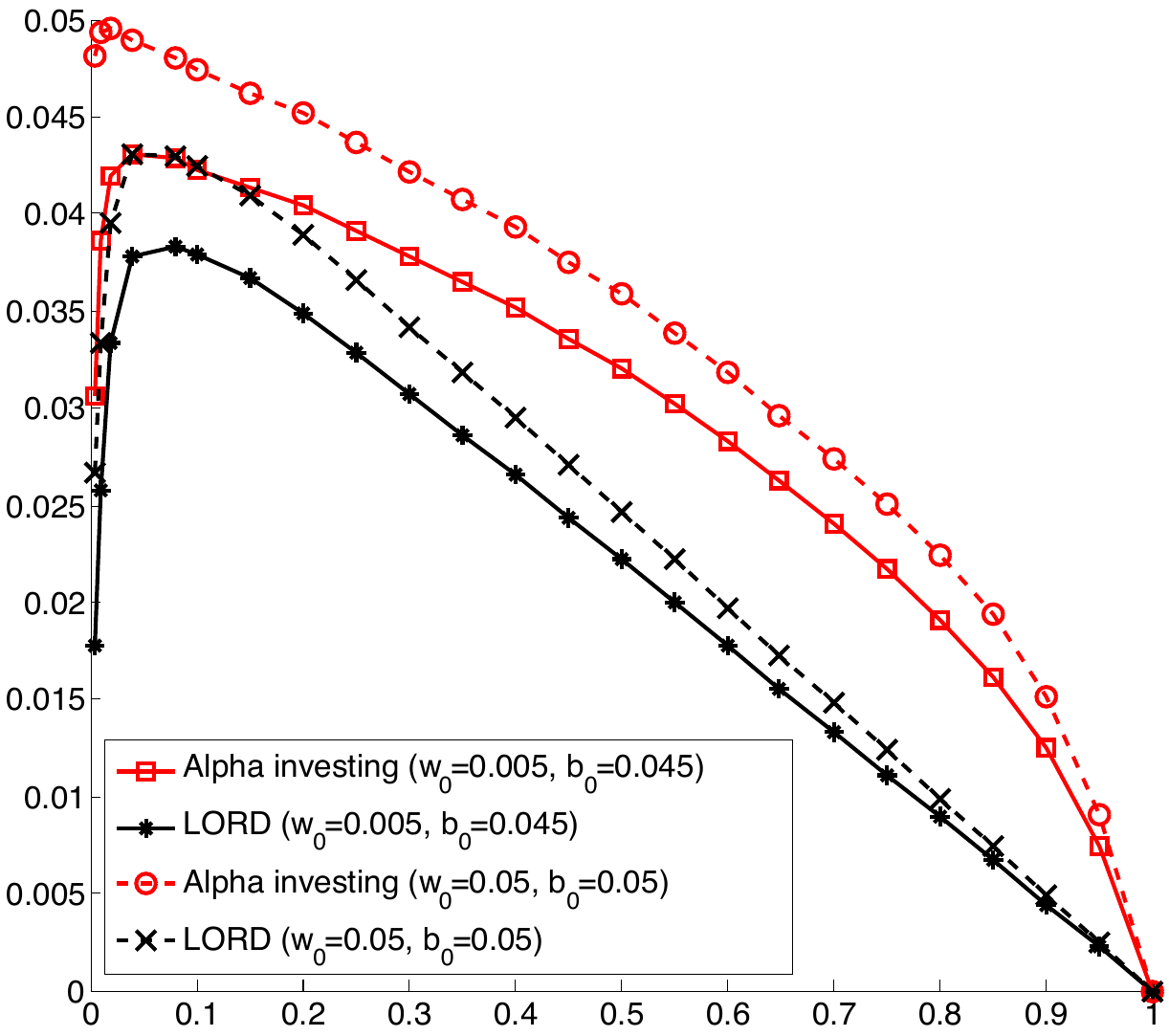}\\
\phantom{A}\\
\includegraphics[width=7.cm]{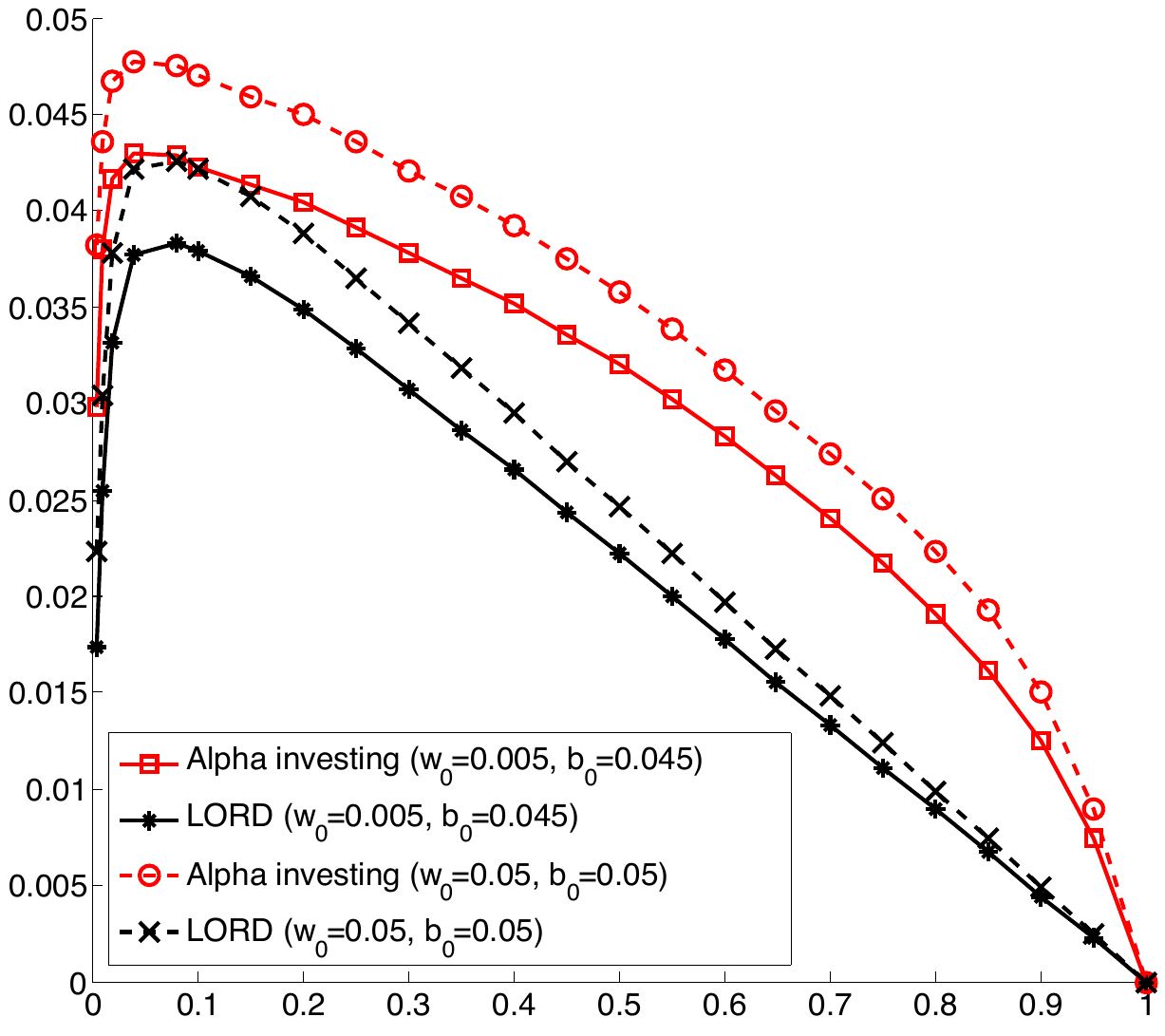}\\
\phantom{A}\\
\includegraphics[width=7.cm]{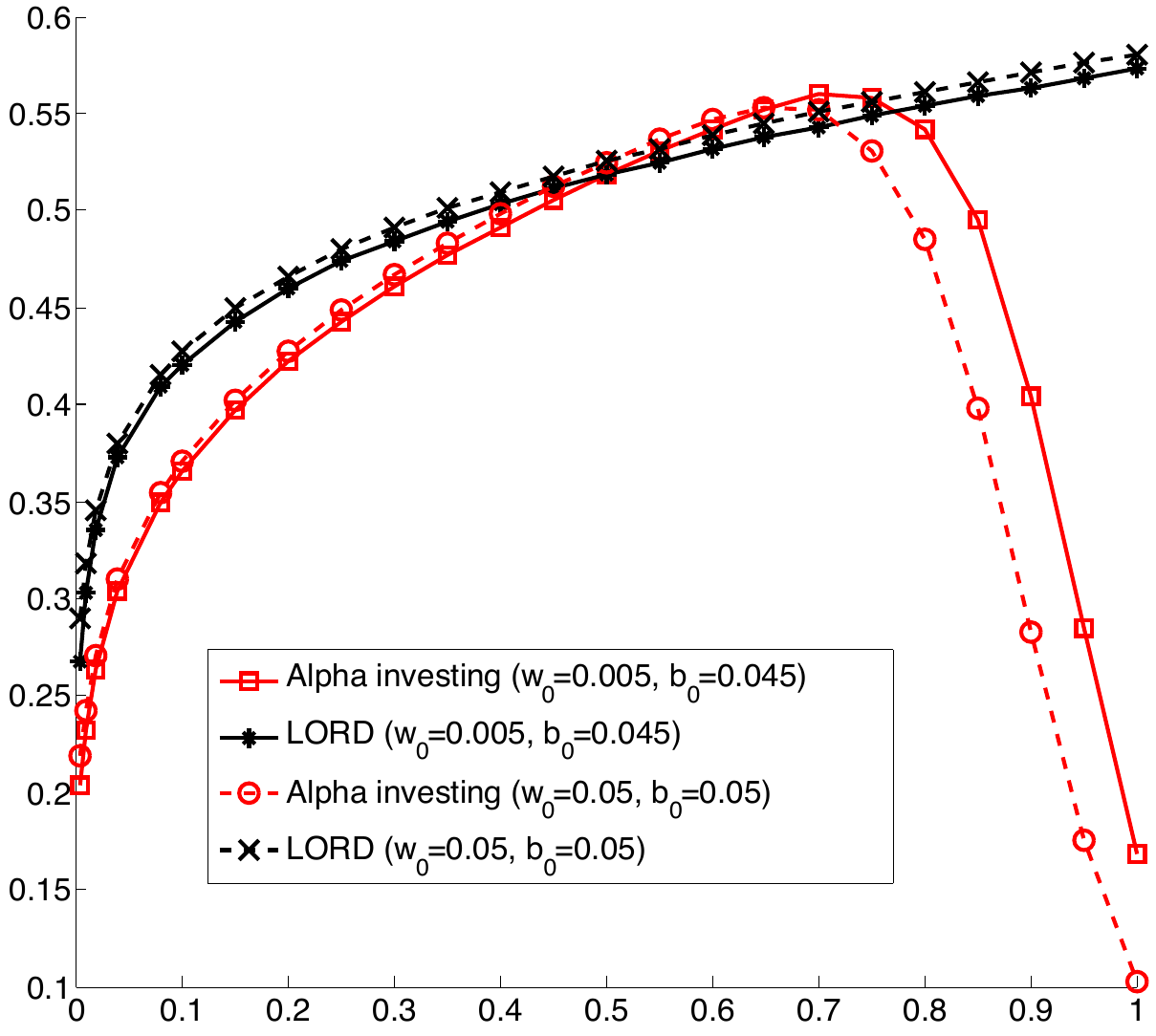}
\end{tabular}
\put(-110,100){$\pi_1$}
\put(-110,-90){$\pi_1$}
\put(-110,-290){$\pi_1$}
\put(-215,190){\rotatebox{90}{{\scriptsize $\FDR$}}}
\put(-215,-15){\rotatebox{90}{{\scriptsize $\sFDR$}}}
\put(-215,-220){\rotatebox{90}{{\scriptsize Statistical Power}}}

\caption{$\FDR$ (top), $\sFDR$ (center), and statistical power (bottom) versus fraction of non-null hypotheses $\pi_1$,
for the Gaussian setup described in Section~\ref{sec:simulation-syn}. Solid lines: parameters are tuned to control $\FDR$. 
Dashed lines: parameters are tuned to control $\mFDR$ and $\sFDR$. }\label{fig:mFDRvsFDR}
\end{figure}

In Figure \ref{fig:mFDRvsFDR} we simulated two generalized alpha investing rules, namely $\LORD$ and simple alpha investing \cite{alpha-investing},
under the same setting of Section \ref{sec:ComparisonOffline}, with Gaussian alternatives, and compare two different choices of the parameters
$w_0$ (initial wealth) and $b_0$ (bound on the reward function in Eqs.~(\ref{eq:A1a}), (\ref{eq:A1b})):
\begin{enumerate}
\item[] \emph{Solid lines.} Correspond to the choice already used in Section  \ref{sec:ComparisonOffline}, namely $w_0=0.005$, $b_0=0.045$.
By Theorem \ref{thm:FDR-control}, this is guaranteed to control $\FDR\le 0.05$.
\item[] \emph{Dashed lines.} Correspond to a more liberal choice, $w_0=0.05$, $b_0=0.05$. By \cite[Theorem 1]{generalized-alpha}, this controls
$\mFDR_1\le 0.05$. Theorem \ref{thm:FDR-control-2} provides the additional guarantee
\begin{align}
\sFDR_{1}(n) \equiv \sup_{\btheta\in\Theta}\E\Big\{\frac{ V^{\theta}(n)}{R(n)+1}\Big\}\le 0.05\, .
\end{align}
\end{enumerate}

In Figure \ref{fig:mFDRvsFDR} we compare $\FDR$, $\sFDR$ and statistical power for these two choice. As expected $\FDR$ and $\sFDR$
are slightly higher for the second choice, but still $\FDR$ appears to be below the target value $\alpha = 0.05$.
This is to be expected on the basis of Remark \ref{remark:BetterFDR} which implies $\FDR\lesssim b_0$ when the number of discoveries 
$R(n)$ is large with high probability.
We conclude that --in the case of non-nulls arriving at random, and sufficiently many strong signals-- $\mFDR$ control is conservative enough.

The last panel in the same figure shows the increase in power obtained by the second choice of parameters that targets
$\mFDR$ control.  Note that the advantage is --in this example-- somewhat marginal. In other words, $\FDR$ control can be guaranteed
without incurring large losses in power.
 %
\section{Proof of Lemma~\ref{LORD-mon}}\label{proof:LORD-mon}
We first show that all the variants of $\LORD$ rule are generalized alpha investing rule.
Equations~\eqref{eq:A1a} and~\eqref{eq:A1b} hold trivially by these rules due to Equation~\eqref{eq:LORD-general}.

For $\LORD 1$, Equation~\eqref{eq:Nonneg} holds because if $j\le t_1$, then
\begin{align*}
W(j-1)   &= w_0 - \sum_{i=1}^{j-1} \varphi_i = w_0 \Big(1-  \sum_{i=1}^{j-1} \gamma_i \Big) \ge w_0 \,\gamma_{j} = \varphi_j\, .
\end{align*}
If $j > t_1$, then
\begin{align*}
W(j-1)   &= W(\tau_j) - \sum_{i=\tau_j+1}^{j-1} \varphi_i = W(\tau_j) -\Big(\sum_{k=1}^{j-1-\tau_j}\gamma_k\Big) b_0 \\
&\ge b_0 \Big(1 - \sum_{k=1}^{j-1-\tau_j}\gamma_k\Big)\\
& \ge b_0 \gamma_{j-\tau_j} = \varphi_j\, ,
\end{align*}
where in the first inequality we used the fact that $W(\tau_j)\ge b_0$ since $j> t_1$ and the rule adds an amount of $b_0$ towards the wealth when a discovery occurs.

We next show that Equation~\eqref{eq:Nonneg} holds for $\LORD 2$. Let $k = |T(j)|$. We write
\begin{align*}
W(j-1)  &= w_0 + b_0 k - w_0\sum_{i=1}^{j-1}\gamma_i 
- b_0 \sum_{\ell \in T(j)} \sum_{i=\ell+1}^{j-1}\gamma_{i-\ell}\\
&= w_0\Big(1- \sum_{i=1}^{j-1}\gamma_i\Big) + b_0 \sum_{\ell \in T(j)} \Big(1 - \sum_{i=\ell+1}^{j-1}\gamma_{i-\ell} \Big)\\
&= w_0\Big(1- \sum_{i=1}^{j-1}\gamma_i\Big) + b_0 \sum_{\ell \in T(j)} \Big(1 - \sum_{i=1}^{j-\ell-1}\gamma_{i} \Big)\\
&\ge w_0 \gamma_j + b_0 \sum_{\ell \in T(j)} \gamma_{j-\ell} =\varphi(j)\,.
\end{align*}

For $\LORD 3$, Equation~\eqref{eq:Nonneg} stands because 
\begin{align*}
W(j-1)   &= W(\tau_j) - \sum_{i=\tau_j+1}^{j-1} \varphi_i = W(\tau_j) \,\Big\{1-\sum_{k=1}^{j-1-\tau_j}\gamma_k\Big\} \\
& \ge W(\tau_j) \,\gamma_{j-\tau_j} = \varphi_j\, .
\end{align*}
This concludes that Condition \AssI is satisfied by the three variants of $\LORD$ rule. Also, Condition \AssII follows easily for these rules because $W(i) =0$ implies
$\alpha_j=0$ and $W(j) =0$ for all $j\ge i$. 
 
 We next show that $\LORD 1$ and $\LORD 2$ are monotone rules. To this end, it suffices to show that if $\bR_1^{i-1} \preceq \btR_1^{i-1}$, then $\alpha_i\le \tilde{\alpha}_i$.
 For $\LORD 1$, note that $\bR_1^{i-1} \preceq \btR_1^{i-1}$ implies $\tau_i \le \tilde{\tau}_i$. Further, $\bgamma$ is a non-increasing sequence and
 $w_0 \le b_0$. Monotonicity then follows from the way the test levels are set, as by Equation~\eqref{alpha:LORD1}.
 
 Similarly, for $\LORD 2$, $\bR_1^{i-1} \preceq \btR_1^{i-1}$ implies $T(i)\le \tilde{T}(i)$ and monotonicity follows from Equation~\eqref{alpha:LORD2}. 
%
\section{FDR for independent $p$-values (Proof of Theorem 3.1 and Theorem 3.3)}
\label{sec:ProofFDRControl}

\begin{lemma}\label{lem:main1}
Assume the $p$-values $p_1,\dots p_n$ to be independent, and that $\theta_j=0$ (i.e. $p_j$ is a true null $p$-value). Let $R(n)=\sum_{i=1}^nR_i$ be the total number of rejection  up until time $n$ for a monotone online rule. Let $f:\integers_{\ge 0}\to \reals_{\ge 0}$ be a non-increasing, non-negative function on the integers. 
Then
\begin{align}
\E\Big\{\ind\{p_j\le \alpha_j\} \, f(R(n))\Big|\cF_{j-1}\Big\}\le
  \E\Big\{\alpha_j\,  f(R(n))\Big|\cF_{j-1}\Big\}\,.
\end{align}
\end{lemma}
\begin{proof}
We let $\bp = (p_1,p_2,\dots,p_n)$ be the sequence of $p$-values
until time $n$, and denote by  $\tbp = (p_1,p_2, \dots,p_{j-1},0,
p_{j+1}\dots, p_n)$ the vector obtained from $\bp$ by setting $p_j = 0$.
We let $\bR =(R_1,R_2,\dotsc,R_n)$  be the sequence of decisions on
input $\bp$, and denote by $\tbR=(\tR_1,\tR_2,\dots,\tR_n)$
the sequence of decision when the same rule is applied to input $\tbp$.
The total numbers of rejections are denoted by $R(n)$ and $\tR(n)$.
Finally, let $\alpha_\ell$ and $\talpha_\ell$ denote test levels given by the rule
applied to $\bp$ and $\tbp$, respectively.

Since $\alpha_{\ell} = \alpha_{\ell}(\bR_1^{\ell-1})$, $\talpha_{\ell} =
\alpha_{\ell}(\tbR_1^{\ell-1})$, we have $\alpha_\ell = \talpha_\ell$ for $\ell\le j$.
Observe that  on the event $\{p_j\le \alpha_j\}$, we have $R_j=\tR_j$
and therefore $\alpha_\ell = \talpha_\ell$ for all $1\le \ell\le n$. In words, when $H_j$ is rejected,
the actual value of $p_j$ does not matter.
Therefore, on the same event, we have $R(n) = \tR(n)$,  whence
\begin{align}
\ind\{p_j\le\alpha_j\}\, f(R(n))=
  \ind\{p_j\le\alpha_j\}\, f(\tR(n))\,.
\end{align}
Taking conditional expectations
\begin{align}
\E\Big\{\ind\{p_j\le\alpha_j\} \, f(R(n))\Big|\cF_{j-1}\Big\} 
&= \E\Big\{\ind\{p_j\le\alpha_j\} \, f(\tR(n))\Big|\cF_{j-1}\Big\}\nonumber\\
&= \E\Big\{\alpha_j\, f(\tR(n))\Big|\cF_{j-1}\Big\}\,,\label{eq:fdr1}
\end{align}
where we used the fact that, conditional on $\cF_{j-1}=\sigma(R_1,
\dotsc, R_{j-1})$, level $\alpha_j$ is deterministic (it is measurable
on $\cF_{j-1}$). Further,  $p_j$ is independent of the other $p$-values
and thus in particular is independent of the sigma-algebra generated
by $\cF_{j-1}\cup \sigma(\tR(n))$.

Note that $\tR_j =1$ and by monotonicity of the rule, $\talpha_{j+1}\ge \alpha_{j+1}$ and
hence $\tR_{j+1} \ge R_{j+1}$. Repeating this argument, we obtain $\tbR \succeq \bR$ which implies
$\tR(n) \ge R(n)$.  Hence, equation~\eqref{eq:fdr1} yields the desired result.
\end{proof}

We next prove Theorem~\ref{thm:FDR-control} and Theorem \ref{thm:FDR-control-2}.
The argument can be present in a unified way, applying Lemma \ref{lem:main1} to two different choices of the function $f(\,\cdot\,)$.
\begin{proof}[Proof (Theorem~\ref{thm:FDR-control} and Theorem~\ref{thm:FDR-control-2})]
As above, $R(j)=\sum_{i=1}^jR_i$ denotes the number of discoveries up
until time $j$, and $V(j)$ the number of false discoveries among them.
Fixing $f:\integers_{\ge 0}\to\reals_{\ge 0}$ a non-negative, non-increasing function, we define the
sequence of random variables
\begin{align}
A(j) \equiv \big\{b_0 R(j) - V(j)-  W(j)\big\} f(R(n))\,,
\end{align}
indexed by $j\in\{1,2,\dots,n\}$.
First of all, note that $A(j)$ is
integrable. Indeed $0\le R(j), V(j)\le j$, and $W(j)$ takes at most
$2^{j}$ finite values (because it is a function of
$\bR_1^j\in\{0,1\}^j$). 

 Let $A_j = A(j) -A(j-1)$, $V_j  = V(j)-V(j-1)$ and 
$W_j = W(j)-W(j-1) = -\varphi_j + R_j \psi_j$. Hence, $A_j  = \{(b_0-\psi_j)R_j - V_j +\varphi_j\}\,f(R(n))$.
Assuming $\theta_j=0$ (i.e. the $j$-th null hypothesis $H_j$ is true) we have $R_j = V_j$ and thus
$A_j =\{(b_0 - \psi_j -1)R_j + \varphi_j\}\, f(R(n))$.
Taking conditional expectation of this quantity (and recalling that
$\varphi_j$, $\psi_j$ are measurable on $\cF_{j-1}$), we get
\begin{align*}
\E(A_j|\cF_{j-1}) &= (b_0-\psi_j-1) \E\Big(R_j f(R(n))\Big|\cF_{j-1}\Big) + 
\E\Big(\varphi_j f(R(n))\Big|\cF_{j-1}\Big)\\
&\ge (b_0-\psi_j-1) \E\Big(\alpha_j\, f(R(n))\Big|\cF_{j-1}\Big) + \E\Big(\varphi_j\, f(R(n))\Big|\cF_{j-1}\Big)\\
&=  \E\Big\{\big((b_0-\psi_j-1) \alpha_j + \varphi_j\big)\, f(R(n))\Big|\cF_{j-1}\Big\}\\
&\ge \E \Big\{\big((b_0+1-b_0-\varphi_j/\alpha_j - 1) \alpha_j + \varphi_j\big) \, f(R(n))\Big|\cF_{j-1}\Big\} = 0\,.
\end{align*}
The first inequality holds because of Lemma~\ref{lem:main1} and noting that $\psi_j\ge0$ and $b_0\le 1$.
The last step follows from condition~\eqref{eq:A1b} that holds for
generalized alpha investing rules.

Assume next $\theta_j\neq  0$  (i.e. the $j$-th null hypothesis $H_j$
is true).  In this case $V_j = 0$, and therefore $A_j = \{(b_0 - \psi_j) R_j + \varphi_j\}\, f(R(n))$.
Taking conditional expectation, we get
\begin{align*}
\E(A_j|\cF_{j-1}) &= (b_0-\psi_j) \E\Big( R_j \, f(R(n))\Big|\cF_{j-1}\Big) + \E\Big(\varphi_j \, f(R(n))\Big|\cF_{j-1}\Big)\\
&\ge (b_0-\varphi_j-b_0) \E \Big(R_j \, f(R(n))\Big|\cF_{j-1}\Big) + \E\Big(\varphi_j\, f(R(n))\Big|\cF_{j-1}\Big)\\
&\ge  \E\Big\{\big(-\varphi_j R_j + \varphi_j\big) \, f(R(n))\Big|\cF_{j-1}\Big\}\ge 0\,,
\end{align*}
where the first inequality follows from condition~\eqref{eq:A1a} and in the last step we used the fact $R_j \le1$.

We therefore proved that $\E\{A_j|\cF_{j-1}\}\ge 0$ irrespectively of
$\theta_j$. Since $V(0) = R(0) = 0$, we get $A(0) = -W(0) \, f(R(n))\ge -w_0 \, f(R(n))$.  Therefore
\begin{align*}
\E\{A(n)\}& = \E\{A(0)\} +\sum_{j=1}^n\E\{A_j\}\\
& =- w_0 \E\Big\{f(R(n))\Big\}+\sum_{j=1}^n\E\big\{\E(A_j|\cF_{j-1})\big\} \ge -w_0 \E\Big\{f(R(n))\Big\}\, .
\end{align*}
Using the definition of $A(n)$, and $R(n)/(R(n)\vee 1)\le 1$,  this
implies
\begin{align*}
b_0 \E\Big\{R(n) f(R(n))\Big\}-\E\Big\{V(n) \, f(R(n))\Big\} \E\Big\{W(n)\, f(R(n))\Big\} \ge -w_0 \E\Big\{f(R(n))\Big\}
\end{align*}
Since $W(n)\ge 0$ by definition, this yields
\begin{align*}
\E\Big\{V(n) \, f(R(n))\Big\} \le b_0 \E\Big\{R(n)\, f(R(n)) \Big\} + w_0 \E\Big\{f(R(n))\Big\}\, .
\end{align*}
Substituting $f(R) = 1/(R\vee 1)$ we obtain the claim of Theorem \ref{thm:FDR-control} (as well as Remark \ref{remark:BetterFDR}).

By using instead $f(R) = 1/\{R+(w_0/b_0)\}$, we obtain Theorem \ref{thm:FDR-control-2}.
\end{proof}
%
%
\section{A lower bound on FDR}
\label{sec:FDR-lowerbound}

In this section we prove Remark \ref{remark:LowerBound}, stating that  Theorems \ref{thm:FDR-control} and \ref{thm:FDR-control-2} 
cannot be substantially improved, unless we restrict to a subclass of generalized alpha investing rules.
In particular,  Theorem \ref{thm:FDR-control-2} is optimal and Theorem \ref{thm:FDR-control} is sub-optimal at most by 
an additive term $w_0$. 
A formal statement is given below.
\begin{propo}
For any $w_0,b_0\ge 0 $, there exist a generalized alpha investing rule, with parameters $w_0,b_0$, 
and a sequence of $p$-values satisfying the assumptions of Theorem \ref{thm:FDR-control} such that 
\begin{align}
&\lim\inf_{n\to\infty}\FDR(n)\ge b_0\, ,\\
&\lim\inf_{n\to\infty}\E\Big\{\frac{V^{\theta}(n)}{R(n)+(w_0/b_0)}\Big\}\ge b_0\, . \label{eq:SecondLB}
\end{align}
\end{propo}
\begin{proof}
For the generalized alpha investing rule, we use \LORD with sequence of parameters $(\gamma_m)_{m\ge 1}$.
We assume $\gamma_m>0$ for all $m\ge 1$. 
Fix $m_0\ge 2$. We construct the $p$-values $(p_i)_{i\ge 1}$ by assuming that $H_{j}$ is false if $j\in \{m_0,2m_0,3m_0,\dots\}\equiv S$
and true otherwise. For $i \in S$, we let $p_i=0$ almost surely, and for the null hypotheses we have$(p_j)_{j\not \in S }\sim_{i.i.d.}\Unif([0,1])$.

Since $p_j=0$ we also have $R_j=1$ for all $j\in S$, and hence $W(j) \ge b_0$ for all $j\in S$. 
Consider a modified rule in which, every time a discovery is made, the potential is reset to $b_0$.
Denote by $\tW(j)$and $\tV(j)$ the corresponding potential and number of false discoveries, respectively. 
Since the rule is monotone,
we have $W(j)\ge \tW(j)$  and hence $V(j)\ge \tV(j)$, for all $j$.
Further, for all $n\ge m_0$ we have $R(n)\ge 1$ and therefore
\begin{align*}
\FDR(n) & =\E\Big\{\frac{V(n)}{R(n)}\Big\} = \E\Big\{\frac{V(n)}{\lfloor n/m_0\rfloor +V(n)}\Big\} \ge \E\Big\{\frac{\tV(n)}{\lfloor n/m_0\rfloor +\tV(n)}\Big\} \, ,
\end{align*}
where the last inequality follows since $x\mapsto x/(x+a)$ is monotone increasing for $x,a\ge 0$.
Let $X_{\ell}(m_0)$, $\ell\ge 1$ denote the number of false discoveries (in the modified rule) between $H_{\ell m_0+1}$ and $H_{(\ell+1)m_0-1}$.
Note that the $(X_{\ell}(m_0))_{\ell\ge 1}$ are mutually independent, bounded random variables and 
$\tV(n)\ge \sum_{\ell=1}^{\lfloor n/m_0\rfloor-1} X_{\ell}(m_0)$. Hence, denoting by $X(m_0)$ an independent copy of the $X_\ell(m_0)$,
we get
\begin{align}
\lim\inf_{n\to\infty}\FDR(n) & \ge \lim\inf_{n\to\infty}\E\Big\{\frac{\sum_{\ell=1}^{\lfloor n/m_0\rfloor-1} X_{\ell}(m_0)}
{\lfloor n/m_0\rfloor +\sum_{\ell=1}^{\lfloor n/m_0\rfloor-1} X_{\ell}(m_0)}\Big\} \nonumber\\
& = \frac{\E X(m_0)}{1+\E X(m_0)}\, ,\label{eq:oneplusX}
\end{align}
where the last equality follows from the strong law of large numbers and dominated convergence.

We can define $X(m_0)$ as the number of false discoveries under the modified rule between hypotheses $H_1$ and $H_{m_0-1}$
when all nulls are true, i.e. $(p_j)_{j\ge 1}\sim_{i.i.d.}\Unif([0,1])$, and we initialize by $\tW(0) = b_0$.
By this construction,  the sequence of random variables $(X(m_0))_{m_0\ge 2}$ is monotone increasing
with $\lim_{m_0\to\infty} X(m_0) = X(\infty)$, whence $\lim_{m_0\to\infty} \E X(m_0) = \E X(\infty)$ by monotone convergence. 
We next compute $\E X(\infty)$. Let $T_1$ be the time at which the first discovery is made (in particular,
$\prob(T_1=\ell )=b_0\gamma_{\ell}\prod_{i=1}^{\ell-1}(1-b_0\gamma_i)$). 
Denoting by $X(\ell,\infty)=\sum_{i=\ell+1}^{\infty}R_i$ the number
of discoveries after time $\ell$, we have
\begin{align*}
\E\{X(\infty)\} &= \sum_{\ell=1}^{\infty}\E\{X(\infty)|T_1=\ell\} \, \prob(T_1= \ell )\\
&=\sum_{\ell=1}^{\infty}\E\{X(\infty)|T_1=\ell\} \, \prob(T_1= \ell )\\
&=\sum_{\ell=1}^{\infty}\E\{1+X(\ell,\infty)|T_1=\ell\} \, \prob(T_1= \ell )\\
&=\sum_{\ell=1}^{\infty}\big\{1+\E\{X(\infty)\} \big\}\, \prob(T_1= \ell )\\
&=\big\{1+\E\{X(\infty)\} \big\}\, \prob(T_1<\infty)\, .
\end{align*}
Note that by Equation~(\ref{eq:oneplusX}) we can assume $\E\{X(\infty)\}<\infty$. Since $\prob(T_1<\infty)=b_0$,
the above implies $\E\{X(\infty)\} = b_0/(1-b_0)$.

Substituting in Equation~(\ref{eq:oneplusX}), we deduce that, for any $\eps>0$, there exists $m_{0,*}(\eps)$
such that, for the $p$-values constructed above with $m_0\ge m_{0,*}(\eps)$,
\begin{align}
\lim\inf_{n\to\infty}\FDR(n) & \ge (1-\eps)b_0\, .
\end{align}
Finally, we can take $\eps\to 0$ if we modify the above construction by taking the set
of non-null hypotheses to have, instead of equispaced elements, increasing gaps that diverge to infinity.
For instance we can take the set of non-null to be $(H_{2^{\ell}})_{\ell\ge 2}$ and repeat the above analysis.

Equation (\ref{eq:SecondLB}) follows by the same argument.
\end{proof}

%
%
\section{FDR for dependent $p$-values (Proof of Theorem 3.6)}
\label{sec:ProofDependent}

Let $\event_{v,u}$ be the event that the generalized alpha investing
rule rejects exactly $v$ true null and $u$ false null hypotheses 
in $\cH(n) = (H_1,\dots, H_n)$. We further denote by $n_0$ and $n_1=n-n_0$
the number of true null and false null hypotheses in $\cH(n)$.
The false discovery rate for a fixed choice of the parameters
$\btheta$ is
\begin{align}
\FDR^{\theta}(n) &\equiv \E(\FDP^{\theta}(n)) \nonumber\\
& = \sum_{v=0}^{n_0} \sum_{u=0}^{n_1} \frac{v}{(v+u)\vee 1}\,
  \prob(\event_{v,u})\,.\label{eq:FDRSum}
\end{align}
We next use a lemma from~\cite{benjamini2001control}. We present  its
proof here for the reader's convenience.
\begin{lemma}[\cite{benjamini2001control}]\label{lem:dummy_ind}
Let $\Omega_0\subseteq [n]$ be the subset of
true nulls. The following holds true:
\begin{align}
\prob(\event_{v,u}) = \frac{1}{v} \sum_{i\in\Omega_0} \prob((p_i\le \alpha_i) \cap \event_{v,u})\,.
\end{align}
\end{lemma} 

\begin{proof}
Fix $\btheta$ and $u,v$. In particular $|\Omega_0|=n_0$.
For a subset $\Omega \subseteq\Omega_0$ with $|\Omega| = v$, denote by $\event^\Omega_{v,u}\subseteq\event_{v,u}$
the event that the $v$ true null hypotheses in $\Omega$ are the ones rejected,
and additional $u$ false null hypotheses are rejected. 

Note that, for $i\in \Omega_0$, we have
\begin{align}
\prob\big((p_i \le \alpha_i\big) \cap \event^{\Omega}_{v,u}) = \begin{cases}
\prob(\event^\Omega_{v,u})\, & \text{ if }i\in \Omega\,,\\
0& \text{ otherwise .}
\end{cases} 
\end{align}
Therefore,
\begin{align*}
\sum_{i\in\Omega_0} \prob\big((p_i\le \alpha_i) \cap \event_{v,u}\big)
  &= 
\sum_{i\in\Omega_0} \sum_{\Omega\subseteq\Omega_0} \prob\big((p_i\le \alpha_i) \cap \event^\Omega_{v,u}\big)\\
&= \sum_{\Omega\subseteq\Omega_0} \sum_{i\in\Omega_0}\prob\big((p_i\le \alpha_i) \cap \event^\Omega_{v,u}\big)\\
&= \sum_{\Omega\subseteq\Omega_0} \sum_{i\in\Omega_0}\, \ind(i\in \Omega) \,\prob(\event^\Omega_{v,u}) = \sum_{\Omega\subseteq\Omega_0} v\, \prob(\event^\Omega_{v,u})
= v \,\prob(\event_{v,u})\,,
\end{align*}
which completes the proof.
\end{proof}
Applying Lemma~\ref{lem:dummy_ind} in Equation~(\ref{eq:FDRSum}), we obtain
\begin{align}
\FDR^{\theta}(n) = \sum_{v=0}^{n_0}
  \sum_{u=0}^{n_1}\frac{1}{(v+u)\vee1} \sum_{i\in\Omega_0}\prob((p_i\le
  \alpha_i) \cap \event_{v,u})\,. \label{eq:FDRsum1}
\end{align}
Define the measure $\nu_{i,u,v}$ on $(\reals,\cB_{\reals})$ by
letting, for any Borel set $A\in\cB_{\reals}$
\begin{align}
\nu_{i,u,v}(A) \equiv \prob\Big((p_i\le
  \alpha_i) \cap \event_{v,u}\cap \{\cI_{i-1}\in A\}\Big)
\end{align}
Notice that, by definition, $\nu_{i,v,u}$ is supported on
$[\cImin(i-1),\cImax(i-1)]$. Also $\nu_{i,v,u}$ is a finite measure, but
not a probability measure (it does not integrate to one). Indeed
$\int_{\cImin(i-1)}^{\cImax(i-1)}\de\nu_{i,v,u}(s) =\prob((p_i\le
  \alpha_i) \cap \event_{v,u}\big)$. 
Then Equation~(\ref{eq:FDRsum1}) yields
\begin{align}
\FDR^{\theta}(n) = \sum_{i\in\Omega_0}\int_{\cImin(i-1)}^{\cImax(i-1)}\sum_{v=0}^{n_0}
  \sum_{u=0}^{n_1}\frac{1}{(v+u)\vee 1} \, \de\nu_{i,v,u}(s)\,.
\end{align}
Define $\nu_{i,k}=\sum_{v,u:v+u=k}\nu_{i,v,u}$. Note that, by
definition of $\RL_i(s)$, we have $\nu_{i,k}(\{s: k\le \RL_i(s)\}) =
0$, whence $\nu_{i,k} = \ind(k> \RL_i(s))\nu_{i,k}$. Therefore:
\begin{align}
\FDR^{\theta}(n) &= \sum_{i\in\Omega_0}\int_{\cImin(i-1)}^{\cImax(i-1)}
  \sum_{k=\RL_i(s)+1}^{n}\frac{1}{k} \, \de\nu_{i,k}(s)\nonumber\\
&\le \sum_{i\in\Omega_0}\int_{\cImin(i-1)}^{\cImax(i-1)}
  \frac{1}{\RL_i(s)+1}\sum_{k=\RL_i(s)+1}^{n} \, \de\nu_{i,k}(s)\, .\label{eq:BoundFDR_Depend}
\end{align}
Letting $\nu_i = \sum_{k=1}^n \nu_{i,k}$, we have, for any Borel set
$A\in\cB_{\reals}$,
\begin{align*}
\nu_i(A) &= \prob\big(\{p_i\le \alpha_i\}\cap \{\cI_{i-1}\in A\}\big)\\
& = \prob\big(\{p_i\le g_i(\cI_{i-1})\}\cap \{\cI_{i-1}\in A\}\big) \\
& = \int_{\{\tau\le g_i(s)\} \cap\{s\in A\}} \de\hnu_i(\tau,s)\, .
\end{align*}
where $\hnu_i$ is the joint probability measure of $p_i$ and
$\cI_{i-1}$.
Since $g_i$ is non-decreasing and continuous, we will define its
inverse by $g_i^{-1}(\tau) = \inf\{s:\, g_i(s)\ge \tau\}$.
 Using this in Equation~(\ref{eq:BoundFDR_Depend}), we get the bound

\begin{align*}
&\FDR^{\theta}(n)\le
                   \sum_{i\in\Omega_0}\int\,  \frac{1}{\RL_i(s)+1} \, \ind\big(s\in
                   [\cImin(i-1),\cImax(i-1)]\big)\, \ind\big(\tau\in
                   [0,g_i(s)]\big)\; \de\hnu_i(\tau,s) \nonumber\\
& \stackrel{(a)}{\le} \sum_{i\in\Omega_0}\int\,  \frac{1}{\RL_i(s)+1} \, \ind\big(s\in
                   [\cImin(i-1)\vee g_i^{-1}(\tau),\cImax(i-1)]\big)\, \ind\big(\tau\in
                   [0,g_i(\cImax(i-1))]\big)
   \de\hnu_i(\tau,s) \nonumber\\
& \stackrel{(b)}{\le} \sum_{i\in\Omega_0}\int\,\bigg\{  \frac{1}{\RL_i(\cImin(i-1))+1} \, \ind\big(\tau\in
                   [0,g_i(\cImin(i-1))]\big) \nonumber\\
  &\quad\quad \quad \quad \quad +  \frac{1}{\RL_i(g_i^{-1}(\tau))+1} \, \ind\big(\tau\in
                   [g_i(\cImin(i-1)),g_i(\cImax(i-1))]\big) \bigg\}
                    \; \de\hnu_i(\tau,s) \nonumber\\
& \stackrel{(c)}{\le} 
\sum_{i\in\Omega_0}\bigg\{ \frac{g_i(\cImin(i-1))}{\RL_i(\cImin(i-1))+1} +
 \int\,  \frac{1}{\RL_i(g_i^{-1}(\tau))+1} \, \ind\big(\tau\in
                   [g_i(\cImin(i-1)),g_i(\cImax(i-1))]\big)
  \; \de\tau\bigg\} \,,
\end{align*}
\\
where $(a)$ follows from monotonicity of $s\mapsto g_i(s)$, $(b)$ by
the monotonicity of $s\mapsto \RL_i(s)$, and $(c)$ follows by integrating
over $s$ and noting that $\de\hnu_i(\tau)$ is the
uniform (Lebesgue) measure on the interval $[0,1]$ since $p_i$ is a
$p$-value for a true null hypothesis. Therefore by the change of variables $\tau=g_i(s)$, we obtain
\begin{align}\label{eq:FDRdep}
\FDR^{\theta}(n) &\le \sum_{i\in\Omega_0} \bigg\{ \frac{g_i(\cImin(i-1))}{\RL_i(\cImin(i-1))+1} +
\int_{\cImin(i-1)}^{\cImax(i-1)}\, \frac{\dg_i(s)}{\RL_i(s)+1} \de s\bigg\}\,.
\end{align}
Finally, let $\boldsymbol{0}_1^i$ be the zero sequence of length $i$. By definition $\cI_i(\boldsymbol{0}_1^i) \ge \cImin(i)$ and therefore, by definition of $\RL_i(s)$~\eqref{eq:RL} we have $\RL_i(\cImin(i-1)) =0$. The claim follows from equation~\eqref{eq:FDRdep}. 
%
%
%
\section{Examples of derivation of $\bgo$}\label{gammaEx}
\begin{example}
Suppose that non-null $p$-values are generated as per Beta density with parameters $a,b>0$. Then $F(x) = I_x(a,b)$ where $I_x(a,b) = (\int_{0}^x t^{a-1} (1-t)^{b-1} \de t)/B(a,b)$ is the regularized incomplete Beta function and $B(a,b) = \int_{0}^1 t^{a-1} (1-t)^{b-1} \de t$ denotes the Beta function. It is easy to see that for $a<1$ and $b\ge 1$, $F(x)$ is concave. Moreover, $\lim_{x\to 0}x^a/F(x) = a B(a,b)$. Hence, for $a<1$, we get $G(x)= \pi_1 F(x)+ (1-\pi_1) x\asymp x^a$, up to constant factor that depends on $a,b,\pi_1$. Applying Proposition~\ref{pro:optpower}, we obtain
\begin{align}
\gamma_m \asymp \Big(\frac{1}{m} \log m\Big)^{1/a}.\label{eq:GammaBeta}
\end{align}
When $a$ is small, the beta density $F'(x)$ decreases very rapidly with $x$, and thus the non-null $p$-values 
are likely to be very small. It follows from Equation~(\ref{eq:GammaBeta}) that $\gamma_m$ also decreases rapidly
in this case. This is intuitively justified: when the non-null $p$-values are typically small, small test levels 
are adequate  to reject the true non-nulls and not to waste the $\alpha$-wealth. On the other hand, when 
$a$ grows, the range of significant $p$-values becomes broader and the coefficients $\gamma_m$ decay more slowly.
\end{example}

\begin{example}{\bf (Mixture of Gaussians)} \label{ex:Mixture}
Suppose we are getting samples $Z_j\sim \normal(\theta_j,1)$ and
we want to test null hypotheses $H_j:\; \theta_j = 0$ versus alternative $\theta_j = \mu$. In this case, two-sided
$p$-values are given by $p_j = 2\Phi(-|Z_j|)$ and hence
\begin{eqnarray}
F(x) &=& \prob_1\big(|Z_j|\ge \Phi^{-1}(1-x/2)\big)\nonumber\\
&=& \Phi(-\Phi^{-1}(1-x/2)-\mu) + \Phi(\mu-\Phi^{-1}(1-x/2))\,.\label{eq:Fnorm}
\end{eqnarray}
Recall the following classical bound on the c.d.f of normal distribution for $t\ge 0$:
\begin{align}\label{eq:ineqNorm}
\frac{\phi(t)}{t}\Big(1-\frac{1}{t^2}\Big) \le \Phi(-t) \le \frac{\phi(t)}{t}\,.
\end{align}
Define $\xi(x) = \Phi^{-1}(1-x/2)$, and hence $x = 2\Phi(-\xi(x))$. 
A simple calculation shows that 
\begin{align}\label{eq:xi(x)}
\lim_{x\to 0}\xi(x)/\sqrt{2\log(1/x)} = 1\,.
\end{align} 
Applying inequalities~\eqref{eq:ineqNorm} and Equation~\eqref{eq:xi(x)}, simple calculus shows that,
as $x\to 0$,
\begin{align*}
F(x) \sim \frac{1}{2}\, x\, e^{-\mu^2/2} e^{\mu\sqrt{2\log(1/x)}} \,.
\end{align*}
Hence, for $G(x) = \pi_1 F(x)+(1-\pi_1) x$, we obtain
\begin{align*}
G(x)\sim  \frac{\pi_1}{2}\, x\, e^{-\mu^2/2} e^{\mu\sqrt{2\log(1/x)}} \,.
\end{align*}
Using Proposition~\ref{pro:optpower}, we obtain that for large enough $m$,
\begin{align}
C_1 \frac{\log m}{me^{\mu\sqrt{2\log m}}}\le \bgo_m \le C_2 \frac{\log m}{me^{\mu\sqrt{2\log m}}}\,,
\end{align}
with $C_1$, $C_2$ constants depending only on $\mu,b_0$. (In particular, we can take 
$C_1(\mu,\alpha) = 1.9\, b_0^{-1} e^{\mu^2/2}$, $C_2(\mu,\alpha) = 4.1\, b_0^{-1} e^{\mu^2/2}$.)
\end{example}
%
%

\section{FDX control (Proof of Theorem 6.1)}
\label{sec:ProofFDX}

Define $u = (\gamma-b_0-w_0)/(1-\alpha)$.
We will denote by $N$ the first time such that either $W(n)=0$ or the condition in
assumption $\AssIV$ is violated, i.e. 
\begin{align}\label{def:N}
N \equiv \min\big\{\, n\ge 1 \; \mbox{ s.t. } \; W(n)=0 
\mbox{ or }M(n) + \xi_{n+1} > u \,\big\}\, .
\end{align}
Note that this is a stopping time with respect to the filtration $\{\cF_n\}$.
Further, by assumption $\AssIV$, there is no discovery after time
$N$. Namely $R_j=0$ for all $j>N$.

Define the process
\begin{align}
B(j) \equiv \begin{cases}
b_0 R(j) - W(j)-V(j)+\gamma-b_0 &\mbox{ if $j\le N$,}\\
b_0 R(N) - W(N)-V(N)+\gamma-b_0 &\mbox{ if $j> N$.}
\end{cases}
\end{align}
Note that $B(j)$ is measurable on $\cF_j$.
A key step will be to prove the following.
\begin{lemma}\label{subM2}
The process $\{B(j)\}_{j\ge 0}$ is a submartingale with respect to the filtration $\{\cF_{j}\}$.
\end{lemma} 
\begin{proof}
As already pointed out,  $B(j)$ is measurable on $\cF_j$. Further,
since $\cF_j$ is generated by $j$ binary variables, $B(j)$ takes at
most $2^j$ finite values, and is therefore integrable.
Let $B'(j) \equiv b_0 R(j) - W(j)-V(j)+\gamma -b_0$. Since $B$ 
is a stopped version of $B'$, it is sufficient to check that $B'$
is a submartingale. 
Let $B'_j = B'(j) - B'(j-1)$, $W_j = W(j)-W(j-1)$, $V_j =
V(j)-V(j-1)$. By definition, have 
$B'_j = b_0 R_j  - W_j-V_j=(b_0-\psi_j)R_j-V_j+\varphi_j$.

We first assume that the null hypothesis $H_j$ holds, i.e.
$\theta_j=0$. Hence, 
$R_j = V_j$ and $B'_j = (b_0 -\psi_j -1) R_j +\varphi_j$.
Taking conditional expectation, we get
\begin{align*}
\E(B'_j|\cF_{j-1}) &\stackrel{(a)}{=} (b_0-\psi_j-1) \E(R_j|\cF_{j-1}) + \varphi_j\\
& \stackrel{(b)}{\ge} (b_0-\psi_j-1) \alpha_j +\varphi_j\\
& \stackrel{(c)}{\ge} (b_0 - \varphi_j/\alpha_j -b_0+1-1)\alpha_j + \varphi_j = 0\,,
\end{align*}
where $(a)$ follows because $\psi_j$ and $\varphi_j$ are measurable on
$\cF_{j-1}$; $(b)$ by assumption (\ref{eq:NegativeDependence}), since
$\psi_j\ge 0$ and $b_0\le 1$, whence $(b_0-\psi_j-1) \le 0$, and $(c)$
from assumption \AssI, cf. Equation~(\ref{eq:A1b}).

Next, we assume a false null hypothesis, i.e. $\theta_j\neq 0$, and thus $V_j =0$ and $B'_j = (b_0 - \psi_j) R_j +\varphi_j$. 
Taking again conditional expectation, we obtain
\begin{align*}
\E(B'_j|\cF_{j-1}) &= (b_0-\psi_j) \E(R_j|\cF_{j-1}) + \varphi_j\\
&\ge (b_0 - \varphi_j - b_0) \E(R_j|\cF_{j-1}) + \varphi_j \ge 0\,,
\end{align*}
from assumption \AssI, cf. Equation~(\ref{eq:A1a}).
\end{proof}

Applying Doob decomposition theorem, process $B(n)$ has a (unique) decomposition into a martingale $\tM(n)$ and a nonnegative
predictable process $A(n)$ that is almost surely increasing. Specifically, 
\begin{eqnarray*}
\tM(n) &=& B(0) + \sum_{j=1}^n \Big(B_j - \E(B_j|\cF_{j-1})\Big)\,,\\
A(n) &=& \sum_{j=1}^n \E(B_j|\cF_{j-1})\,.
\end{eqnarray*}

We next define the process $Q(n)$ as follows.
\begin{align*}
Q(n) = \begin{cases}
\tM(n) &\text{ if } \underset{0 \le i\le n}{\min} \tM(i) > 0\,,\\
0 &\text{ otherwise }\,.
\end{cases}
\end{align*}
Equivalently, define the stopping time
\begin{align*}
N_{*}\equiv \min\big\{ \, n\ge 1\; :\;  \tM(n)\le 0\, \big\}\, .
\end{align*}
(Note that $\tM(0) = B(0) = \gamma - b_0-w_0>0$ by assumption \AssIII.)
Then $Q(n)$ is the positive part of $\tM(n)$ stopped at $N_*$:
\begin{align*}
Q(n) \equiv \max\big(\,0 ,\tM(n\wedge N_*)\,\big)\, .
\end{align*}
Since it is obtained by applying a convex function to a stopped
submartingale, $Q(n)$ is also a submartingale. 

Further,
\begin{align}\label{eq:C-Bound}
0\le Q(n) \le \max(\tM(n),0) \le \max_{n\le N}\tM(n)\,,
\end{align}
where we used the fact that $\tM(n)$ rains unchanged after $N$ and $\tM(0) >0$.
Observe that, defining $\tM_j = \tM(j)-\tM(j-1)$, for $j\le N$,
\begin{align}
\tM_j = \begin{cases}
(b_0-\psi_j) (R_j - \E(R_j|\cF_{j-1}))\, \;\;& \theta_j \neq 0\,,\\
(b_0-\psi_j-1) (R_j -\alpha_j)\,& \theta_j = 0\,.
\end{cases}
\end{align}
Recalling definition of $M_j$, cf. Equation~\eqref{eq:Mj}, we have
$\tM_j\le M_j$ and $\tM(0) = M(0) = \gamma- b_0-w_0$.
Hence $\tM(n)\le M(n)$ for all $n\in \naturals$.
Furthermore, 
\begin{align*} 
\max_{n\le N} M(n)  = \max_{n\le N} \big(M(n-1)+M_n\big) \le \max_{n\le N} \big(u-\xi_n +M_n\big) \le u,
\end{align*}
where we used the fact that $M(n-1)+\xi_n\le u$ by definition of
stopping time $N$, cf. Equation~\eqref{def:N}. Using~\eqref{eq:C-Bound}, we obtain 
\begin{align}\label{eq:C-Bound2}
0\le Q(n)\le \max_{n\le N} M(n) \le u\,. 
\end{align}

We next upper bound $\prob(N_*<\infty)$ which directly yields an upper bound on $\FDX_\gamma$. Define the event $\cE_n\equiv \{Q(n) =0\}$ and set $q_n \equiv \prob(\cE_n)$ for $n\in \naturals$. Using the sub-martingale property 
of $Q(n)$ and equation~\eqref{eq:C-Bound2}, we obtain
\begin{align*}
0 < \gamma-b_0-w_0 = \E(Q(0)) \le \E(Q(n)) \le (1-q_n) u\,,
\end{align*}
whence we obtain $q_n \le \alpha$, by plugging in for $u$.
Note that $\event_n\subseteq \event_{n+1}$ for all $n\in \naturals$. Clearly $\{N_*<\infty\} = \cup_{n=0}^\infty \event_n $ and by monotone convergence properties of probability measures
\begin{align}
\prob(N_*<\infty) = \lim_{n\to \infty} \prob(\event_n) = \lim_{n\to \infty} q_n \le \alpha\,.
\end{align}
We lastly write $\FDX_\gamma$ in terms of event $\{N_*<\infty\}$ as follows.
\begin{eqnarray}
\begin{split}
\Big\{\sup_{n\ge 1} \FDP^\theta(n) \ge\gamma \Big\} &\stackrel{(a)}\equiv \Big\{\exists\, 1\le n
\le N:\, V^{\theta}(n) \ge \gamma(R(n)\vee 1) \Big\}\\
& \subseteq \Big\{\exists\, 1\le j \le N:\,  b_0 R(n) -V^\theta(n) + \gamma -b_0 \le 0 \Big\}\\
&= \Big\{\exists\, 1\le n \le N:\, B(n) \le -W(n) \Big\}\\
&\stackrel{(b)}{\subseteq} \Big\{\exists\, 1\le n \le N:\, B(n) \le 0 \Big\}\\
&\stackrel{(c)}{\subseteq}\Big\{\exists\, 1\le n \le N:\, \tM(n) \le 0 \Big\}\\
&\subseteq \big\{N_* <\infty\}\,.
\end{split}
\end{eqnarray}
Here $(a)$ holds because  there is no discovery after time $N$ and $\FDP^\theta(n)$ remains unaltered; $(b)$ holds since $W(n)\ge  0$ and $(c)$ follows from the decomposition $B(n) = \tM(n) + A(n)$ and $A(n) \ge 0$. Therefore,
\begin{align*}
\FDX_\gamma \le \prob (N_* <\infty) \le \alpha\,.
\end{align*}
This concludes the proof.

%
\section{Proof of Theorem 4.1}
\label{app:dependence}
We prove the theorem for $\LORD 1$, namely:
\begin{eqnarray}~\label{eq:MLORD}
\begin{split}
W(0) &= b_0\,,\\
\varphi_i &= \alpha_i = b_0 \gamma_{i-\tau_i} \,,\\
\psi_i &= b_0\,.
\end{split}
\end{eqnarray}
In words, we replace $W(\tau_i)$ in $\varphi_i$ with $b_0$. Given that in the other versions of $\LORD$, $W(\tau_i) \ge b_0$
at each step, the test level for rule~\eqref{eq:MLORD} is smaller than or equal to the test level of \LORD. Therefore, discoveries made by~\eqref{eq:MLORD} are a subset of discoveries made by \LORD and the statistical power of~\eqref{eq:MLORD} lower bounds the power of \LORD.

For testing rule~\eqref{eq:MLORD}, it is clear that the  times between successive discoveries are i.i.d. under the mixture model. Therefore, the process $R(n) = \sum_{\ell=1}^n R_\ell$ is a renewal process. Denote by $t_i$ the time of $i$-th discovery. We let $\Delta_i = t_i-t_{i-1}$ be the  
$i$-{th} interval between discoveries and define $r_i\equiv \ind(t_i \in \Omega_0^c)$ as the reward associated with 
inter-discovery $\Delta_i$. In other words, at each discovery we get reward one only if that discovery corresponds to a non-null hypothesis. Recall that under the mixture model, each hypothesis is truly null/non-null independently of others.    
Therefore, $(r_i, \Delta_i)$ are i.i.d across index $i$ and form a renewal-reward process.  Clearly $\E(r_i) = \pi_1$ and  
we can compute $\E(\Delta_i)$ as follows:
\begin{eqnarray*}
\prob(\Delta_i \ge m) = \prob\Big(\cap_{\ell=1}^m \{p_\ell>\alpha_\ell\}\Big) = \prod_{\ell=1}^m \Big(1-G(\alpha_\ell)\Big) \,.
\end{eqnarray*}
Substituting for $\alpha_{\ell} =b_0 \gamma_{\ell}$,
\begin{eqnarray*}
\E(\Delta_i) = \sum_{m=1}^\infty \prob(\Delta_i \ge m) \le \sum_{m=1}^\infty 
 \prod_{\ell=1}^m \Big(1-G(b_0\gamma_{\ell})\Big)\,.
\end{eqnarray*}
Without loss of generality we can assume $\E(\Delta_i) <\infty$; otherwise bound~\eqref{eq:powerLORD} becomes trivial.
Applying the strong law of large numbers for renewal-reward processes, the following holds true almost surely
\begin{eqnarray*}
\lim_{n\to \infty} \frac{1}{n} \sum_{i=1}^{R(n)} r_i = \frac{\E(r_i)}{\E(\Delta_1)} = 
{\pi_1} \Big(\sum_{m=1}^\infty \prod_{\ell=1}^m \big(1-G(b_0\gamma_{\ell})\big)\Big)^{-1}\,.
\end{eqnarray*}
Further, $\lim_{n\to \infty} |\Omega_0^c(n)|/n = \pi_1$, almost surely. Therefore, almost surely
\begin{eqnarray*}
\lim_{n\to \infty} \frac{1}{|\Omega_0^c(n)|} \sum_{i\in \Omega_0^c(n)} R_i 
\ge \lim_{n\to \infty} \frac{1}{|\Omega_0^c(n)|} \sum_{i=1}^{R(n)} r_i =  \Big(\sum_{m=1}^\infty 
 \prod_{\ell=1}^m \big(1-G(b_0 \gamma_{\ell})\big)\Big)^{-1}\,,
\end{eqnarray*}
where the first inequality follows from the fact that $\sum_{i=1}^{R(n)}r_i/|\Omega_0^c(n)|$ is the average power of rule~\eqref{eq:MLORD}, which as discussed, serves as a lower bound for the average power of \LORD.

\section{Proof of Proposition~4.2}
\label{app:optpower}

We set $\beta_m = b_0 \gamma_m$ for $m\ge 1$.
The Lagrangian for the optimization problem reads as
\begin{eqnarray*}
\mathcal{L} = \sum_{m=1}^\infty e^{-mG(\beta_m)}  + \eta (\sum_{m=1}^\infty  \beta_m - b_0)\,,
\end{eqnarray*}
where $\eta$ is a Lagrange multiplier. 
Setting the derivative with respect to 
$\beta_m$ to zero, we get
\begin{align}\label{eq:KKT}
\eta = mG'(\bopt_m) e^{-mG(\bopt_m)}\,.
\end{align}
Note that $G'(x) = \pi_1 F'(x)+(1-\pi_1) \ge 1-\pi_1$, since $F(x)$ is nondecreasing. 
Hence 
$$m(1-\pi_1) e^{-mG(\bopt_m)} \le \eta\,,$$
whereby using the non-decreasing monotone property of $G$, we obtain
\begin{eqnarray}\label{eq:betaLB}
\bopt_m \ge G^{-1}\Big(\frac{1}{m} \log\Big(\frac{m(1-\pi_1)}{\eta}\Big) \Big)\,.
\end{eqnarray} 
To obtain the upper bound, note that be concavity of $F(x)$ on $(0,x_0)$, we have $G'(x) \le G(x)/x$ for $x\in (0,x_0)$. Since $\beta_m \to 0$ as $m\to \infty$, for large enough $m$, we have
\begin{eqnarray}\label{eq:betaUB1}
m G(\bopt_m) e^{-mG(\bopt)} \ge \eta \bopt_m\,. 
\end{eqnarray}
Further, by equation~\eqref{eq:betaLB} we have $mG(\bopt_m) \ge 1$. Let $\xi_0 \ge 1$ be the solution of $\xi e^{-\xi} = \eta \bopt_m$. Simple algebraic manipulation shows that $\xi_0 \le -2\log(\eta\bopt_m)$.  Therefore, by Equation~\eqref{eq:betaUB1} we have
\begin{eqnarray*}
m G(\bopt_m) \le \xi_0 \le 2\log\Big(\frac{1}{\eta\bopt_m} \Big)\le
2\log\Big(\frac{1}{\eta G^{-1}(1/m)} \Big)\,,
\end{eqnarray*}
where the last step follows from Equation~\eqref{eq:betaLB}. 
Using the non-decreasing property of $G(x)$ we get
\begin{eqnarray}\label{eq:betaUP}
\bopt_m \le G^{-1}\Big(\frac{2}{m}\log\Big(\frac{1}{\eta G^{-1}(1/m)} \Big)\Big)\,,
\end{eqnarray}
The Lagrange multiplier $\eta$ is chosen such that $\sum_{m=1}^\infty \bopt_m = b_0$, 
equivalently, $\sum_{m=1}^\infty \gamma^{{\rm opt}}_m = 1$. 

\section{Diabetes prediction data}
\label{sec:AdClick}

In order to explore a more realistic setting, we apply online testing  to a health screening example.
The adoption of electronic health records has been accelerating in recent years both because of
regulatory and technological forces. A broadly anticipated use of  these  data is to compute predictive health scores 
\cite{bates2014big}. A high-risk score for some chronic disease can trigger an intervention,
such as an incentive for healthy behavior, additional tests, medical follow-up and so on. 
Predictive scores for long term health status are already computed within wellness programs that are supported by
many large employers in the US. 

To be concrete, we will consider a test to identify patients that are at risk of  developing diabetes
(see also \cite{bayati2015low} for related work).
Notice that a positive test outcome triggers an intervention (e.g. medical follow-up) that cannot be revised, 
and it is important to control the  fraction of alerts that are false discoveries. It is therefore natural to view this as an online
hypothesis testing problem as the ones considered in the previous sections.
For each patient $j$, we form the null hypothesis $H_{j}$: ``The patient will not develop diabetes" versus its alternative. 

We use a diabetes dataset released by Practice Fusion as part of a Kaggle competition\footnote{See 
{\sf http://www.kaggle.com/c/pf2012-diabetes}.}.  
We only use the `train' portion of this dataset, which contains de-identified medical records of $n_{\rm tot} = 9,948$ patients.
For each of patient, we have a response variable that indicates if the patient is diagnosed with Type 2 diabetes mellitus, 
along with information on
medications, lab results, immunizations, allergies and vital signs. 

We develop a predictive score based on the available records, 
and will generalized alpha investing rules to control FDR in these tests.

We split the data in three sets ${\sf Train1}$, comprising $60\%$ of the data,  ${\sf Train2}$, $20\%$ of the data, 
and $ {\sf Test}$, $20\%$ of the data. The {\sf Train} sets are used to construct a model, 
which allows to compute $p$-values. 
The $p$-values are then used in an online testing procedure applied to the {\sf Test} set. 
In detail, we proceed as follows:

\bigskip

\noindent{\bf Feature extraction.} For each patient $i$, we denote by $y_i\in\{0,1\}$ 
the response variable (with $y_i=1$ corresponding to diagnosis of diabetes) and we construct a vector of covariates
$\bx_i\in\reals^{d}$, $d=805$ by using the following attributes:
\begin{itemize}
\item \emph{Transcript records}: Year of birth, gender, and BMI
\item \emph{Diagnosis information}: We include $80$ binary features corresponding to different ICD-9 codes.
\item \emph{Medications}: We include $80$ binary features indicating the use of various medications.
\item \emph{Lab results}: For $70$ lab test observations we include a binary feature indicating whether the patient has taken the test. We further includes
abnormality flags and  the observed outcomes as features. In addition, we bin the outcomes into 10 quantiles and make 10 binary features via one-hot encoding. 
\end{itemize}

\vspace{0.1cm}

\noindent{\bf Construction of logistic model.} We use a logistic link function to model the probability of developing 
diabetes. Let  us emphasize that we do not optimize the link function. The primary goal of this section is 
to show applicability of \emph{online} multiple testing setup in many real world problems. 

Explicitly, we model the probability of no diabetes as
\begin{align}
\prob(Y_i=0|\bX_i=\bx_i) = \frac{1}{1+e^{\<\btheta,\bx_i\>}}\, .
\end{align}
The parameter vector $\btheta\in \reals^d$ is estimated using the {\sf Train1} data set. 

\bigskip

\noindent{\bf Construction of $p$-values.} Let $T_0$ be the subset of   {\sf Train2} with $Y_i=0$, and let $n_0 \equiv |T_0|$. 
For $i\in T_0$, we compute the predictive score  $q_i = 1/(1+e^{\<\btheta,\bx_i\>})$. For each $j\in {\sf Test}$,
we compute $q^{{\sf Test}}_j = 1/(1+e^{\<\btheta,\bx_j\>})$, and construct a $p$-value $p_j$ by 
\begin{align}
p_j \equiv \frac{1}{n_0}\big|\big\{\; i\in T_0:\; q_i\le q^{{\sf Test}}_j\big\}\big|\,.\label{eq:Ranking}
\end{align}
For patients with high-risk of a diabetes (according to the model), $q_j$ is small resulting in  
small $p$-values, as expected. 
We will use these $p$-values as input to our online testing procedure.

Note that, since the {\sf Train1} and {\sf Test} sets are exchangeable, the null $p$-values
will be uniform in expectation (and asymptotically uniform under mild conditions).

\bigskip

\noindent{\bf Online hypothesis testing.}   
We consider several online hypothesis testing procedures aimed at controlling FDR
below a nominal value $\alpha=0.1$ (in particular, 
without adjusting for dependency  among the $p$-values).
For $\LORD$, we choose the sequence $\bgamma = (\gamma_m)_{m\ge 1}$ 
following Equation~(\ref{eq:gamma-sim}).
For each rule, we compute corresponding false discovery proportion (FDP) and statistical power and average them over $20$ random splittings of data into ${\sf Train1}$, ${\sf Train2}$, ${\sf Test}$. 
Table~\ref{tbl:yahoo} summarizes the results. As we see,  generalized alpha investing rules
control FDR close to the desired level, and have statistical power comparable to the benchmark offline approaches.

\begin{center}
\vspace{0.5cm}

\begin{tabular}{>{\centering\arraybackslash}p{2.2in}ccc}
\hline
\multicolumn{3}{c}{Result for Diabetes data} \\
\hline
\multirow{2}{*}{}  & \multirow{2}{*}{$\FDR$}  & \multirow{2}{*}{Power}  \\[0.5cm]
\hline
 $\LORD$& 0.126 & 0.531   \\
 Alpha investing & 0.141 & 0.403 \\
Bonferroni &0.095 & 0.166\\
Benjamin-Hochberg (BH)  & 0.121 & 0.548   \\
Adaptive BH (Storey-$\alpha$) &0.123 & 0.569\\
Adaptive BH (Storey-$0.5$)&0.128& 0.573\\
\hline
\end{tabular}
\captionof{table}{False discovery rate (FDR) and statistical power for different online hypotheses testing rules for the diabetes dataset. The reported numbers are averages over $20$ random splits of data into training set and test set.}
\label{tbl:yahoo}
\end{center}



\bibliographystyle{amsalpha} 
\bibliography{all-bibliography}

\end{document}